\documentclass[11pt,reqno]{amsart}

\usepackage[leqno]{amsmath}
\usepackage{amsthm}
\usepackage{amsfonts, tikz-cd}
\usepackage{amssymb}
\usepackage{eucal}
\usepackage[all]{xy}
\usepackage{mathtools}
\usepackage{stmaryrd}
\usepackage{csquotes}
\usepackage{enumitem}
\usepackage{upgreek}
\usepackage{color}

\setenumerate{itemsep=2pt,parsep=2pt,before={\parskip=2pt}}

\usepackage[colorlinks=true,hyperindex, linkcolor=magenta, pagebackref=false, citecolor=cyan,pdfpagelabels]{hyperref}

\usepackage{relsize}
\usepackage[bbgreekl]{mathbbol}
\DeclareSymbolFontAlphabet{\mathbb}{AMSb} 
\DeclareSymbolFontAlphabet{\mathbbl}{bbold}
\newcommand{\Prism}{{\mathlarger{\mathbbl{\Delta}}}}

\usepackage{xspace}
\usepackage{epsfig,epic,eepic,latexsym,color}
\usepackage[all]{xy}
\usepackage{comment}

\usepackage[english]{babel}
\usepackage{latexsym}

\newcommand{\nc}{\newcommand}
\nc{\rnc}{\renewcommand}
\rnc{\P}{\mathbf P}
\nc{\R}{\mathbf R}
\rnc{\rm}{\mathrm}

\nc{\C}{\mathbf C}
\nc{\Q}{\mathbf Q}
\nc{\Z}{\mathbf Z}
\nc{\N}{\mathbf N}
\nc{\A}{\mathbf A}

\nc{\an}{\operatorname{an}}
\nc{\red}{\operatorname{red}}
\nc{\coker}{\operatorname{coker}}

\nc{\et}{\text{\'et}}
\nc{\fet}{\text{f\'et}}

\nc{\htt}{\operatorname{ht}}
\nc{\Nm}{\operatorname{Nm}}
\nc{\Ker}{\operatorname{Ker}}
\nc{\mmod}{\operatorname{mod}}
\nc{\End}{\operatorname{End}}
\nc{\Aut}{\operatorname{Aut}}
\nc{\cont}{\text{cont}}
\nc{\sep}{\text{sep}}
\nc{\Hom}{\mathrm{Hom}}
\nc{\Gal}{\mathrm{Gal}}
\nc{\Spec}{\text{Spec}\,}
\nc{\RZ}{\operatorname{RZ}}
\nc{\HHom}{\ud{\rm{Hom}}}
\nc{\hocolim}{\rm{hocolim}}
\nc{\diam}{\diamondsuit}
\nc{\cl}{\rm{cl}}

\rnc{\t}{\tau}
\nc{\mm}{\pmb{\mu}}
\rnc{\a}{\alpha}
\nc{\n}{\mathfrak n}
\nc{\m}{\mathfrak m}
\nc{\mfs}{\mathfrak s}
\nc{\Cat}{\cal{C}\rm{at}}
\rnc{\Pr}{\cal{P}\rm{r}^{\rm{L}}}

\nc{\p}{\mathfrak p}
\nc{\q}{\mathfrak q}

\nc{\Sym}{\operatorname{Sym}}
\nc{\codim}{\operatorname{codim}}
\nc{\rk}{\operatorname{rk}}
\nc{\GL}{\operatorname{GL}}
\nc{\SL}{\operatorname{SL}}
\nc{\Lie}{\operatorname{Lie}}
\nc{\Ind}{\operatorname{Ind}}
\nc{\Div}{\underline{Div}}
\nc{\Pic}{\mathbf{Pic}}
\nc{\uPic}{\underline{ \mathbf{Pic}}}
\nc{\rH}{\mathrm{H}}
\nc{\Spf}{\operatorname{Spf}}
\nc{\Frac}{\operatorname{Frac}}
\nc{\colim}{\operatorname{colim}}
\nc{\Spa}{\operatorname{Spa}}
\nc{\tr}{\operatorname{tr}}
\nc{\Corr}{\operatorname{Corr}}

\rnc{\an}{\operatorname{an}}

\nc{\xr}{\xrightarrow}

\nc{\eps}{\epsilon}
\nc{\ov}{\overline}
\nc{\ud}{\underline}
\nc{\wdh}{\widehat}

\nc{\I}{\mathcal I}
\nc{\F}{\mathcal F}
\nc{\G}{\mathcal G}
\nc{\E}{\mathcal E}
\nc{\M}{\mathcal M}
\nc{\cal}{\mathcal}
\rnc{\rm}{\mathrm}
\rnc{\bf}{\mathbf}

\nc{\id}{\mathrm{id}}
\nc{\alg}{\mathrm{alg}}

\nc{\X}{\mathfrak X}
\nc{\Y}{\mathfrak Y}
\nc{\T}{\mathfrak T}

\rnc{\AA}{\mathbf{A}}
\nc{\PP}{\mathbf{P}}

\nc{\LL}{\mathcal{L}}
\rnc{\S}{\mathcal S}

\nc{\ra}{\rangle}

\nc{\os}{\overset}

\rnc{\O}{\mathcal O}
\nc{\J}{\mathcal J}
\theoremstyle{definition}

\newtheorem{thm}{Theorem}[subsection]
\newtheorem{lemma}[thm]{Lemma}
\newtheorem{defn}[thm]{Definition}
\newtheorem{construction}[thm]{Construction}
\newtheorem{notation}[thm]{Notation}

\newtheorem{setup}[thm]{Setup}

\newtheorem{question}[thm]{Question}
\newtheorem{application}[thm]{Application}

\newtheorem{warning}[thm]{Warning}
\newtheorem{prop}[thm]{Proposition}

\newtheorem{rmk}[thm]{Remark}

\newtheorem{cor}[thm]{Corollary}

\newtheorem{variant}[thm]{Variant}

\begin{document}
\bibliographystyle{halpha-abbrv}
\title{Poincar\'e Duality in abstract $6$-functor formalisms}
\author{Bogdan Zavyalov}
\maketitle

\begin{abstract}
We study Poincar\'e Duality in the context of abstract $6$-functor formalisms. In particular, we give a small and simple list of assumptions that implies Poincar\'e Duality. As an application, we give new uniform (and essentially formal) proofs of some previously established Poincaré Duality results. We also discuss how our approach to Poincar\'e Duality give new (and formal) proofs of some standard results in \'etale cohomology. 
\end{abstract}

\tableofcontents
\section{Introduction}

\subsection{General overview}

A formalism of $6$-functors is an approach to formalize cohomology theories admitting a ``coefficient'' theory $D(X)$ accompanied by six operations $\left(f^*, \, f_*, \, \otimes,\, \ud{\Hom},\, f_!,\, f^!\right)$ satisfying the following list of axioms:
\begin{enumerate}
    \item every second functor is a right adjoint to the previous one;  
    \item $f^*$ is symmetric monoidal;
    \item $f_!$ commutes with arbitrary base change and satisfies the projection formula. 
\end{enumerate}

Historically, the first $6$-functor formalism was introduced by A.\,Grothendieck in \cite{SGA4} in the context of \'etale cohomology of schemes. Namely, in {\it loc.\,cit.}, it is shown that, for each integer $n$, the assignment
\[
X\mapsto D(X_\et; \Z/n\Z)
\]
can be promoted to a $6$-functor formalism. Since then, it turned out that many other cohomology theories come equipped with the corresponding $6$-functor formalisms (e.g.\,$D$-modules, mixed Hodge modules, etc.). \smallskip

Recently, there has been a new huge rise of interest in constructing new $6$-functor formalisms (see \cite{Liu-Zheng}, \cite{Scholze-diamond}, \cite{Scholze-Clausen}, \cite{complex-analytic}, \cite{Lucas-thesis}, \cite{Lucas-nuclear}, \cite{scholze-notes}). What unites all these examples (and the previous examples) is that they all satisfy a version of Poincar\'e Duality. Namely, in each of these $6$-functor formalisms, any smooth morphism $f\colon X \to Y$ admits an  invertible object $\omega_f\in D(X)$ (called the {\it dualizing object}) and an equivalence 
\[
f^!(-) \simeq f^*(-) \otimes \omega_f
\]
of functors $D(Y) \to D(X)$. Furthermore, it is often possible to give a precise formula for the dualizing object $\omega_f$. \smallskip

Despite these similarities, the proofs of Poincar\'e Duality in each case are different and require a lot of work specific to each situation. As far as we are aware, there is no uniform approach\footnote{In fact, there is a different approach to formalizing $6$-functors based on the notion of motivic triangulated categories. Unlike our situation, Poincar\'e Duality is a (non-trivial) consequence of the axioms in this setup. However, this approach seems difficult to use in practice. See Section~\ref{section:comparison-motivic} for a more detailed discussion and the relation to the results of our paper.}. For this reason, it seems important to study when an abstract $6$-functor formalism $\cal{D}$ satisfies Poincar\'e Duality. \smallskip

The main goal of this paper is to provide an approach to Poincar\'e Duality that is viable in any $6$-functor formalism. In particular, we give a uniform approach to the question of proving Poincar\'e Duality, also simplifying previously existing proofs. A somewhat surprising consequence of our approach is that it gives a method of proving Poincar\'e Duality without understanding the whole ``coefficient'' category $D(X)$. This becomes extremely useful in examples like \cite{Scholze-diamond}, \cite{complex-analytic}, or \cite{Lucas-thesis}, where the coefficient category is defined via descent, so one does not really have a good understanding of the coefficient categories on a general space $X$. \smallskip

As an explicit application of our methods, we give new proofs of Poincar\'e Duality in \'etale cohomology of schemes and adic spaces. Our proof has the advantage that it is essentially formal and does not use much beyond proper base change and the computation of cohomology of $\bf{P}^1$. We also apply our methods to the recent $6$-functor formalism $\cal{D}_\Box^{\rm{a}}(-; \O^+/p)^{\varphi}$ constructed by L.\,Mann (see \cite{Lucas-thesis}) to give a new proof of Poincar\'e Duality there without any explicit computations with the formal model of the torus. \smallskip

Lastly, we also discuss the potential applications of our methods to give a new proof of Poincar\'e Duality in the (potential) prismatic $6$-functor formalism (see Section~\ref{section:potential-application}).  

\subsection{Our results}

\subsubsection{Formulation of the questions}

Before stating the main results, we explicitly formulate the central questions we address in this paper and the primary obstacles to answering them. \smallskip

To this end, let $S$ be a base scheme (resp. a locally noetherian analytic adic space) and let $\mathcal{C}$ denote the category of locally finitely presented $S$-schemes (resp. locally finite type adic $S$-spaces). We fix a $6$-functor formalism
\[
\mathcal{D}\colon \mathrm{Corr}(\mathcal{C}) \to \mathrm{Cat}_\infty
\]
in the sense of Definition~\ref{defn:six-functors}. And we also recall that, roughly speaking, a morphism $f\colon X \to Y$ is cohomologically smooth (with respect to $\mathcal{D}$) if it (universally) satisfies the conclusion of Poincar\'e Duality up to identifying the dualizing object $\omega_f$ with the Tate twist.

Then the question of proving Poincar\'e Duality reduces to the following three (essentially independent) questions: 

\begin{question}\label{question:intro-duality} What is a minimalistic set of conditions on $\cal{D}$ that would ensure that any smooth morphism $f\colon X \to Y$ is cohomologically smooth?
\end{question}

\begin{question}\label{question:intro-duality-2} If every smooth morphism is cohomologically smooth, is there a reasonable formula for the dualizing object $\omega_f$? 
\end{question}

\begin{question}\label{question:intro-duality-3} Is there a minimalistic set of conditions on $\cal{D}$ that would ensure that $\omega_f$ is equal to the Tate twist (appropriately defined)?
\end{question}

The main goal of this paper is to give positive answers to all these questions. Our answer to Question~\ref{question:intro-duality} is optimal: it gives a characterization of all such $\cal{D}$. For Questions~\ref{question:intro-duality-2} and~\ref{question:intro-duality-3}, it seems harder to give an optimal answer; however, we give some results that cover all interesting examples of $6$-functors established up until the present moment. \smallskip

\begin{rmk} Surprisingly, our answers are uniform for schemes and adic spaces. Furthermore, the same results can be achieved in any ``geometry'' satisfying the property that, for any $f\colon X\to Y$, the diagonal morphism $X \to X\times_Y X$ is ``locally closed'' and admitting a reasonable notion of vector bundles and blow-ups. In particular, the arguments of this paper can be applied to complex-analytic spaces, formal schemes, or derived schemes almost without any change. However, it seems hard to make precise what the word ``geometry'' should mean, so we stick to the examples of schemes and adic spaces in this paper. 
\end{rmk}

\begin{rmk} With further effort, most results of this paper should extend to \emph{representable} morphisms between stacks. However, going beyond representable morphisms appears to require genuinely new tools. The arguments in Sections~\ref{section:dualizing-complex} and \ref{section:Chern-classes} rely crucially on deformation to the normal cone, a technique which currently assumes that the diagonal morphisms are locally closed immersions. The more general construction of the deformation to the normal cone in \cite{deformation2025} may provide a way to eliminate this assumption.
\end{rmk}

Before we discuss the main results of this paper, we want to point out the main problem in answering Question~\ref{question:intro-duality} and Question~\ref{question:intro-duality-2}, especially in the situation of an abstract $6$-functor formalism. \smallskip

Suppose that we have somehow guessed the correct formula for the dualizing object $\omega_f$. So the question of proving Poincar\'e Duality essentially boils down to the question of constructing an isomorphism
\[
\Hom_{D(Y)}\left(\F, f^*\G \otimes \omega_f\right) \simeq \Hom_{D(X)}\left(f_!\, \F, \G\right),
\]
functorial in $\F\in \cal{D}(X)$ and $\G\in \cal{D}(Y)$. Now the problem is that we do not have almost any control over the categories $\cal{D}(X)$ and $\cal{D}(Y)$ for a general $6$-functor formalism $\cal{D}$. This is probably not a big issue in the classical $6$-functor formalisms, but this becomes a serious issue in the recent $6$-functor formalisms (for example,  \cite{complex-analytic} or \cite{Lucas-thesis}), where the categories $\cal{D}(X)$ are defined abstractly via descent so one does not have good control over $\cal{D}(X)$ for a general $X$. \smallskip

Therefore, the main problem is to prove adjunction without really understanding the involved categories. Miraculously, it turns out to be possible, as we explain in the next section.

\subsubsection{Our Answers}

Now we are ready to discuss the answers to Questions~\ref{question:intro-duality},~\ref{question:intro-duality-2}~and~\ref{question:intro-duality-3} that we obtain in this paper. To address the first question, we separate the exact conditions needed to prove cohomological smoothness of one particular morphism $f$. We do this via the concept of a trace-cycle theory. For this, we fix a morphism $f\colon X \to Y$ with the diagonal morphism 
\[
\Delta\colon X\to X\times_Y X
\]
and the projections $p_1, p_2\colon X\times_Y X \to X$. 

\begin{defn}\label{defn:intro-trace-cycle}(Definition~\ref{defn:trace-cycle})  A {\it trace-cycle theory} on $f$ is a triple $(\omega_f, \tr_f, \cl_\Delta)$ of 
\begin{enumerate}
    \item an invertible object $\omega_f\in \cal{D}(X)$,
    \item a trace morphism
    \[
        \rm{tr}_{f}\colon f_!\, \omega_f \to \bf{1}_Y
    \]
    in the homotopy category $D(Y)$,
    \item a cycle map  
    \[
    \rm{cl}_{\Delta}\colon \Delta_! \bf{1}_{X} \longrightarrow p_2^*\,\omega_f
    \]
    in the homotopy category $D(X\times_S X)$
\end{enumerate}
such that 
\begin{equation}
\begin{tikzcd}
\bf{1}_X \arrow{r}{\sim} \arrow{d}{\rm{id}}& {p}_{1, !}\left(\Delta_! \bf{1}_X\right) \arrow{d}{{p}_{1, !}\left(\rm{cl}_{\Delta}\right)}\\
\bf{1}_X & \arrow{l}{\tr_{p_1}} p_{1, !} \left(p_{2}^*\, \omega_f\right),
\end{tikzcd}
\end{equation}
\begin{equation}
\begin{tikzcd}[column sep = 4em]
\omega_f \arrow{r}{\sim} \arrow{d}{\rm{id}}& {p}_{2, !}\left(p_{1}^* \omega_f \otimes \Delta_! \bf{1}_X\right)  \arrow{r}{p_{2, !}(\rm{id}\otimes \rm{cl}_\Delta)} & p_{2, !}(p_1^* \omega_f \otimes p_2^* \omega_f)  \arrow{d}{\wr} \\
\omega_f & \arrow{l}{\sim} \bf{1}_X\otimes \omega_f & \arrow{l}{\tr_{p_2} \otimes \rm{id}} p_{2, !}p_1^* \omega_f \otimes \omega_f,
\end{tikzcd}
\end{equation}
commute\footnote{See Construction~\ref{construction:trace-base-change} for the precise definition of $\tr_{p_i}$. Roughly, it is just the corresponding base change of $\tr_f$.} in $D(X)$ (with the right vertical arrow in the second diagram being the projection formula isomorphism).
\end{defn}

With this definition at hand, we are ready to characterize cohomologically smooth morphisms via trace-cycle theories: 

\begin{thm}\label{thm:intro-cohomologically-smooth-criterion}(Theorem~\ref{thm:cohomologically-smooth-criterion}, Remark~\ref{rmk:equivalent-trace-cycle}) Let $f\colon X \to Y$ be a morphism in $\cal{C}$. Then $f$ is cohomologically smooth if and only if $f$ admits a trace-cycle theory $(\omega_f, \tr_f, \cl_\Delta)$.
\end{thm}

The main point of Theorem~\ref{thm:intro-cohomologically-smooth-criterion} is that it allows us to ``decategorify'' the question of proving cohomological smoothness and reduce it to the question of constructing two morphisms and verifying commutativity of two diagrams. In particular, one does not need to understand the categories $\cal{D}(X)$ and $\cal{D}(Y)$ themselves, only maps between very specific objects. In particular, Theorem~\ref{thm:intro-cohomologically-smooth-criterion} is sufficiently strong to answer Question~\ref{question:intro-duality} in full generality: 

\begin{thm}\label{thm:intro-cohomologically-smooth}(Theorem~\ref{thm:cohomologically-smooth}) The relative projective line $g\colon \bf{P}^1_S \to S$ admits a trace-cycle theory $\left(\omega_{g}, \rm{tr}_g, \rm{cl}_\Delta\right)$ if and only if every smooth morphism $f\colon X \to Y$ is cohomologically smooth (with respect to $\cal{D}$).
\end{thm}

Theorem~\ref{thm:intro-cohomologically-smooth} implies that, in the presense of a trace-cycle theory on the relative projective line, the question of proving the full version of Poincar\'e Duality boils down to the question of {\it computing} the dualizing object $\omega_f=f^!\bf{1}_Y$ for any smooth morphism $f\colon X\to Y$. \smallskip

In general, this is a pretty hard question. To see that there could not be any ``trivial'' formula for the dualizing object, one could think about the case of the (solid) quasi-coherent $6$-functor formalism $\cal{D}_\Box(-; \O)$ on locally finite type (derived) $\Z$-schemes (see \cite{Scholze-Clausen}). In this situation, for a smooth morphism $f\colon X \to Y$ of pure dimension $d$, the dualizing object is given by $\Omega^d_{X/Y}[d]$. In particular, this object remembers the geometry of $f$ in a non-trivial way. \smallskip

Nevertheless, we are able to give a formula for the dualizing object for any smooth morphism $f\colon X \to Y$ in terms of the relative cotangent bundle $\rm{T}_f$ under some extra assumptions on the $6$-functor formalism $\cal{D}$. For the next construction, we assume that all smooth morphisms are cohomologically smooth with respect to $\cal{D}$. 

\begin{construction}\label{intro:cons-vector-bundles}(Variant~\ref{variant:vector-bundles})
Let $f\colon \rm{V}_X(\cal{E}) \to X$ be the total space of a vector bundle $\cal{E}$ on $X$ with the zero section $s\colon X \to \rm{V}_X(\cal{E})$. Then we define $C_X(\cal{E})\in \cal{D}(X)$ as
\[
C_X(\cal{E}) = s^*f^!\bf{1}_X \in \cal{D}(X).
\]
\end{construction}

With this definition at hand, we are ready to give the promised above formula for the dualizing complex under some additional hypothesis on $\cal{D}$:

\begin{thm}\label{thm:intro-formula-dualizing-complex}(Theorem~\ref{thm:formula-dualizing-complex} and Theorem~\ref{thm:formula-dualizing-complex-geometric}) Suppose the $6$-functor formalism $\cal{D}$ is weakly motivic or geometric (see Definition~\ref{defn:motivic-6-functors} and Definition~\ref{defn:6-functors-geometric}). Let $f\colon X \to Y$ be a smooth morphism. Then there is a canonical isomorphism
\[
f^!\bf{1}_Y \simeq C_X(\rm{T}_f) \in \cal{D}(X),
\]
where $\mathrm{T}_f$ is the relative tangent bundle of $f$.
\end{thm}

Theorem~\ref{thm:intro-formula-dualizing-complex} answers Question~\ref{question:intro-duality-2} in a pretty big generality. In practice, to use Theorem~\ref{thm:intro-formula-dualizing-complex}, one first needs to establish Theorem~\ref{thm:intro-cohomologically-smooth} to guarantee that all smooth morphisms are cohomologically smooth. And then the question of applicability of Theorem~\ref{thm:intro-formula-dualizing-complex} reduces to the question of computing cohomology of the affine line or studying the invertible objects on the projective line. \smallskip

Finally, we discuss our answer to Question~\ref{question:intro-duality-3}. The main tool in answering this question will be the notion of first Chern classes. To introduce an abstract notion of first Chern classes, we need to introduce some notation. \smallskip

\begin{notation} For the rest of this section, we fix an invertible object $\bf{1}_S\langle 1\rangle \in \cal{D}(S)$. For each $f\colon X\to S$, we define $\bf{1}_X\langle 1\rangle \coloneqq f^*\bf{1}_S\langle 1\rangle \in \cal{D}(X)$. For each integer $d\geq 0$, we define $\bf{1}_X\langle d\rangle\coloneqq \bf{1}_X\langle 1\rangle^{\otimes d} \in \cal{D}(X)$. For $d\leq 0$, we define $\bf{1}_X\langle d\rangle \coloneqq \bf{1}_X\langle -d\rangle^{\vee}\in \cal{D}(X)$. 
\end{notation}

Now we note that one can define the theory of first Chern classes (resp. weak, resp. strong first Chern classes) with respect to $\bf{1}_X\langle 1\rangle$. We do not repeat the precise definitions here; instead we refer to Definition~\ref{defn:chern-classes} and Definition~\ref{defn:geometric-chern-classes-without-length}, and only mention the main idea of the definition. Namely, a weak theory of first Chern classes is roughly just a sufficiently functorial additive way to assign first Chern classes 
\[
c_1(\cal{L})\in \rm{H}^0(X, \bf{1}_X\langle 1\rangle)
\]
for any line bundle $\cal{L}$ on a space $X$. A theory of first Chern classes is a weak theory satisfying the projective bundle formula for $\bf{P}^1_S \to S$. And a {\it strong} theory of first Chern classes is a weak theory of first Chern classes satisfying the projective bundle formula $\bf{P}^d_S \to S$ for {\it all} $d\geq 1$. \smallskip

That being said, we note Definition~\ref{defn:geometric-chern-classes-without-length} implies that, if $c_1$ is a theory of first Chern classes, then
\[
\bf{1}_S\langle -1\rangle \simeq \rm{Cone}\left(\bf{1}_S \to f_* \bf{1}_{\bf{P}^1_S}\right).
\]
So the invertible object $\bf{1}_S\langle 1\rangle$ is unique up to an isomorphism, and axiomitizes the ``Tate twist''. \smallskip

Finally, we are ready to answer Question~\ref{question:intro-duality-3} in the two theorems below:

\begin{thm}\label{thm:intro-main-thm}(Theorem~\ref{thm:main-thm}) Let $\cal{D}$ be a $6$-functor formalism satisfying the excision axiom (see Definition~\ref{defn:excision-axiom}) and admitting a theory of first Chern classes  $c_1$. Suppose that $f\colon X \to Y$ is a smooth morphism of pure relative dimension $d$. Then the right adjoint to the functor $f_!\colon \cal{D}(X) \to \cal{D}(Y)$ is given by the formula
\[
f^!(-) = f^*(-)\otimes \bf{1}_X\langle d\rangle \colon \cal{D}(Y) \to \cal{D}(X).
\]
\end{thm}

Theorem~\ref{thm:intro-main-thm} is essentially the best possible answer to Question~\ref{thm:intro-main-thm} in the presence of the excision axiom. It reduces the question of proving Poincar\'e Duality to constructing a (weak) theory of first Chern classes and computing the cohomology of the projective line. \smallskip

If $\cal{D}$ does not satisfy the excision axiom, we can also prove an analogue of Theorem~\ref{thm:intro-main-thm}. Unfortunately, it becomes less elegant, but it is still sufficiently strong to apply to the potential crystalline or prismatic $6$-functor formalisms (see Section~\ref{section:potential-application}):
 
\begin{thm}\label{thm:intro-poincare-duality-chern-classes}(Theorem~\ref{thm:poincare-duality-chern-classes}) Suppose that a $6$-functor formalism $\cal{D}$ is either weakly $\bf{A}^1$-invariant or pre-geometric (see Definition~\ref{defn:homotopy-invariant} and Definition~\ref{defn:6-functors-geometric}). And let $c_1$ be a {\it strong} theory of first Chern classes on $\cal{D}$ underlying a theory of cycle maps $\cl_\bullet$ (see Definition~\ref{defn:theory-of-cycle-classes-chern-classes}), and $f\colon X \to Y$ be a smooth morphism of pure relative dimension $d$. Then the right adjoint to the functor
\[
f_!\colon \cal{D}(X) \to \cal{D}(Y)
\]
is given by the formula
\[
f^!(-) = f^*(-)\otimes \bf{1}_X\langle d\rangle \colon \cal{D}(Y) \to \cal{D}(X).
\]
\end{thm}


\subsection{Applications}

We discuss some applications of the methods developed in this paper. We first discuss how our methods help to simplify some of the already existing proofs of Poincar\'e Duality. \smallskip

Applying these simplifications to the case of \'etale sheaves, we discuss how these simplification also lead to simplifying the proofs of other foundational results (e.g., smooth base change, cohomological smoothness, preservation of constructible (resp. lisse) sheaves under $Rf_!$ for a smooth (resp. smooth and proper) morphism $f$). \smallskip

Then we also discuss how our methods could be used to prove Poincar\'e Duality in prismatic cohomology once the formalism of $6$-functors is established in this context.

\subsubsection{Simplification of the previous proofs}

Our first application is a new, uniform, and almost formal proof of the two following classical Poincar\'e Duality results:

\begin{thm}\label{thm:intro-main-PD-ell-adic-scheme}(\cite[Exp.\,XVIII, Th.\,3.2.5]{SGA4}, Remark~\ref{rmk:scheme-duality}) Let $Y$ be a scheme and $f\colon X \to Y$ a smooth morphism of pure dimension $d$, and $n$ an integer invertible in $\O_Y$. Then the functor 
\[
\rm{R}f_!\colon \cal{D}(X_\et; \Z/n\Z) \to \cal{D}(Y_\et; \Z/n\Z)
\]
admits a right adjoint given by the formula
\[
f^*(d)[2d] \colon \cal{D}(Y_\et; \Z/n\Z) \to \cal{D}(X_\et; \Z/n\Z).
\]
\end{thm}

\begin{thm}\label{thm:intro-main-PD-ell-adic}(\cite[Th.\,7.5.3]{H3}, Theorem~\ref{thm:main-PD-ell-adic}) Let $Y$ be a locally noetherian analytic adic space, and $f\colon X\to Y$ a smooth morphism is of pure dimension $d$, and $n$  an integer invertible in $\O_Y^+$. Then the functor
\[
\rm{R}f_!\colon \cal{D}(X_\et; \Z/n\Z) \to \cal{D}(Y_\et; \Z/n\Z)
\]
admits a right adjoint given by the formula
\[
f^*(d)[2d] \colon \cal{D}(Y_\et; \Z/n\Z) \to \cal{D}(X_\et; \Z/n\Z).
\]
\end{thm}

We note that our results are slightly stronger than the classical versions appearing in \cite[Exp.\,XVII, Th.\,3.2.5]{SGA4} and \cite[Th.\,7.5.3]{H3} respectively. Namely, we do not assume that $f$ is separated and we do not make any boundedness assumptions on the derived categories $\cal{D}(X_\et; \Z/n\Z)$ and $\cal{D}(Y_\et; \Z/n\Z)$. 

\begin{rmk} In \cite{Scholze-diamond}, Scholze constructs the theory of \'etale cohomology of diamonds that generalizes the theory of \'etale cohomology of $p$-adic spaces developed in \cite{H3}. Furthermore, \cite{Scholze-diamond} is essentially independent of \cite{H3} with the following two important exceptions: quasi-compact base change and Poincar\'e Duality. We do not have anything to say about the quasi-compact base change, but we want to point out that the proof of Theorem~\ref{thm:intro-main-PD-ell-adic} can be repeated verbatim in the world of diamonds making it independent of \cite{H3}.
\end{rmk}

Before we go into the proofs of Theorem~\ref{thm:intro-main-PD-ell-adic-scheme} and Theorem~\ref{thm:intro-main-PD-ell-adic}, we would like to mention that these results formally imply many of the standard foundational results in the theory of \'etale cohomology. Since our proof does not use these results as an input, we get the following consequences of Theorems~\ref{thm:intro-main-PD-ell-adic-scheme}~and~\ref{thm:intro-main-PD-ell-adic} essentially for free (assuming that $n$ is invertible in $\O^+$):

\begin{application}
\begin{enumerate}
    \item ({\it Cohomological purity}) If $i\colon X \to Y$ is a (Zariski)-closed immersion of smooth $S$-schemes (resp. adic spaces) of pure dimension $d_X$ and $d_Y$ respectively, then 
    \[
        \rm{R}i^!\ud{\Z/n\Z} \simeq \ud{\Z/n\Z}(-c)[-2c],
    \]
    where $c=d_Y-d_Y$. This follows directly from Poincar\'e Duality and the isomorphism $\rm{R}i^! \circ \rm{R}f_Y^! \simeq \rm{R}f_X^!$, where $f_X$ and $f_Y$ are the structure morphisms.
    \item ({\it Smooth base change}) Theorem~\ref{thm:intro-main-PD-ell-adic-scheme} (resp. Theorem~\ref{thm:intro-main-PD-ell-adic}) and Proposition~\ref{lemma:cohomologically-smooth-base-change} imply the smooth base change in \'etale cohomology\footnote{We are not aware of any other proof of smooth base change simpler than the original proof in \cite[Exp.\,XVI, Cor.\,1.2]{SGA4}. The classical proof of Poincar\'e Duality uses smooth base as an input. Therefore, one cannot deduce smooth base change from the classical proof of Poincar\'e Duality and Proposition~\ref{lemma:cohomologically-smooth-base-change}.}.
    \item ({\it Preservation of constructible sheaves}) If $f\colon X \to Y$ is a smooth qcqs morphism, then $\rm{R}f_!$ restricts to the functor 
\[
\rm{R}f_!\colon \cal{D}^{(b)}_{\rm{cons}}(X_\et; \Z/n\Z) \to \cal{D}^{(b)}_{\rm{cons}}(Y_\et; \Z/n\Z).
\]
For this, we can assume that $Y$ is qcqs and $f$ is of pure dimension $d$, then \cite[Lemma 10.2]{adic-notes} implies that we only need to show that $\rm{R}f_!\colon \cal{D}^{\geq -N}(X_\et; \Z/n\Z) \to \cal{D}^{\geq -N}(Y_\et; \Z/n\Z)$ preserves compact objects for any integer $N$. This can be easily seen from the fact that the right adjoint $\rm{R}f^!=f^*(d)[2d]$ commutes with infinite direct sums and is of finite cohomological dimension. 
    \item ({\it Preservation of lisse sheaves}) If $f\colon X \to Y$ is proper and smooth, then $\rm{R}f_*$ restricts to the functor 
\[
\rm{R}f_*\colon \cal{D}^{(b)}_{\rm{lisse}}(X_\et; \Z/n\Z)\to \cal{D}^{(b)}_{\rm{lisse}}(Y_\et; \Z/n\Z).
\]
One easily reduces to the case $n=\ell$ is a prime number. By \cite[Lemma 11.1]{adic-notes}, it then suffices to show that $\rm{R}f_*$ preserves dualizable objects. Now using Poincar\'e Duality, it is formal to see that, for a dualizable object $L$, $\rm{R}f_* L$ is also dualizable with the dual $\rm{R}f_*(L^{\vee}(d)[2d])$. 
\end{enumerate}
\end{application}

Now we briefly discuss the proofs of Theorems~\ref{thm:intro-main-PD-ell-adic-scheme}~and~\ref{thm:intro-main-PD-ell-adic}. Our strategy is to use Theorem~\ref{thm:intro-main-thm} to reduce the question to constructing first Chern classes (in a sufficiently functorial manner) and verifying the projective bundle formula for the relative projective line. \smallskip

The construction of the first Chern classes comes from the Kummer short exact sequence (see Definition~\ref{defn:etale-first-chern-classes}), so the question of proving Poincar\'e Duality essentially boils down to the question of computing cohomology of the relative projective line. For this, one can reduce to the case of $S=\Spec C$ or $S=\Spa(C, \O_C)$ for an algebraically closed (non-archimedean) field $C$. Computation of cohomology groups of the projective line over an algebraically closed field is quite standard (both in the algebraic and analytic worlds). Apart from this computation, the proofs in the analytic and algebraic situations are uniform. \smallskip

Another concrete example of Poincar\'e Duality that we consider in this paper is  the version of Poincar\'e Duality for the $6$-functor formalism of ``solid almost $\O^+/p$-$\varphi$-modules'' $\cal{D}^{\rm{a}}_\Box(X; \O_{X}^+/p)^{\varphi}$ developed by L.\,Mann in \cite{Lucas-thesis}. In this context, we can give a new proof of the following result: 

\begin{thm}\label{thm:intro-main-PD-p-adic}(\cite[Th.\,3.10.20]{Lucas-thesis}, Theorem~\ref{thm:main-PD-p-adic}) Let $Y$ be a locally noetherian analytic adic space over $\Spa(\Q_p, \Z_p)$, and $f\colon X\to Y$ a smooth morphism of pure dimension $d$. Then the functor 
\[
f_!\colon \cal{D}^{\rm{a}}_\Box(X; \O_X^+/p)^{\varphi} \to \cal{D}^{\rm{a}}_\Box(Y; \O_Y^+/p)^{\varphi}
\]
admits a right adjoint given by the formula
\[
f^* \otimes \O_X^{+, \rm{a}}/p(d)[2d] \colon \cal{D}^{\rm{a}}_\Box(Y; \O_Y^+/p)^{\varphi} \to \cal{D}^{\rm{a}}_\Box(X; \O_X^+/p)^{\varphi}.
\]
\end{thm}

The proof of Theorem~\ref{thm:intro-main-PD-p-adic} follows the same strategy as the one of Theorem~\ref{thm:intro-main-PD-ell-adic}: we define first Chern classes and then compute cohomology of the relative $\bf{P}^1_Y \to Y$. \smallskip

The two main complications come from the fact that it is not, a priori, clear that this $6$-functor formalism satisfies the excision axiom, and the definition of this $6$-functor formalism is so abstract that it seems difficult to compute even cohomology of the projective line from first principles. \smallskip

However, it turns out that the verification of the excision axiom is not that hard, and we resolve the second issue via the Primitive Comparison Theorem that reduces the computation to the computation in \'etale cohomology. Besides these relatively minor points, the proof of Theorem~\ref{thm:intro-main-PD-p-adic} is essentially identical to that of Theorem~\ref{thm:intro-main-PD-ell-adic}.

\subsubsection{Potential new examples of Poincar\'e Duality}\label{section:potential-application}

Recently, V.\,Drinfeld \cite{Drinfeld22} and B.\,Bhatt--J.\,Lurie \cite{BL2} gave a new (stacky) perspective on prismatic cohomology. Namely, for a bounded prism $(A, I)$ and a bounded $p$-adic formal scheme $X$ over $A/I$, they construct its (relative derived) prizmatization stack $\rm{WCart}_{X/A}$. For an lci $X$, this comes equipped with an isomorphism 
\[
\cal{D}_{qc}(\rm{WCart}_{X/A})\simeq \widehat{\cal{D}}_{\rm{crys}}((X/A)_\Prism, \O_\Prism)
\]
between the $\infty$-categories of quasi-coherent sheaves on $\rm{WCart}_{X/A}$ and prismatic $\O_\Prism$-crystals on $X$. Therefore, it is reasonable to expect that $\cal{D}_{qc}(\rm{WCart}_{X/A})$ provide a reasonable coefficient theory for (relative) prismatic cohomology. Unfortunately, this assignment can not be promoted to a $6$-functor formalism because this is already impossible for $\cal{D}_{qc}(-)$ (even on schemes); the problem being that the open immersion pullback $j^*$ {\it does not} admit a left adjoint. \smallskip

In the case of (derived) schemes, P.\,Scholze and D.\,Clausen \cite{Scholze-Clausen} were able to enlarge the category $\cal{D}_{qc}(-)$ to the category of all {\it solid modules} $\cal{D}_{\Box}(-)$ to get a $6$-functor formalism on (derived) schemes. Therefore, it is reasonable to expect that appropriately defined $\infty$-category $\cal{D}_\Box(\rm{WCart}_{X/A})$ of solid sheaves on the stack $\rm{WCart}_{X/A}$ should give the correct coefficient theory for the prismatic cohomology and admit a $6$-functor formalism. \smallskip

Furthermore, L.\,Tang has recently proven Poincar\'e Duality for prismatic cohomology of smooth and proper $p$-adic formal $A/I$-schemes (see \cite[Th.\,1.2]{Tang}). This makes it reasonable to expect that this potential $6$-functor formalism should satisfy the full version of Poincar\'e Duality with all solid coefficients. Once this $6$-functor formalism $\cal{D}$ is constructed, Theorem~\ref{thm:intro-poincare-duality-chern-classes} reduces Poincar\'e Duality to the question of constructing (strong) first Chern classes, cycle class maps for divisors, and showing that $\cal{D}$ is pre-geometric (see Definition~\ref{defn:6-functors-geometric}). We expect that, under the correct formalization of $\cal{D}_{\Box}(\rm{WCart}_{X/A})$, all these questions should follow from the already existing results:

\begin{enumerate}
    \item (First Chern classes) A strong theory of prismatic first Chern classes has already been constructed in \cite[Notation 7.5.3 and Variant 9.1.6]{BL1};
    \item (Cycle maps for divisors) we expect that a theory of cycle maps should follow from \cite[Construction 5.32]{Tang};
    \item ($\cal{D}$ is pre-geometric) It suffices to show that every invertible object on $\bf{P}^1_Y$ comes from $Y$. We expect that, there should be an equivalence of the $\infty$-categories of invertible objects
    \[
    \cal{P}ic\left( \cal{D}_{\Box}(\rm{WCart}_{Y/A}) \right) \simeq \cal{P}ic\left( \cal{D}_{qc}(\rm{WCart}_{Y/A}) \right).
    \]
    This would reduce the question to showing that any prismatic line bundle on $\bf{P}^1_{B/J}$ comes from a line bundle $\Spec B/J$ for any morphism of bounded prisms $(A, I) \to (B, J)$. This can be explicitly seen by showing that the pullback along the natural morphism
    \[
    \bf{P}^1_B \to \rm{WCart}_{\bf{P}^1_{B/J}/B}
    \]
    is fully faithful on line bundles and first Chern class considerations to trivialize the pullback. We do not spell out the precise argument as it is beyond the scope of this paper. 
\end{enumerate}

We expect that similar considerations should apply to the absolute prismatizations $X^{\Prism}$, $X^{\cal{N}}$, and $X^{\rm{syn}}$ introduced in \cite{Drinfeld22} and \cite{bhatt-teaching}.

\subsection{Comparison with motivic categories}\label{section:comparison-motivic}

There is a different approach to $6$-functor formalisms via the theory of motivic triangulated categories (at least in case of schemes); see \cite{Ayoub-1}, \cite{Ayoub-2}, and \cite[Def.\,2.4.45]{Cisinski-Deglise}. Unlike our formalization of a $6$-functor formalism, Poincar\'e Duality turns out to be a (non-trivial) consequence of the axioms of a motivic triangulated category and, therefore, holds in any such category (see \cite[\textsection 1.5.3]{Ayoub-1} or \cite[Th.\,2.4.50]{Cisinski-Deglise}). However, the drawback of this approach is that, in order to construct motivic triangulated categories, one usually needs to establish results closely related to Poincar\'e duality. In particular, this approach does not seem particularly useful for the purpose of proving Poincar\'e Duality results in concrete examples. However, some ideas similar to the ones used in \cite{Ayoub-1}, \cite{Ayoub-2}, \cite[Def.\,2.4.45]{Cisinski-Deglise}, and \cite{Deglise-Gysin-2} also play a role in this paper. For this reason, we summarize the main differences and similarities. \smallskip

Recall that our notion of a $6$-functor formalism formalizes ``cohomology theories'' with coefficient theories that satisfy proper base-change and projection formula. Unlike this, motivic triangulated categories formalize different cohomological properties. From our point of view, one of the main axioms of a motivic triangulated category is the assumption that, for every smooth morphism $f\colon X \to Y$, the functor $f^*$ admits a {\it left} adjoint $f_\sharp$ that satisfies base change and projection formula. This is a very strong assumption since it formally implies both the usual smooth base change theorem and that the functor $f^*$ commutes with all limits (and colimits). \smallskip

Unlike proper base-change and projection formula, both of the above properties seem very difficult to verify in practice. For example, in either \cite{H3}, \cite{Lucas-thesis}, or \cite{Scholze-diamond}, smooth base change is obtained as a {\it consequence of Poincar\'e Duality}. In fact, it seems difficult to prove that $f^*$ admits a left adjoint with the desired properties without first verifying that $f$ is cohomologically smooth in the sense of Definition~\ref{defn:cohomologically-smooth}. Since proving that a smooth morphism is cohomologically smooth is arguably the hardest part of Poincar\'e Duality, it seems hard to use motivic triangulated categories for the purpose of proving Poincar\'e Duality in concrete examples of $6$-functor formalisms. \smallskip 

The axioms of a motivic triangulated category also include the excision axiom and (a strong) version of $\bf{A}^1$-invariance. Even though we also sometimes need to assume either weak $\bf{A}^1$-invariance or the excision axiom, we do not put these assumptions into the foundations of the theory and can achieve results without these assumptions (see Theorem~\ref{thm:intro-poincare-duality-chern-classes}). This is crucial for the potential application for the prismatic (or crystalline) $6$-functor formalism sketched in Section~\ref{section:potential-application} since this $6$-functor formalism can satisfy neither the excision nor the $\bf{A}^1$-invariance axiom. \smallskip

Even though the formalisms are quite different, we also note that some of the {\it ideas} used in this paper have been also extensively used in the motivic literature. In particular, the idea of using deformation to the normal cone to compute the dualizing object goes back to Voevodsky and has also been used in \cite{Ayoub-1} and \cite{Cisinski-Deglise}. Also, the idea of using first Chern classes to ``trivialize'' the dualizing complex has previously appeared in \cite{Deglise-Gysin-2}. In this regard, the key new idea of this paper is to use the $2$-category of cohomological correspondences to find a minimalistic set of conditions that ensures cohomological smoothness of a morphism.

\subsection{Strategy of the proof}

Now we discuss the strategy of our proof of Theorem~\ref{thm:intro-main-thm}: 

\begin{enumerate}\itemsep0.5em

\item(Section~\ref{section:trace-cycle}) We start by proving Theorem~\ref{thm:intro-cohomologically-smooth-criterion}. The main step in the proof is to ``decategorify'' the question. The key idea is to use the $2$-category of cohomological correspondences originally introduced in \cite{Lu-Zheng} and reviewed in Section~\ref{section:cohomological-correspondences}. After we establish Theorem~\ref{thm:intro-cohomologically-smooth-criterion}, we show that it implies Theorem~\ref{thm:intro-cohomologically-smooth}. In particular, this guarantees that any smooth morphism is cohomologically smooth if $\bf{P}^1_S \to S$ admits a trace-cycle theory. \smallskip

\item(Section~\ref{section:dualizing-complex}) The next goal is to deduce a formula for the dualizing object $f^!\bf{1}_Y$ for a smooth morphism $f\colon X \to Y$. This is done via a version of Verdier's diagonal trick and deformation to the normal cone, we show\footnote{At least under the assumptions of Theorem~\ref{thm:intro-formula-dualizing-complex} that we will prove in later steps.} that \[
f^!\bf{1}_Y \simeq C_X(\rm{T}_f) \in \cal{D}(X),
\]
where $C_X(\rm{T}_f)$ is defined in Construction~\ref{intro:cons-vector-bundles}.  \smallskip

Now the question of proving Poincar\'e Duality boils down to the question of constructing a trace-cycle theory of $\bf{P}^1_S \to S$ and then computing $C_X(\rm{T}_f)$ for every smooth morphism $f\colon X \to Y$. \smallskip

\item(Sections~\ref{section:first-chern-classes}-\ref{section:first-chern-classes-excision}) We introduce the notion of a theory of first Chern classes. Then we show that, in the presence of the excision axiom, existence of a theory of first Chern classes automatically implies weak $\bf{A}^1$-invariance of $\cal{D}$, existence of cycle maps for divisors, and the projective bundle formula. \smallskip

\item(Section~\ref{section:trace-map}) Then we construct the trace morphism for projective bundles in the presence of a theory of first Chern classes. Then we show that, for the projective line $f\colon \bf{P}^1_S \to S$, the triple $(\bf{1}_{\bf{P}^1_S}\langle 1\rangle, \tr_f, \cl_\Delta)$ forms a trace-cycle theory; this is essentially just a formal diagram chase. \smallskip

\item(Section~\ref{section:duality}) Finally, the question of proving Theorem~\ref{thm:intro-cohomologically-smooth-criterion} boils down to the question of computing
\[
f^!\bf{1}_Y \simeq C_X(\rm{T}_f) 
\]
for every smooth morphism $f\colon X \to Y$ of relative pure dimension $d$. For this, we compactify the morphism $g\colon \rm{V}_X(\rm{T}_f) \to X$ to the morphism $\ov{g}\colon \bf{P}_X(\rm{T}_f^{\vee} \oplus \O) \to X$ with the ``zero'' section $s\colon X \to \bf{P}_X(\rm{T}_f^{\vee} \oplus \O)$. Then the question reduces to constructing an isomorphism 
\[
s^*\ov{g}^!\bf{1}_X \simeq \bf{1}_X\langle d\rangle.
\]
Roughly, the morphism comes from the trace map constructed in the previous step. In order to show that this is an isomorphism, we can work locally on $X$. Thus we can assume that $\rm{T}_f$ is a trivial vector bundle, so $\bf{P}_X(\rm{T}_f^{\vee} \oplus \O) \simeq \bf{P}_X^d$. Then the cycle map of a point gives an inverse to this map. 
\end{enumerate}

\subsection{Terminology}

We say that an analytic adic space $X$ is {\it locally noetherian} if there is an open covering by affinoids $X= \bigcup_{i\in I} \Spa(A_i, A_i^+)$ with strongly noetherian Tate $A_i$. Sometimes, such spaces are called locally {\it strongly} noetherian.\smallskip

We follow \cite[Def.\,1.3.3]{H3} for the definition of a locally finite type, locally weakly finite type, and locally $+$-weakly finite type morphisms of locally noetherian adic spaces. \smallskip 

For a Grothendieck abelian category $\cal{A}$, we denote by $D(A)$ its \emph{triangulated derived category} and by $\cal{D}(\cal{A})$ its $\infty$-enhancement. \smallskip

For a symmetric monoidal $\infty$-category $\cal{C}^\otimes$, we denote by $\cal{P}ic(\cal{C}^\otimes)$ the full $\infty$-subcategory of $\cal{C}$ consisting of invertible objects. We also denote by $\rm{Pic}(\cal{C}^\otimes)$ the group of isomorphism classes of invertible objects in $\cal{C}^\otimes$. \smallskip

\subsection{Acknowledgements} We heartfully thank Ofer Gabber and Peter Scholze for their questions after author's presentation of his previous work on $p$-adic Poincar\'e Duality \cite{zav-duality} at the RAMpAGe seminar and the Oberwolfach workshop respectively; this was the starting point of this paper. We are also very grateful to Peter Scholze for numerous illuminating conversations, which have greatly influenced the development of this paper. The paper owes a huge intellectual debt to these conversations. 

We thank Marc Hoyois for suggesting the argument of Proposition~\ref{prop:8-2-category-of-cohomological-correspondences}, Ko Aoki and Peter Haine for explaining some necessary $\infty$-categorical background to the author, and Adeel Khan for patiently answering author's questions on his paper \cite{Khan-duality}.

We also thank Toni Annala, Bhargav Bhatt, Dustin Clausen, Dmitry Kaledin, Dmitry Kubrak, Shizhang Li, Lucas Mann, and Emanuel Reinecke for many interesting conversations. We thank Fr\'{e}d\'{e}ric D\'eglise, Nikon Kurnosov, Slava Naprienko, Alexander Petrov, Alberto Vezzani, and Mingjia Zhang for their comments on the previous draft of this paper. We are also very grateful to Fr\'ed\'eric D\'eglise for drawing our attention to \cite{Ayoub-1}, \cite{Cisinski-Deglise}, and \cite{Deglise-Gysin-2}. 

We thank the Max Planck Institute and the Institute for Advanced Study for funding and excellent working conditions during author's stay at these institutes.

\section{Abstract six functor formalisms}\label{section:abstract-six-functors}

In this section, we remind the reader the notion of a $6$-functor formalism and give some constructions that will be important for the rest of the paper. In particular, we fix the notation that will be freely used in the rest of the paper. After that, we construct the $2$-category of cohomological correspondences that will play a crucial role in the proof of Poincar\'e Duality. \smallskip

For the rest of the section, we fix $\cal{C}$ a category of locally finite type adic $S$-spaces (resp. a category of locally finitely presented $S$-schemes). \smallskip

\subsection{$6$-functor formalisms I}

In this section, we discuss the general notion of a $6$-functor formalism. Since this is the main object of study of this paper, we have decided to spent this section to explicitly set-up all the notation that we will use later. We also wish to convey the idea that almost all familiar structures on the classical $6$-functor formalisms can be defined in this abstract situation in a similar manner. \smallskip

We start by recalling that Y.\,Liu and W.\,Zheng have defined a symmetric monoidal\footnote{See \cite[Def.\,2.0.0.7]{HA} for the precise definition of this notion.} $\infty$-category $\rm{Corr}(\cal{C})\coloneqq \rm{Corr}(\cal{C})_{\rm{all}, \rm{all}}$ of correspondences in $\cal{C}$. We do not explain the full construction here and instead refer to \cite[Prop.\,6.1.3]{Liu-Zheng} (and to \cite[Def.\,A.5.4]{Lucas-thesis} for a nice exposition). However, we specify some lower dimensional data that will be useful for us later:

\begin{rmk}\label{rmk:category-of-corr}
\begin{enumerate}
    \item objects of $\rm{Corr}(\cal{C})$ coincide with objects of $\cal{C}$, i.e. locally finite type adic $S$-spaces;
    \item $1$-edges between $X$ and $Y$ are given by correspondences of the form
    \[
    \begin{tikzcd}
    & Z \arrow{dl}\arrow{dr} & \\
    X & & Y;
    \end{tikzcd}
    \]
    \item in the homotopy category $\rm{h}\rm{Corr}(\cal{C})$, the composition of morphisms $X\leftarrow T \rightarrow Y$ and $Y\leftarrow S \to Z$ is given by the following outer correspondence (in red):
    \[
    \begin{tikzcd}[column sep = 2em, row sep = 1em]
    & & T\times_Y S \arrow{rd} \arrow{ld} \arrow[red, bend left]{ddrr} \arrow[red, bend right]{ddll}& & \\
    & T \arrow{ld} \arrow{rd}& & S\arrow{ld} \arrow{rd} & \\
    X & & Y & & Z;
    \end{tikzcd}
    \]
    \item the tensor product $X\otimes Y$ of two objects $X$ and $Y$ is their cartesian product $X\times_S Y$.
\end{enumerate}
\end{rmk}

In the next definition, we consider the cartesian symmetric monoidal structure on $\Cat_\infty$, the $\infty$-category of (small) $\infty$-categories. 

\begin{defn}\label{defn:weak-6-functors}(\cite[Def.\,A.5.7]{Lucas-thesis}) A {\it weak $6$-functor formalism} is a lax symmetric-monoidal\footnote{By a lax symmetric-monoidal functor, we mean a functor of the associated $\infty$-operads, see \cite[Def.\,2.1.2.7]{HA}} functor
\[
\cal{D}\colon \Corr(\cal{C}) \to \Cat_\infty
\]
such that
\begin{enumerate}
    \item for each morphism $f\colon X\to Y$ in $\cal{C}$, the functors $\cal{D}([X\xleftarrow{\rm{id}}X\xr{f} Y])\colon \cal{D} \to \cal{D}(Y)$ and $\cal{D}([Y \xleftarrow{f} X \xr{\rm{id}}X]) \colon \cal{D}(Y)\to \cal{D}(X)$ admit right adjoints;
    \item for each $X\in \cal{C}$, the symmetric monoidal $\infty$-category $\cal{D}(X)$ is closed (in the sense of \cite[Def.\,4.1.1.15]{HA}). The associated homotopy $1$-category $\rm{h}\cal{D}(X)$ is denoted by $D(X)$. 
\end{enumerate}  
\end{defn}

\begin{rmk}\label{rmk:homotopy-cat-of-six-functor} One can compose $\cal{D}$ with the functor $\rm{h}\colon \Cat_\infty \to \Cat_1^{\simeq}$ to the $(2,1)$-category of categories that sends an $\infty$-category $X$ to its homotopy category $\rm{h}X$. By the universal property of homotopy $2$-categories, this functor (essentially) uniquely descends to the functor
\[
D\coloneqq h\cal{D}\colon \rm{h}_2 \Corr(\cal{C})\to \Cat_1^\simeq
\]
such that $D(X)=\rm{h}\cal{D}(X)$.
\end{rmk}

\begin{rmk}\label{rmk:a-lot-of-data} The data of a weak $6$-functor formalism is a very dense piece of data. Below, we mention some consequences of this definition, and refer to \cite[Def.\,A.5.6, Def.\,A.5.7, Prop.\,A.5.8]{Lucas-thesis} for the discussion on how to derive these consequences from Definition~\ref{defn:weak-6-functors}.
\begin{enumerate}
    \item for each $X\in \cal{C}$, a closed symmetric monoidal $\infty$-category $\cal{D}(X)$. We denote the tensor product functor and the inner Hom functor by 
    \[
     - \otimes - \colon \cal{D}(X) \times \cal{D}(X) \to \cal{D}(X), \text{ and }
    \]
    \[
    \ud{\Hom}_{X}(-, -) \colon \cal{D}(X)^{\rm{op}} \times \cal{D}(X) \to \cal{D}(X);
    \]
    \item for each morphism $f\colon X \to Y$ in $\cal{C}$, we have a symmetric monoidal functor $f^*\colon \cal{D}(Y) \to \cal{D}(X)$, and a functor $f_!\colon \cal{D}(X) \to \cal{D}(Y)$;
    \item for each $f\colon X \to Y$, $f^*$ and $f_!$ admit right adjoints that we denote by $f_*\colon \cal{D}(X) \to \cal{D}(Y)$ and $f^!\colon \cal{D}(Y) \to \cal{D}(X)$;
    \item the functor $f_!$ satisfies the projection formula, i.e., there is an isomorphism 
    \[
    f_!(-)\otimes (-) \simeq f_!(-\otimes f^*(-))
    \]
    of functors $\cal{D}(X) \times \cal{D}(Y) \to \cal{D}(Y)$;
    \item the functor $f_!$ satisfies base-change, i.e., for any cartesian diagram
    \[
    \begin{tikzcd}
    X' \arrow{r}{g'}\arrow{d}{f'} & X \arrow{d}{f} \\
    Y' \arrow{r}{g} & Y,
    \end{tikzcd}
    \]
    there is a specified isomoprhism of functors $g^* \circ f_! \simeq f'_! \circ (g')^*$;
    \item a lot of higher coherences...
\end{enumerate}
\end{rmk}

\begin{notation}\label{notation:morphisms}
\begin{enumerate}
    \item(Unit object) In what follows, we fix a unit object $\bf{1}_S \in \cal{D}(S)$. For each $f\colon X \to S$ in $\cal{C}$, we denote by $\bf{1}_X\coloneqq f^*\left(\bf{1}_S\right)$ the pullback of $\bf{1}_S$ to $X$. It is a unit object in $\cal{D}$ because $f^*$ is a (symmetric) monoidal functor;
    \item\label{notation:morphisms-coprojection}(Co-projection morphism) for any $f\colon X \to Y$ in $\cal{C}$, there is a natural morphism of functors
    \[
    w_{(-), (-)}\colon f^!(-) \otimes f^*(-) \to f^!(-\otimes -)
    \]
    from $\cal{D}(Y)\times \cal{D}(Y)$ to $\cal{D}(X)$ that is defined to be adjoint to the morphism
    \[
    f_!(f^!(-)\otimes f^*(-)) \simeq f_!(f^!(-)) \otimes (-) \xr{\rm{adj}\otimes \rm{id}} -\otimes -;
    \]
    \item\label{notation:morphisms-shriek-base-change}(Upper shriek base-change) If 
    \[
    \begin{tikzcd}
    X' \arrow{d}{f'} \arrow{r}{g'} & X \arrow{d}{f} \\
    Y' \arrow{r}{g} & Y
    \end{tikzcd}
    \]
    is a cartesian diagram in $\cal{C}$, there is a natural morphism $(g')^* \circ f^! \to (f')^! \circ g^*$ defined as an adjoint to 
    \[
    f'_! \circ (g')^* \circ f^! \simeq g^* \circ f_! \circ f^! \xr{g^*(\rm{adj})} g^*,
    \]
    where the first morphism is the base-change morphism.

\end{enumerate}
\end{notation}

For the later use, we prove the following very general (but easy) lemma:

\begin{lemma}\label{lemma:invertible-projection} Let $f\colon X \to Y$ a morphism in $\cal{C}$, and let $\F, \cal{E}$ be objects of $\cal{D}(Y)$. Suppose that $\cal{E}$ is invertible. Then the co-projection morphism
\[
w_{\cal{F}, \cal{E}}\colon f^!\cal{F}\otimes f^* \cal{E} \to f^! (\cal{F} \otimes \cal{E})
\]
is an isomorphism.
\end{lemma}
\begin{proof}
    Consider the morphism $w_{\cal{F} \otimes \cal{E}, \cal{E}^{-1}}\colon f^!(\cal{F}\otimes \cal{E})\otimes f^*(\cal{E}^{-1}) \to f^!\cal{F}$. It induces a morphism $w'\colon f^!(\cal{F}\otimes \cal{E}) \to f^!\cal{F} \otimes f^*\cal{E}$. Using that projection morphisms compose well, one easily checks that $w'$ is the inverse to $w_{\cal{F}, \cal{E}}$ up to a homotopy.  
\end{proof}

\begin{rmk} We put the word ``weak'' in Definition~\ref{defn:weak-6-functors} for the following reasons:
\begin{enumerate}
    \item in practice, $\infty$-categories $\cal{D}(X)$ are stable (in the sense of \cite[Def.\,1.1.1.9]{HA}). It seems reasonable to put this into the definition of a $6$-functor formalism;
    \item furthermore, in practice, there is a natural equivalence between $f_!$ and $f_*$ for a proper morphism $f$, and between $f^!$ and $f^*$ for an étale morphism $f$. It therefore seems appropriate to incorporate these equivalences directly into the definition.
\end{enumerate}
\end{rmk}

We address these matters in Section~\ref{section:2}. In the remainder of this section, we discuss another useful condition to impose on an abstract weak $6$-functor formalism.

\begin{defn}\label{defn:homotopy-invariant} A weak $6$-functor formalism $\cal{D}$ on $\cal{C}$ is {\it weakly $\bf{A}^1$-invariant} if, for every $X\in \cal{C}$ and the morphism $f\colon \bf{A}^1_X \to X$, the natural morphism
\[
\bf{1}_X \to f_* \bf{1}_{\bf{A}^1_X}
\]
is an isomorphism.
\end{defn}

\begin{rmk} In Definition~\ref{defn:homotopy-invariant}, we call $\cal{D}$ {\it weakly} $\bf{A}^1$-invariant to emphasize that we do \emph{not} require that $F \to f_*f^* F$ is an isomorphism for any $F\in \cal{D}(X)$. This stronger requirement is what would usually be called {\it $\bf{A}^1$-invariant}. 
\end{rmk}

In the next lemma, we denote by $\cal{P}ic(\cal{D}(X))$ the $\infty$-subcategory of $\cal{D}(X)$ consisting of invertible objects. 

\begin{lemma}\label{lemma:invertible-ff} Let $\cal{D}$ be an weakly $\bf{A}^1$-invariant weak $6$-functor formalism, let $X\in \cal{C}$, and let $f\colon \bf{A}^1_X \to X$ be the natural morphism. Then the pullback functor
\[
f^* \colon \cal{P}ic\bigl(\cal{D}(X)\bigr) \to \cal{P}ic\bigl(\cal{D}(\bf{A}^1_X)\bigr)
\]
is fully faithful. 
\end{lemma}
\begin{proof}
    We fix two invertible objects $L, L'\in  \cal{P}ic\bigl(\cal{D}(X)\bigr)$. Then the claim follows from the following sequence of isomorphisms:
    \begin{align*}
        \rm{Hom}(f^*L, f^*L') &\simeq \rm{Hom}(L, f_*f^*L') \\
        & \simeq \rm{Hom}(L, f_* \bf{1}_{\bf{A}^1_X} \otimes L') \\
        & \simeq \rm{Hom}(L, L').
    \end{align*}
    The first isomorphism follows from the $(f^*, f_*)$-adjunction, the second isomorphism follows from the projection formula for invertible objects (argue as in the proof of Lemma~\ref{lemma:invertible-projection}). The last isomorphism follows from the $\bf{A}^1$-invariance.
\end{proof}

\subsection{$(\infty, 2)$-category of cohomological correspondences}\label{section:cohomological-correspondences}

The main goal of this section is to construct the $(\infty, 2)$-category of cohomological correspondences, a $2$-categorical variant of which was first introduced in \cite[IV.2.3.3]{Fargue-Scholze} (based on \cite{Lu-Zheng}). We learnt\footnote{Ko Aoki has informed the author that a similar construction has also been known to Adam Dauser.} the arguments of this section from Marc Hoyois. \smallskip

In the rest of the paper, we will never need the $(\infty, 2)$-version of this category; the $2$-categorical version will be sufficient for all our applications. However, it seems that a rigorous explicit construction even of the associated $2$-category is an extremely tedious exercise. Even though it is probably possible to do by hand, we are not aware of any place in the literature where this has been done in full detail. \smallskip

For instance, to verify the pentagon axiom in the context of \'etale cohomology, one needs to check that the pentagon diagram of $5$ associativity constraints is commutative. Each associativity constraint includes $2$ base-change morphisms and $2$ projection formula morphisms (and a lot of implicit identifications). Each base change and projection formula morphism is, in turn, constructed by decomposing a morphism into a composition of an \'etale and a proper morphism. Therefore, the pentagon axiom effectively has at least $40$ arrows involved. Even though it is probably formal that it commutes, it seems really tedious to prove it without some other machinery. \smallskip

Because of this reason, we take another approach (explained to us by Marc Hoyois) that actually produces an $(\infty, 2)$-categorical version of this category. Since, in this approach, it is essentially the same amount of pain to construct it as an $(\infty,2)$-category as to construct it simply as a $2$-category, and the $(\infty, 2)$-categorical version may be useful for other purposes, we write the proof in this generality. We then sketch how the same argument could be run entirely in the realm of $2$-categories. \smallskip

For the rest of the section, we fix a weak $6$-functor formalism $\cal{D}\colon \Corr(\cal{C}) \to \Cat_\infty$ in the sense of Definition~\ref{defn:weak-6-functors}. \smallskip

We start the section by giving an informal definition of the $2$-categorical vesion of the category of correspondences. For this, we need to fix some notation:

\begin{defn} Let $X_1, X_2, X_3$ be objects of $\cal{C}$, and $\F\in D(X_1\times_S X_2)$ and $\G\in D(X_2\times_S X_3)$. Then the {\it composition} $\G\circ \F \in D(X_1\times_S X_3)$ is equal to
\[
p_{1, 3, !}\left(p_{1, 2}^* \F\otimes p_{2, 3}^* \G\right) \in D(X_1\times_S X_3),
\]
where $p_{i, j}\colon X_1\times_S X_2\times_S X_3 \to X_i\times_S X_j$ are the natural projections. 
\end{defn}

\begin{lemma}\label{lemma:constraints} Let $X,Y,Z, W$ be objects of $\cal{C}$, and let $\F\in D(X\times_S Y)$, let $\G\in D(Y\times_S Z)$, and let $\cal{H} \in D(Z\times_S W)$. Then 
\begin{enumerate}
    \item there is a canonical isomorphism $\Delta_!\bf{1}_X \simeq \Delta_!\bf{1}_X \circ \Delta_!\bf{1}_X$, where $\Delta\colon X \to X\times_S X$ is the diagonal morphism;
    \item there is a canonical isomorphism 
    \[
    \cal{H} \circ (\G \circ \F) \simeq (\cal{H} \circ \G) \circ \F.
    \]
\end{enumerate}
\end{lemma}
\begin{proof}
    We claim that both results are formal consequences of shriek base-change and the projection formula. We show the first part, and refer to \cite[\href{https://stacks.math.columbia.edu/tag/0G0F}{Tag 0G0F}]{stacks-project} for the proof of the second part. \smallskip
    
    We first consider the cartesian square
    \[
    \begin{tikzcd}[column sep = 6em]
    X\times_S X \arrow{r}{\Delta\times_S \rm{id}} \arrow{d}{p_1} & X\times_S X\times_S X \arrow{d}{p_{1, 2}} \\
    X \arrow{r}{\Delta} & X\times_S X.
    \end{tikzcd}
    \]
    Then shriek base-change implies that 
    \[
    p_{1, 2}^*\Delta_!\left(\bf{1}_X\right) \simeq (\Delta \times_S \rm{id})_!\left(\bf{1}_{X\times_S X}\right),
    \]
    and similarly $p_{2, 3}^*\Delta_!\left(\bf{1}_X\right) \simeq (\rm{id}\times_S\Delta )_!\left(\bf{1}_{X\times_S X}\right)$. Now we use the cartesian square
    \[
    \begin{tikzcd}[column sep = 6em]
    X\arrow{r}{\Delta} \arrow{d}{\Delta} & X\times_S X \arrow{d}{\rm{id}\times_S \Delta} \\
    X\times_S X \arrow{r}{\Delta \times_S \rm{id}} & X\times_S X \times_S X,
    \end{tikzcd}
    \]
    shriek base-change, and the projection formula to get a sequence of isomorphisms
    \begin{align*}
        \Delta_!\,\bf{1}_X \circ \Delta_!\,\bf{1}_X & \simeq p_{1,3, !} \left(p_{1,2}^* \Delta_!\,\bf{1}_X \otimes p_{2,3}^* \Delta_!\,\bf{1}_X \right) \\
        & \simeq p_{1, 3, !} \bigl(\left(\Delta\times_S \rm{id}\right)_!\,\bf{1}_{X\times_S X} \otimes \left(\rm{id}\times_S \Delta\right)_!\,\bf{1}_{X\times_S X}\bigr) \\ 
        & \simeq p_{1, 3, !}\left(\Delta \times_S \rm{id}\right)_!\bigl(\left(\Delta \times_S \rm{id}\right)^*\left(\rm{id}\times_S \Delta)_!\,\bf{1}_{X\times_S X}\right)\bigr)  \\
        & \simeq p_{1, 3, !} (\Delta \times_S \rm{id})_! \Delta_! \bf{1}_X \\
        & \simeq \Delta_! \left(\bf{1}_X\right). \qedhere
        \end{align*}
\end{proof}

Now we are ready to define the $2$-category of cohomological correspondences.

\begin{defn}\label{defn:category-of-coh-correspondences}(\cite[IV.2.3.3]{Fargue-Scholze}) The {\it $2$-category of cohomological correspondences} $\cal{C}_S$ is the following $2$-category: 
\begin{enumerate}
    \item the objects of $\cal{C}_S$ are objects of $\cal{C}$;
    \item for every two objects $X, Y\in \rm{Ob}(\cal{C}_S)$, the Hom-category is defined as
    \[
    \ud{\rm{Hom}}_{\cal{C}_S}(X, Y) = D(X\times_S Y);
    \]
    \item for every triple $X_1, X_2, X_3 \in \rm{Ob}(\cal{C}_S)$, the composition functor
    \[
    \ud{\rm{Hom}}_{\cal{C}_S}(X_2, X_3) \times \ud{\rm{Hom}}_{\cal{C}_S}(X_1, X_2) \to \ud{\rm{Hom}}_{\cal{C}_S}(X_1, X_3)
    \]
    is defined as 
    \[
    (A, B) \mapsto p_{13, !}\left(p^*_{12} B\otimes p^*_{23} A\right),
    \]
    where $p_{i, j}\colon X_1 \times_S X_2 \times_S X_3$ is the projection on $X_i\times_S X_j$;  
    \item for every $X\in \rm{Ob}(\cal{C}_S)$, the identity $1$-morphism is $\rm{id}_X=\Delta_{!}\left(\bf{1}_X\right)$, where $\Delta\colon X \to X\times_S X$ is the diagonal morphism;
    \item the unit and associativity constraints come from Lemma~\ref{lemma:constraints}.
\end{enumerate}
\end{defn}

In the rest of the section, we show that Definition~\ref{defn:category-of-coh-correspondences} actually defines a $2$-category. As explained at the beginning of this section, the hard part is to verify axiom (P) from \cite[\href{https://kerodon.net/tag/007Q}{Tag 007Q}]{kerodon}.\smallskip

\begin{lemma}\label{lemma:rigid-closed} Let $\cal{D}$ be a symmetric monoidal $\infty$-category such that each object $X\in \cal{D}$ is dualizable (in the sense of \cite[Def.\,4.6.1.7 and Rmk.\,4.6.1.12]{HA}). Then $\cal{D}$ is a closed symmetric monoidal $\infty$-category.
\end{lemma}
\begin{proof}
    Since $\cal{D}$ is symmetric monoidal, it suffices to show that $\cal{D}$ is right closed. In other words, we have to show that, for every object $X\in \cal{D}$, the functor $-\otimes X \colon \cal{D} \to \cal{D}$ admits a right adjoint. Since $X$ is dualizable, there is a dual object $X^\vee$ with the coevaluation and evaluation morphisms
    \[
    c\colon \bf{1}_{\cal{D}} \to X\otimes X^\vee,
    \]
    \[
    e\colon X\otimes X^\vee \to \bf{1}_{\cal{D}}.
    \]
    We claim that the functor $- \otimes X^\vee \colon \cal{D} \to \cal{D}$ is right adjoint to $-\otimes X$. Indeed, we define the unit and counit transformations explicitly as
    \[
    \eta \colon \rm{id} \xr{\rm{id}\otimes c} \rm{id} \otimes X \otimes X^\vee,
    \]
    \[
    \eps \colon \rm{id} \otimes X \otimes X^{\vee} \simeq X\otimes X^\vee \otimes \rm{id} \xr{e\otimes \rm{id}} \rm{id}.
    \]
    One easily checks that this defines the desired adjunction. 
\end{proof}

\begin{lemma}\label{lemma:self-dual} Any object of the symmetric monoidal $\infty$-category $\Corr(\cal{C})$ is self-dual. In particular, the symmetric monoidal $\infty$-categorical structure on $\Corr(\cal{C})$ is closed. 
\end{lemma}
\begin{proof}
    Let $X\in \Corr(\cal{C})$ be an adic $S$-space with the structure morphism $f\colon X \to S$. We wish to show that $X$ is self-dual. For this, we define the co-evaluation morphism $c\colon S \to X\otimes X$ to be represented by the correspondence 
    \[
    S \xleftarrow{f} X \xr{\Delta} X\times_S X,
    \]
    where $\Delta$ is the diagonal morphism. Likewise, we define the evaluation morphism $e\colon X \otimes X \to S$ to be represented by the correspondence
    \[
    X\times_S X \xleftarrow{\Delta} X \xr{f} S. 
    \]
    Then it is easy to check that this morphisms define a self-duality on $X$ (see Remark~\ref{rmk:category-of-corr}).
\end{proof}

Now we are ready to rigorously construct the category $\cal{C}_S$, and even its $(\infty, 2)$-enhancement. A crucial technical tool that we will use is the formalism of $\infty$-categories enhanced in monoidal $\infty$-categories. We refer to \cite{enrichment} for a detailed discussion of this notion, and especially to \cite[Def.\,2.4.5]{enrichment}.

\begin{prop}\label{prop:8-2-category-of-cohomological-correspondences} There is an $(\infty, 2)$-category $\cal{C}_S^{(\infty, 2)}$ such that its $2$-homotopy category $\rm{h}_2\cal{C}_S^{(\infty, 2)}$ is equivalent to $\cal{C}_S$ from Definition~\ref{defn:category-of-coh-correspondences}. In particular, $\cal{C}_S$ is indeed a $2$-category.
\end{prop}
\begin{proof}
    Lemma~\ref{lemma:self-dual} implies that every object in $\rm{Corr}(\cal{C})$ is self-dual. Therefore, Lemma~\ref{lemma:rigid-closed} ensures that $\rm{Corr}(\cal{C})$ is a closed symmetric monoidal $\infty$-category with the inner Hom given by
    \[
    \ud{\Hom}_{\rm{Corr}(\cal{C})}(X, Y) = X\times Y.
    \]
    Therefore, \cite[Cor.\,7.4.10]{enrichment} implies that $\rm{Corr}(\cal{C})$ is enriched over itself. Now we use the lax-monoidal functor $\cal{D}\colon \rm{Corr}(\cal{C}) \to \cal{C}at_\infty$  to transfer\footnote{For this, look at \cite[Def.\,2.4.5, Def.\,2.4.2, and Def.\,2.2.14]{enrichment}.} the constructed above $\rm{Corr}(\cal{C})$-enrichment on $\rm{Corr}(\cal{C})$ to a $\Cat_\infty$-enrichment on $\rm{Corr}(\cal{C})$. This defines the desired $(\infty,2)$-category\footnote{We refer to \cite{8-2} for the relation with other models for the theory of $(\infty,2)$-categories.} $\cal{C}_S^{(\infty, 2)}$ by \cite[Def.\,6.1.5 and Th.\,5.4.6]{enrichment}\footnote{See also \cite[Rmk.\,5.7.13]{enrichment} for the meaning of a somewhat confusing notation $\Cat^{(-)}_{(\infty, k)}$}. Essentially by construction, the associated $2$-homotopy category $\rm{h}_2 \cal{C}_S^{(\infty, 2)}$ is equivalent to $\cal{C}_S$. 
\end{proof}

\begin{rmk} One can run the proof of Proposition~\ref{prop:8-2-category-of-cohomological-correspondences} entirely in the realm of $2$-categories. In this approach, one constructs a $2$-category weakly enriched over $\cal{C}at_1^{\simeq}$ that is tautologically equivalent to $\cal{C}_S$. \smallskip

More precisely, we mention the main changes that one needs to make in the proof of Proposition~\ref{prop:8-2-category-of-cohomological-correspondences} to avoid any mention of $(\infty, 2)$-categories. Firstly, one should use the notion of a monoidal $2$-category\footnote{See \cite[Def.\,12.1.3 and Exp.\,12.1.4]{2-cat})} in place of the notion of a monoidal $\infty$-category. Secondly, one should replace enrichments in the sense of \cite{enrichment} with weak enrichments in the sense of  \cite[\textsection 3]{enriched-bicategories}. Thirdly, one should use the $2$-categorical version of the category of correspondences. Lastly, and one should replace the $\infty$-functor $\cal{D}$ with its $2$-categorical version $D$ from Remark~\ref{rmk:homotopy-cat-of-six-functor}. \smallskip 

Then the same argument works in the world of $2$-categories with the only\footnote{Use \cite[\textsection 13.2]{enriched-bicategories} to transfer a weak enrichment along a lax-monoidal functor.} caveat that we do not know a reference for the fact that a closed monoidal $2$-category is enriched over itself. 
\end{rmk}

\subsection{$6$-functor formalisms II}\label{section:2}

We begin this section by discussing the notion of cohomologically smooth morphisms. Subsequently, following \cite[Lecture VI]{scholze-notes}, we define cohomologically 'etale and proper morphisms. We adopt a minimalistic approach sufficient for our purposes, noting that the cited reference offers a more thorough treatment. These definitions are essential for spelling out the full $6$-functor formalism used in this paper. 

\subsubsection{Cohomologically smooth morphisms}

We follow \cite{Scholze-diamond} and introduce the notion of a cohomologically smooth morphism; the idea is to require the morphism $f\colon X \to Y$ to satisfy Poincar\'e Duality ``up to a trivialization of the dualizing object $f^!\bf{1}_Y$''. \smallskip

In this section, we fix a weak $6$-functor formalism $\cal{D}\colon \rm{Corr}(\cal{C}) \to \Cat_\infty$ in the sense of Definition~\ref{defn:weak-6-functors}.

\begin{defn}\label{defn:cohomologically-smooth} A morphism $f\colon X \to Y$ in $\cal{C}$ is called {\it cohomologically smooth (with respect to $\cal{D}$)} if 
\begin{enumerate}
    \item\label{defn:cohomologically-smooth-1} the co-projection morphism $f^!\left(\bf{1}_Y\right) \otimes f^*(-) \to f^!(-)$ from Notation~\ref{notation:morphisms}(\ref{notation:morphisms-coprojection}) is an equivalence;
    \item\label{defn:cohomologically-smooth-2} the {\it dualizing object} $\omega_f \coloneqq f^!\left(\bf{1}_Y\right)$ is an invertible object of $\cal{D}(X)$;
    \item\label{defn:cohomologically-smooth-3} for any $g\colon Y' \to Y$ with base change $f'\colon X'\coloneqq X\times_Y Y' \to Y'$, properties (\ref{defn:cohomologically-smooth-1}) and (\ref{defn:cohomologically-smooth-2}) also hold for $f'$ and, moreover, the natural map from Notation~\ref{notation:morphisms}(\ref{notation:morphisms-shriek-base-change})
    \[
    g'^* f^!\bf{1}_Y \to f'^! \bf{1}_{Y'}
    \]
    is an isomorphism, where $g'\colon X' \to X$ is the base change of $g$.
\end{enumerate}
\end{defn}

\begin{rmk}\label{rmk:coh-smooth-base-composition} Definition~\ref{defn:cohomologically-smooth} formally implies that cohomologically smooth morphisms are closed under composition and (arbitrary) base change.
\end{rmk}

We first mention some formal properties of this definition:

\begin{lemma}\label{lemma:cohomologically-smooth-base-change} Let 
\[
\begin{tikzcd}
X' \arrow{r}{g'} \arrow{d}{f'} & X \arrow{d}{f} \\
Y' \arrow{r}{g} & Y
\end{tikzcd}
\]
be a cartesian square in $\cal{C}$. Then 
\begin{enumerate}
    \item\label{lemma:cohomologically-smooth-base-change-1} the natural morphism $f'_*\circ (g')^! \to g^! \circ f_*$ is an isomorphism;
    \item\label{lemma:cohomologically-smooth-base-change-2}(Cohomologically smooth base change) the natural morphism $g^*\circ f_* \to (f')_* \circ (g')^*$ is an isomorphism if $g$ is cohomologically smooth;
    \item\label{lemma:cohomologically-smooth-base-change-3} the natural morphism $(g')^* \circ f^! \to (f')^! \circ g^*$ is an isomorphism if either $f$ or $g$ is cohomologically smooth. 
\end{enumerate}
\end{lemma}
All these claims are well-known; we spell out the proof only for the reader's convenience.
\begin{proof}
    The proof of (\ref{lemma:cohomologically-smooth-base-change-1}) is formal: it follows from general shriek base-change by passing to right adjoints. \smallskip
    
    The proof of (\ref{lemma:cohomologically-smooth-base-change-2}) is largely formal. We begin with an intuitive outline, followed by a complete treatment via reference to \cite{Heyer-Mann}.
        
    Since $g$ is cohomologically smooth, there is an invertible object $\omega_g\in \cal{D}(Y')$ such that $g^!(-)\simeq g^*(-) \otimes \omega_g$ and $(g')^!(-)\simeq (g')^*(-) \otimes (f')^*\omega_g$. Then Part~(\ref{lemma:cohomologically-smooth-base-change-1}) provides an equivalence 
    \[
    \omega_g \otimes g^* \circ f_*(-) \simeq g^!\circ f_*(-) \simeq f'_* \circ g'^!(-)\simeq f'_*(f'^*\omega_g \otimes g'^*(-)) \simeq \omega_g \otimes f'_*g'^*(-).
    \]
    Since $\omega_g$ is invertible, we can cancel it to arrive to an equivalence
    \[
    g^* \circ f_* \simeq (f')_* \circ (g')^*.
    \]
    The main difficulty is to check that this isomorphism is the inverse of the natural morphism. We do not do this in this paper and instead explain a different argument. 

    First, we note that \cite[Lemma 4.5.4]{Heyer-Mann} and Definition~\ref{defn:cohomologically-smooth}(\ref{defn:cohomologically-smooth-3}) imply that cohomologically smooth morphisms are $\cal{D}$-suave in the sense of \cite[Def.\,4.5.1]{Heyer-Mann}. Therefore, the result follows from \cite[Lemma 4.5.13(i)]{Heyer-Mann}. 
    
    Part~(\ref{lemma:cohomologically-smooth-base-change-3}) similarly follows from \cite[Lemma 4.5.13(i)]{Heyer-Mann} and the observation that cohomologically smooth morphisms are $\cal{D}$-suave. 
\end{proof}

\subsubsection{Cohomologically proper and \'etale morphisms}

In this section, we fix a weak $6$-functor formalism $\cal{D}\colon \rm{Corr}(\cal{C}) \to \Cat_\infty$ in the sense of Definition~\ref{defn:weak-6-functors}. \smallskip

We seek to axiomatize the conditions $f_! \simeq f_*$ and $f^! \simeq f^*$ by introducing the notions of cohomologically \'etale and cohomologically proper morphisms. We begin with the case of a monomorphism $f\colon X \to Y$ in $\mathcal{C}$ (i.e., the diagonal $\Delta\colon X \to X\times_Y X$ is an isomorphism). This yields the following cartesian diagram:

\begin{equation}\label{eqn:diagonal-iso}
    \begin{tikzcd}
        X \arrow{r}{\rm{id}} \arrow{d}{\rm{id}} & X\arrow{d}{f} \\
        X \arrow{r}{f} & Y.
    \end{tikzcd}
\end{equation}

\begin{construction} Suppose that $f\colon X \to Y$ is a monomorphism in $\cal{C}$. Then 
\begin{enumerate}
    \item there is the natural transformation of functors $\cal{D}(X) \to \cal{D}(Y)$
    \[
        \alpha_f \colon f_! \to f_*
    \]
    defined as the adjoint to the shriek base-change equivalence $f^*f_! \simeq \rm{id}_{\cal{D}(X)}$ coming from Diagram~(\ref{eqn:diagonal-iso});
    \item there is the natural transformation of functors $\cal{D}(Y) \to \cal{D}(X)$
    \[
        \beta_f \colon f^! \to f^*
    \]
    defined as the upper shriek base-change morphism (see Notation~\ref{notation:morphisms}(\ref{notation:morphisms-shriek-base-change})) applied to Diagram~(\ref{eqn:diagonal-iso}).
\end{enumerate}
\end{construction}

\begin{defn}\label{defn:cohomologically-proper-1} A monomorphism $f\colon X \to Y$ is {\it cohomologically proper} (resp. {\it cohomologically \'etale}) if the natural tranformation $\alpha_f\colon f_! \to f_*$ (resp. $\beta_f\colon f^! \to f^*$) is an equivalence .
\end{defn}

Now we move to the case of a general morphism $f\colon X \to Y$ in $\cal{C}$ and consider the commutative diagram
    \begin{equation}\label{eqn:diagonal-morphism-gen}
    \begin{tikzcd}[column sep =3em, row sep = 3em]
        X \arrow{dr}{\Delta} \arrow[rrd, bend left, "\rm{id}"] \arrow[ddr, bend right, swap, "\rm{id}"] & & \\
        & X\times_Y X \arrow{r}{q} \arrow{d}{p} & X\arrow{d}{f} \\
        & X \arrow{r}{f} & Y.
    \end{tikzcd}
    \end{equation}

Note that $\Delta$ is always a monomorphism, so it makes sense to ask if $\Delta$ is cohomologically proper (resp. cohomologically \'etale). 

\begin{construction} Let $f\colon X\to Y$ be a morphism in $\cal{C}$ with the diagonal morphism $\Delta\colon X\to X\times_Y X$. Then
\begin{enumerate}
    \item if $\Delta$ is cohomologically proper, there is a natural transformation of functors $\cal{D}(X) \to \cal{D}(Y)$
    \[
        \alpha_f \colon f_! \to f_*
    \]
    defined as the adjoint to the composition
    \[
        f^*f_! \simeq p_!q^* \xr{p_!\left(\rm{adj}_{\Delta} \circ q^*\right)} p_! \Delta_{*} \Delta^* q^* \simeq p_! \Delta_{!} \Delta^* q^* \simeq \rm{id},
    \]
    where the first isomorphism is the shriek base-change equivalence, the second morphism is induced by the $(\Delta^*, \Delta_{*})$-adjunction, the third isomorphism comes from cohomological properness of $\Delta$, and the last isomorphism comes from the fact that $p\circ     \Delta=\rm{id}_X$ and $q \circ \Delta = \rm{id}_X$;
    \item if $\Delta$ is cohomologically \'etale, there is a natural transformation of functors $\cal{D}(Y) \to \cal{D}(X)$
    \[
        \beta_f \colon f^! \to f^*
    \]
    defined as the composition
    \[
    f^! \simeq \Delta^* q^* f^! \to \Delta^* p^! f^* \simeq \Delta^! p^! f^* \simeq f^*,
    \]
    where the first isomorphism comes from the fact that $q\circ \Delta=\rm{id}_X$, the second isomorphism is induced from the upper shriek base-change morphism (see Notation~\ref{notation:morphisms}(\ref{notation:morphisms-shriek-base-change})), the third isomorphism comes from cohomological \'etaleness of $\Delta$, and the last isomorphism comes from the fact that $p\circ \Delta \simeq \rm{id}_X$. 
\end{enumerate}
\end{construction}

\begin{defn}\label{defn:cohomologically-proper} A morphism $f\colon X \to Y$ in $\cal{C}$ is {\it cohomologically proper} (resp. {\it cohomologically \'etale}) if the diagonal morphism $\Delta\colon X \to X\times_Y X$ is cohomologically proper (resp. cohomologically \'etale) in the sense of Definition~\ref{defn:cohomologically-proper-1}, and the natural transformation $\alpha_f\colon f_! \to f_*$ (resp. $\beta_f\colon f^! \to f^*$) is an equivalence.
\end{defn}


\begin{lemma}\label{lemma:properties-of-etale/proper-maps}
    \begin{enumerate}
    \item\label{lemma:properties-of-etale/proper-maps-1} Cohomologically \'etale and cohomologically proper maps are closed under pullbacks and compositions;
    \item\label{lemma:properties-of-etale/proper-maps-2} cohomologically \'etale maps are cohomologically smooth. 
    \end{enumerate}
\end{lemma}
\begin{proof}
    First, \cite[Lemma 4.6.4]{Heyer-Mann} implies that the notion of cohomologically \'etale maps (resp. cohomologically proper maps) is equivalent to the notion of  $\cal{D}$-\'etale maps (resp.\,$\cal{D}$-proper maps) in the sense of \cite[Def.\,4.6.1]{Heyer-Mann}. Therefore, Part~(\ref{lemma:properties-of-etale/proper-maps-1}) follows from \cite[Lemma 4.6.3]{Heyer-Mann}.

    The second claim essentially follows from the identification $f^*\simeq f^!$ for a cohomologically \'etale morphism $f$. However, it is complicated to show that this isomorphism is compatible with the co-projection morphism $f^*(-)\otimes f^!(\bf{1}) \to f^!(-)$. To avoid dealing with this issue, we also invoke the results from \cite{Heyer-Mann}.
    
    Namely, \cite[Lemma 4.5.4, Cor.\,4.5.11(i), and Lemma 4.5.13]{Heyer-Mann} imply that $\cal{D}$-smooth maps in the sense of \cite[Def.\,4.5.1(i)]{Heyer-Mann} are cohomologically smooth. Therefore, it suffices to show that cohomologically \'etale maps are $\cal{D}$-smooth. Since we already know that cohomologically \'etale morphisms are $\cal{D}$-\'etale in the sense of \cite[Def.\,4.6.1]{Heyer-Mann}, we know that they are $\cal{D}$-suave in the sense of \cite[Def.\,4.5.1]{Heyer-Mann}. Therefore, \cite[Cor.\,4.5.11]{Heyer-Mann} ensures that we only need to show that $f^!\bf{1}_Y \in \cal{D}(X)$ is invertible for a cohomologically \'etale morphism $f\colon X \to Y$. This follows immediately from the identification $f^!\bf{1}_Y\simeq f^*\bf{1}_Y\simeq \bf{1}_X$. 
\end{proof}


\subsubsection{$6$-functor formalisms}

Now we are ready to give the definition of a $6$-functor formalism that will be used in this paper: 

\begin{defn}\label{defn:six-functors} A {\it $6$-functor formalism} is a weak $6$-functor formalism $\cal{D}\colon \Corr(\cal{C}) \to \Cat_\infty$ such that 
\begin{enumerate}
    \item for each $X\in \cal{C}$, the $\infty$-category $\cal{D}(X)$ is stable and presentable;
     \item $\cal{D}^*|_{\cal{C}^{\rm{op}}}\colon \cal{C}^{\rm{op}} \to \Cat_\infty$ satisfies analytic (resp. Zariski in case of schemes) descent, i.e., for any analytic open covering $U=\{U_i \to X\}_{i\in I}$, the natural morphism
    \[
   \cal{D}(X) \to \lim_{n\in \Delta} \prod_{i_1, \dots, i_n\in I}\cal{D}(U_{i_1}\times_X \dots \times_X U_{i_n})
    \]
  is an equivalence. 
    \item every proper morphism $f$ is cohomologically proper\footnote{Strictly speaking, we should first require that any Zariski-closed immersion is cohomologically proper in the sense of Definition~\ref{defn:cohomologically-proper-1}. And then it makes sense to require that any proper morphism is cohomologically proper in the sense of Definition~\ref{defn:cohomologically-proper}.}. In particular, for any proper morphism $f\colon X \to Y$, there is a canonical identification $f_!=f_*$;
    \item every \'etale morphism $f$ is cohomologically \'etale. In particular, for any \'etale morphism $f\colon X \to Y$, there is a canonical identification $f^!=f^*$.
\end{enumerate}
\end{defn}

\begin{rmk}\label{rmk:extended-6-functors} The same definition makes sense if we everywhere replace the category $\cal{C}$ with the category $\cal{C'}$ of $+$-weakly finite type adic $S$-spaces. In the adic world, this version is actually useful for {\it constructing} $6$-functor formalisms in the sense of Definition~\ref{defn:six-functors} because it is easier to construct compactifications in the category $\cal{C}'$ (see \cite[\textsection 5.1]{H3}).
\end{rmk}

\begin{rmk} If $\cal{D}$ is a $6$-functor formalism, all the functors $f^*, f_*, f^!, f_!, \otimes, \ud{\Hom}$ are exact in the sense \cite[Prop.\,1.1.4.1]{HA} (i.e., commute with finite limits and colimits). Indeed, all of them are either left or right adjoints, so they commute with all colimit or limits respectively. But then \cite[Prop.\,1.1.4.1]{HA} implies they must be exact.
\end{rmk}

\begin{rmk} For the most part of the paper, we do not need to assume that $\cal{D}(X)$ are stable $\infty$-categories. However, we lack any examples of non-stable $6$-functor formalisms, so we prefer to put stability of $\cal{D}(X)$ into the definition. In the unstable case, the upper shriek functor $i^!$ usually does not exist even for a Zariski-closed immersion $i$. 
\end{rmk}

\begin{rmk} We recall that any stable $\infty$-category is canonically enriched over $\rm{Sp}$ the $\infty$-category of spectra (see \cite[Ex.\,7.4.14 and Prop.\,4.8.2.18]{enrichment}). In particular, for a $6$-functor formalism $\cal{D}$, $\cal{D}(X)$ is naturally enriched over $\rm{Sp}$ for every $X\in \cal{C}$. 
\end{rmk}

\begin{notation}\label{notation:defns}(Different Homs) For any two objects $\F, \G\in \cal{D}(X)$, we denote their {\it inner Hom} by $\ud{\rm{Hom}}_X(\F, \G)\in \cal{D}(X)$, their {\it Hom-spectrum} by $\rm{Hom}_X(\F, \G) \in \rm{Sp}$, and the {\it Hom-group} in the associated triangulated category $D(X)$ by $\rm{Hom}_{D(X)}(\F, \G)$. The relation between these objects is the following:
    \begin{align*}
    \rm{Hom}_X\left(\bf{1}_X, \ud{\rm{Hom}}_X\left(\F, \G\right)\right) \simeq \rm{Hom}_X\left(\F, \G\right) \\
    \rm{H}^0\left(\rm{Hom}_X\left(\F, \G\right)\right) = \rm{Hom}_{D(X)}(\F, \G). 
    \end{align*} 
\end{notation}

We first show that, for a $6$-functor formalism, the notion of a cohomologically smooth morphism (see Definition~\ref{defn:cohomologically-smooth}) is sufficiently local:  

\begin{lemma}\label{lemma:coh-smooth-formal-properties} Let $\cal{D}$ be a $6$-functor formalism. Then 
\begin{enumerate}
    \item\label{lemma:coh-smooth-formal-properties-2} \'etale morphisms are cohomologically smooth;
    \item\label{lemma:coh-smooth-formal-properties-1} the notion of a cohomologically smooth morphism is analytically (resp. Zariski) local on $X$ and $Y$.
    \item\label{lemma:coh-smooth-formal-properties-3} there is a natural equivalence $\cal{D}(\varnothing) \simeq 0$.
\end{enumerate}
\end{lemma}
\begin{proof}
    The first claim follows immediately from the assumption that \'etale maps are cohomologically \'etale and Lemma~\ref{lemma:properties-of-etale/proper-maps}(\ref{lemma:properties-of-etale/proper-maps-2}). The second claim is formal from analytic (resp. Zariski) descent, the fact that \'etale maps are cohomologically smooth, and Lemma~\ref{lemma:cohomologically-smooth-base-change}(\ref{lemma:cohomologically-smooth-base-change-3}). Claim~(\ref{lemma:coh-smooth-formal-properties-3}) follows from the fact that $\cal{D}(-) \colon \cal{C}^{\rm{op}} \to \Cat_\infty$ is a sheaf, so it sends the final object in $\cal{C}$ (equivalently, the inital object in $\cal{C}^{\mathrm{op}}$) to the final object. 
\end{proof}

\subsubsection{Excision axiom} We now turn to the excision axiom, an additional assumption on the $6$-functor formalism $\cal{D}$ that will play a big role in Section~\ref{section:Chern-classes}.

For the rest of the subsection, we fix a $6$-functor formalism $\cal{D}\colon \mathrm{Corr}(\cal{C}) \to \Cat_\infty$ in the sense of Definition~\ref{defn:six-functors}.  \smallskip

To introduce the excision axiom, we pick a Zariski-closed immersion $i\colon Z\hookrightarrow X$ and its open complement $j\colon U\hookrightarrow X$. Then Lemma~\ref{lemma:coh-smooth-formal-properties}(\ref{lemma:coh-smooth-formal-properties-3}) and the shriek base-change equivalence specify a homotopy $i^*j_! \simeq 0$. Since $j^*\simeq j^!$ due to the very definition of a $6$-functor formalism, so we obtain a natural counit transformation $j_!j^* \to \rm{id}$. Therefore, the data of a homotopy $i^*j_!\simeq 0$ defines a commutative diagram
\begin{equation}\label{eqn:excision}
\begin{tikzcd}
j_!j^* \arrow{r} \arrow{d} & \rm{id}_{\cal{D}(X)} \arrow{d} \\
0 \arrow{r} & i_*i^*
\end{tikzcd}
\end{equation}
in the $\infty$-category $\rm{Fun}(\cal{D}(X), \cal{D}(X))$. In particular, it makes sense to ask if this diagram is cartesian. 

\begin{defn}\label{defn:excision-axiom} A $6$-functor formalism $\cal{D}$ {\it satisfies the excision axiom} if Diagram~(\ref{eqn:excision}) is cartesian for any Zariski-closed $S$-immersion $Z\hookrightarrow X$. An equivalent way to say this is that Diagram~(\ref{eqn:excision}) defines an exact triangle of functors
\begin{equation}\label{eqn:excision-sequence}
j_!j^* \to \rm{id} \to i_*i^*.
\end{equation}
\end{defn}

\begin{rmk}\label{rmk:excision-sequence} If $\cal{D}$ satisfies the excision axiom, we can pass to right adjoints in (\ref{eqn:excision-sequence}) to get an exact triangle of functors
\[
i_*i^! \to \rm{id} \to j_*j^*.
\]
\end{rmk}

\section{Abstract Poincar\'e Duality}\label{section:abstract-duality}

The main goal of this section is to give a ``formal'' proof of (a weak version of) Poincar\'e Duality in any $6$-functor formalism. \smallskip

We recall that the usual proof of Poincar\'e Duality in \'etale cohomology is inductive and does not really tell the exact input one has to check to get Poincar\'e Duality for one particular smooth morphism $f$. We abstract out this condition. Surprisingly, it turns out that one needs a very limited amount of extra data. We give such a characterization in terms of the trace-cycle theories (see Definition~\ref{defn:trace-cycle}). It roughly says that, in order to prove Poincar\'e Duality, one only needs to construct a trace morphism for $f$ and a cycle map of the relative diagonal with some natural compatibilities. \smallskip

After that, we give a minimalistic set of hypothesis that ensures that any smooth morphism is cohomologically smooth. This step reduces the question of proving Poincar\'e Duality to the question of computing the dualizing object. This question is studied in more detail in the next two sections. \smallskip

For the rest of the section, we fix a locally noetherian analytic adic space $S$ (resp. a scheme $S$). We denote by $\cal{C}$ the category of locally finite type (resp. locally finitely presented) adic $S$-spaces (resp. $S$-schemes), and fix a {\it weak} $6$-functor formalism $\cal{D}\colon \rm{Corr}(\cal{C})\to \Cat_\infty$ (see Definition~\ref{defn:six-functors}).  

In what follows, we will freely use the terminology of Section~\ref{section:abstract-six-functors}. In particular, for each $X\in \cal{C}$, we denote the associated stable $\infty$-category by $\cal{D}(X)$ and its triangulated homotopy category by $D(X)$. 

\subsection{Formal Poincar\'e Duality}

In this section, we use the $2$-category of cohomological correspondences $\cal{C}_S$ to reduce the question of proving Poincar\'e Duality to the question of constructing an adjoint to $1$-morphism in the $2$-category of cohomological correspondences $\cal{C}_S$ (see Definition~\ref{defn:category-of-coh-correspondences}). 

We start by considering  the (co-)representable $2$-functor
\[
h^S=\ud{\rm{Hom}}_{\cal{C}_S}(S, -) \colon \cal{C}_S \to \Cat_1
\]
that is a $2$-functor from the $2$-category of cohomological correspondences to the $2$-category of categories (see \cite[\textsection 8.2]{2-cat} for the (dual) theory of representable functors in the $2$-categorical context). \smallskip

It turns out that $h^S$ is quite easy to describe explicitly. For this, it will be convenient to introduce the notion of a Fourier--Mukai functor:

\begin{defn} Let $X_1, X_2$ be objects in $\cal{C}$ and let $\F\in D(X_1\times_S X_2)$. Then the {\it Fourier--Mukai} functor
\[
\rm{FM}_\F \colon D(X_1) \longrightarrow D(X_2)
\]
is defined by the rule
\[
\G \mapsto p_{2, !}\left(p_{1}^* \G \otimes \F\right),
\]
where $p_{i}\colon X_1\times_S X_2 \to X_i$ is the natural projection.
\end{defn}

\begin{rmk}\label{rmk:explicit-yoneda} Explicitly, the functor $h^S$ is quite easy to describe: 
\begin{enumerate}
    \item to every object $X\in \cal{C}_S$, it associates the category
    \[
    h^S(X)=D(X);
    \]
    \item to every pair of objects $X, Y\in \cal{C}_S$, it associates the functor
    \[
    \rm{FM}_{(-)}\colon D\left(X\times_S Y\right) \to \ud{\rm{Fun}}_{\Cat_1}\left(D(X), D(Y)\right)
    \]
    \[
    \F \mapsto \rm{FM}_\F.
    \]
    \end{enumerate}
    It is also possible to describe the identity and composition constraints in terms of the projection formula and shriek base-change. We do not do this here because we will never explicitly need it. 
\end{rmk}

We also recall the definition of adjoint morphisms in a $2$-category. For this, we fix a $2$-category $\cal{C}'$, objects $C$ and $D$ of $\cal{C}'$, and a pair $f\colon C \to D$, $g\colon D \to C$  of $1$-morphisms in $\cal{C}'$.

\begin{defn}\label{defn:2-adjunction}(\cite[\href{https://kerodon.net/tag/02CG}{Tag 02CG}]{kerodon}) An {\it adjunction} between $f$ and $g$ is a pair of $2$-morphisms $(\eta, \epsilon)$, where $\eta\colon \rm{id}_C\to g\circ f$ is a morphism in the category $\HHom_{\cal{C}'}(C, C)$ and $\epsilon \colon f\circ g\to \rm{id}_D$ is a morphism in the category $\HHom_{\cal{C}'}(D, D)$, which satisfy the following compatibility conditions:
\begin{enumerate}
    \item[(Z1)] The composition
    \[
    f \xrightarrow[\sim]{\rho_f^{-1}} f\circ \rm{id}_C \xrightarrow{\rm{id}_f \circ \eta} f\circ (g\circ f) \xrightarrow[\sim]{\alpha_{f, g, f}} (f\circ g) \circ f \xrightarrow{\epsilon \circ \rm{id}_g} \rm{id}_D \circ f \xrightarrow[\sim]{\lambda_f} f
    \]
    is the identity $2$-morphism from $f$ to $f$. Here $\lambda_f$ and $\rho_f$ are the left and right unit constraints of the $2$-category $\cal{C}'$ (see \cite[\href{https://kerodon.net/tag/00EW}{Tag 00EW}]{kerodon}) and $\alpha_{f, g,f }$ is the associativity constraint for the $2$-category $\cal{C}'$.
    \item[(Z2)] The composition
    \[
    g \xrightarrow[\sim]{\lambda_g^{-1}} \rm{id}_C\circ g \xrightarrow{\eta \circ \rm{id}_g} (g\circ f)\circ g \xrightarrow[\sim]{\alpha_{g, f, g}^{-1}} g\circ (f \circ g) \xrightarrow{\rm{id}_g \circ \epsilon} g\circ \rm{id}_D \xrightarrow[\sim]{\rho_g} g
    \]
    is the identity $2$-morphism from $g$ to $g$. 
\end{enumerate}
\end{defn}

\begin{rmk} If $\cal{C}'=\Cat_1$ is the $2$-category of (small) categories. Definition~\ref{defn:2-adjunction} recovers the usual notion of adjunction of functors. 
\end{rmk}

\begin{rmk}\label{rmk:functors-and-adjoints}(\cite[\href{https://kerodon.net/tag/02CM}{Tag 02CM}]{kerodon}) Let $F\colon \cal{C}' \to \cal{C}''$ be a $2$-functor between $2$-categories, and $(f, g)$ is a pair of adjoint morphisms in $\cal{C}'$. Then $(F(c), F(g))$ is a pair of adjoint morphisms in $\cal{C}''$.
\end{rmk}

\begin{prop}\label{prop:abstract-duality}(Formal Poincar\'e Duality. I) Let $f\colon X \to S$ be a morphism in $\cal{C}$. Suppose that the $1$-morphism $A=\bf{1}_{X}\in \HHom_{\cal{C}_S}(X, S)$ is left adjoint to a $1$-morphism $B=I\in \HHom_{\cal{C}_S}(S, X)$. Then the functor 
\[
f_!(-)\colon \cal{D}(X) \longrightarrow \cal{D}(S)
\]
admits a right adjoint given by the formula
\[
f^*(-)\otimes I \colon \cal{D}(S) \longrightarrow \cal{D}(X).
\]
\end{prop}
\begin{proof}
    First of all, it suffices to check that two functors are adjoint by passing to the corresponding homotopy categories by (see \cite[\href{https://kerodon.net/tag/02FX}{Tag 02FX}]{kerodon}), so we can argue with the associated homotopy catetories. \smallskip
    
    We consider the (co)-representable $2$-functor $h^S\colon \cal{C}_S \to \Cat_1$. Remark~\ref{rmk:functors-and-adjoints} guarantees that $\left(h^S(A), h^S(B)\right)$ is a pair of adjoint functors between the categories $h^S(X)$ and $h^S(S)$. Then Remark~\ref{rmk:explicit-yoneda} provides us with the identifications $h^S(X)\simeq D(X)$, $h^S(S)\simeq D(S)$, $h^S(A)=f_!(-)$ and $h^S(B)=f^*(-)\otimes I$. In particular, we conclude that $f_!$ is left adjoint to $f^*(-)\otimes I$. 
\end{proof}

\subsection{Trace-cycle theories}\label{section:trace-cycle}

In this section, we ``decategorify'' Poincar\'e Duality and reduce it to constructing two morphisms subject to two commutativity relations. The main tool for this decategorification process will be the $2$-category of cohomological correspondences $\cal{C}_S$.\smallskip

We recall that throughout this section we have fixed a {\it weak} $6$-functor formalism $\cal{D}\colon \Corr(\cal{C}) \to \Cat_\infty$. \smallskip

\begin{defn} Let $f\colon X \to Y$ be a morphism in $\cal{C}$. A {\it trace theory} on $f$ is a pair $(\omega_f, \tr_f)$ of an invertible object $\omega_f\in \cal{D}(X)$ and a morphism
\[
\rm{tr}_{f}\colon f_!\left(\omega_f\right) \to \bf{1}_Y
\]
in the homotopy category $D(Y)$. 
\end{defn}


\begin{construction}\label{construction:trace-base-change} We point out that shriek base-change implies that any base change of a morphism with a trace theory $\left(\omega_f, \tr_f\right)$ admits a canonical trace theory given by $\left(g'^*\,\omega_f, g^*(\tr_f)\right)$. More precisely, let  
\[
\begin{tikzcd}
    X' \arrow{d}{f'} \arrow{r}{g'} & X \arrow{d}{f} \\
    Y' \arrow{r}{g} & Y
\end{tikzcd}
\]
be a cartesian diagram in $\cal{C}$. Then shriek base-change tells us that the natural morphism
\[
g^*f_!\,\omega_f \overset{\sim}{\longrightarrow}  f'_!\,(g')^*\,\omega_f
\]
is an isomorphism. Therefore, the pullback $g^*(\tr_f)$ defines a trace map
\[
\tr_{f'}\coloneqq g^*\left(\tr_f\right) \colon f'_! \left(g'^{*} \omega_f\right) \to \bf{1}_{Y'}. 
\]
\end{construction}

\begin{warning} The construction of $\tr_{f'}$ depends on the choice of $g\colon Y' \to Y$. However, this will never cause any confusion in the examples where we apply this construction.
\end{warning}


For the next definition, we fix a morphism $f\colon X \to Y$ with the diagonal morphism 
\[
\Delta\colon X\to X\times_Y X
\]
and the projections $p_1, p_2\colon X\times_Y X \to X$. 


\begin{defn}\label{defn:trace-cycle}   A {\it trace-cycle theory} on $f$ is a triple $(\omega_f, \tr_f, \cl_\Delta)$ of 
\begin{enumerate}
    \item an invertible object $\omega_f\in \cal{D}(X)$,
    \item a trace morphism
    \[
        \rm{tr}_{f}\colon f_!\, \omega_f \to \bf{1}_Y
    \]
    in the homotopy category $D(Y)$,
    \item a cycle map  
    \[
    \rm{cl}_{\Delta}\colon \Delta_! \bf{1}_{X} \longrightarrow p_2^*\,\omega_f
    \]
    in the homotopy category $D(X\times_S X)$
\end{enumerate}
such that 
\begin{equation}
\begin{tikzcd}
\bf{1}_X \arrow{r}{\sim} \arrow{d}{\rm{id}}& {p}_{1, !}\left(\Delta_! \bf{1}_X\right) \arrow{d}{{p}_{1, !}\left(\rm{cl}_{\Delta}\right)}\\
\bf{1}_X & \arrow{l}{\tr_{p_1}} p_{1, !} \left(p_{2}^*\, \omega_f\right),
\end{tikzcd}
\end{equation}
\begin{equation}
\begin{tikzcd}[column sep = 4em]
\omega_f \arrow{r}{\sim} \arrow{d}{\rm{id}}& {p}_{2, !}\left(p_{1}^* \omega_f \otimes \Delta_! \bf{1}_X\right)  \arrow{r}{p_{2, !}(\rm{id}\otimes \rm{cl}_\Delta)} & p_{2, !}(p_1^* \omega_f \otimes p_2^* \omega_f)  \arrow{d}{\wr} \\
\omega_f & \arrow{l}{\sim} \bf{1}_X\otimes \omega_f & \arrow{l}{\tr_{p_2} \otimes \rm{id}} p_{2, !}p_1^* \omega_f \otimes \omega_f,
\end{tikzcd}
\end{equation}
commute in $D(X)$ (with the right vertical arrow in the second diagram being the projection formula isomorphism).
\end{defn}

\begin{rmk} The name trace-cycle theory comes from the fact that, in the case of the \'etale $6$-functor formalism, the morphism $\rm{cl}_\Delta$ is equivalent to a class in $\rm{H}^{2d}_\Delta(X\times_Y X, \Z/n\Z(d))$, which comes from the cycle class of the diagonal. 
\end{rmk}

\begin{rmk} Commutativity of the first diagram in Definition~\ref{defn:trace-cycle} should be thought as a formal way of saying that trace of the cycle class of a point is ``universally'' equal to $1$.
\end{rmk}

\begin{rmk}\label{rmk:pullback-trace-cycle} Similarly to Constrution~\ref{construction:trace-base-change}, one can pullback trace-cycle theories along any morphism $Y' \to Y$ in $\cal{C}$. 
\end{rmk}

Now we are ready to show the main result of this section: 

\begin{thm}\label{thm:abstract-poincare-duality}(Formal Poincar\'e Duality II) Let $f\colon X \to S$ be a morphism in $\cal{C}$. Suppose that $f$ admits a trace-cycle theory $(\omega_f, \tr_f, \cl_\Delta)$. Then 
\[
f_!\,(-)\colon \cal{D}(X) \to \cal{D}(S)
\]
admits a right adjoint given by the formula
\[
f^*(-)\otimes \omega_f \colon \cal{D}(S) \to \cal{D}(X).
\]
\end{thm}
\begin{proof}
    By Proposition~\ref{prop:abstract-duality}, it suffices to verify that $A=\bf{1}_X\in \HHom_{\cal{C}_S}(X, S)$ is left adjoint to $B=\omega_f \in \HHom_{\cal{C}_S}(S, X)$ in the $2$-category of cohomological correspondences $\cal{C}_S$. \smallskip
    
    {\it Step~$1$. Construction of the counit $\epsilon \colon A\circ B \to \rm{id}_S$.} By definition, the composition $A \circ B$ corresponds to 
    \[
    f_! \left(\omega_f\right) \in D(S)=\HHom_{\cal{C}_S}(S,S).
    \]
    We also note the the identity morphism $\rm{id}_S$ is given by $\bf{1}_S$ since $S\times_S S=S$. We define the counit $2$-morphism 
    \[
    \epsilon \colon f_! \left(\omega_f\right) \to \bf{1}_S
    \]
    to be the trace morphism $\tr_f$. \smallskip
    
    {\it Step~$2$. Construction of the unit $\eta\colon \rm{id}_X \to B\circ A$.} By definition, the composition $B \circ A$ corresponds to the object $p_2^*\left( \omega_f\right) \in D(X\times_S X)$, and the identity $1$-morphism $\rm{id}_X$ corresponds to the object $\Delta_! \bf{1}_X$. Thus we define the unit $2$-morphism 
    \[
    \eta \colon \Delta_! \bf{1}_X \to p_2^*\left(\omega_f \right)
    \]
    to be the cycle morphism $\rm{cl}_\Delta$. \smallskip
    
    {\it Step~$3$. Verification of the axiom (Z1).} One needs to check that the composition
    \[
    A \xrightarrow[\sim]{\rho_A^{-1}} A\circ \rm{id}_{X} \xrightarrow{\rm{id}_A \circ \eta} A\circ (B\circ A) \xrightarrow[\sim]{\alpha_{A, B, A}} (A\circ B) \circ A \xrightarrow{\epsilon \circ \rm{id}_B} \rm{id}_{S} \circ A \xrightarrow[\sim]{\lambda_A} A
    \]
    is equal to the identity morphism. After unravelling the definitions, this verification essentially boils down to the definition of a trace-cycle theory. We explain this verification in more detail for the convenience of the reader. \smallskip
    \begin{center}
        {\textbf{We make the diagram explicit:}}
    \end{center}
    \begin{enumerate}
        \item First, we see that $A\circ \rm{id}_{X}$ is equal to the 
    \[
    A\circ \rm{id}_{X} =p_{1, !}\left(p_2^* \bf{1}_X\otimes \Delta_!\bf{1}_X\right) = p_{1, !}\,\Delta_!\,\bf{1}_X \in D(X).
    \]
    The right unit constraint $\rho_{A}^{-1}$ is identified with the natural isomorphism
    \[
    \bf{1}_X \xrightarrow{\sim} p_{1, !}\Delta_{!}\bf{1}_X
    \]
    coming from the fact that $p_1\circ \Delta =\rm{id}_X$;
    \item the composition $A \circ (B\circ A)$ is the object
    \[
    A \circ (B\circ A) = p_{1, !}\left(p_2^* \, \omega_f\right)  \in D(X)
    \]
    and the morphism $\rm{id}_{X} \circ \eta$ is given by $p_{1,!}(\rm{cl}_\Delta)$; 
    \item the composition $(A\circ B) \circ A$ is given by $f^*f_!\, \omega_f$ and the associativity constraint $\alpha_{A, B, A}$ is the inverse of the base change isomorphism 
    \[
    f^*f_!\, \omega_f \to p_{1, !}\left(p_2^* \omega_f\right);
    \]
    \item $\rm{id}_{S}\circ A$ is just equal to $\bf{1}_X$ since the diagonal $S \to S\times_S S$ is the identity morphism. And the composition $\epsilon \circ \rm{id}_{A}$ is equal to 
    \[
    f^*(\tr_f) \colon f^*(f_! \omega_f) \to \bf{1}_X;
    \]
    \item finally, the left unit constraint $\lambda_A$ is the identity morphism because the diagonal $S \to S\times_S S$ is the identity morphism. 
    \end{enumerate}
    
    After making all these identifications, we see that the composition $\alpha_{A, \beta, A}\circ(\rm{id}_A \circ \eta)$ is equal to $\tr_{p_1}$ by the very definition of $\tr_{p_1}$. Therefore, the axiom (Z1) boils down to checking that the diagram  
    \[
    \begin{tikzcd}
    \bf{1}_X \arrow{r}{\sim} \arrow{d}{\rm{id}}& {p}_{1, !}\left(\Delta_! \bf{1}_X\right) \arrow{d}{{p}_{1, !}\left(\rm{cl}_{\Delta}\right)}\\
    \bf{1}_X & \arrow{l}{\tr_{p_1}} p_{1, !} \left(p_2^* \omega_f\right)
    \end{tikzcd}
    \]
    commutes. We finish the proof by noting that this is part of the definition of a trace-cycle theory. \smallskip
    
    {\it Step~$4$. Verification of the axiom (Z2).} The verification is essentially the same as the one in Step~$3$. After unravelling all the definitions, the axiom boils down to the commutativity of the second diagram in Definition~\ref{defn:trace-cycle}.
\end{proof}

\begin{cor}\label{cor:pullback-PD} Let $f\colon X \to S$ be as in Theorem~\ref{thm:abstract-poincare-duality}, let $S' \to S$ be a morphism in $\cal{C}$, and let $f'\colon X' \to S'$ be the base change of $f$ along $g$. Then the functor
\[
f'_!(-)\colon \cal{D}(X') \to \cal{D}(S')
\]
admits a right adjoint given by the formula
\[
(f')^*(-)\otimes (g')^*\left( \omega_f \right)\colon \cal{D}(S') \to \cal{D}(X'),
\]
where $g'\colon X' \to X$ is the base-change morphism.
\end{cor}
\begin{proof}
    By Remark~\ref{rmk:pullback-trace-cycle}, we can pullback the trace-cycle theory on $f$ to a trace cycle theory on $f'$. Then we denote by $\cal{C}'$ the slice category $\cal{C}_{/S'}$ and restrict the $6$-functor formalism $\cal{D}$ on $\Corr(\cal{C'})$ to apply Theorem~\ref{thm:abstract-poincare-duality} to $f'$. 
\end{proof}

\begin{rmk} We note that Corollary~\ref{cor:pullback-PD} is already a quite non-trivial statement. It is not clear why duality for $f$ should imply duality for $f'$ from first principles. 
\end{rmk}

\subsection{Cohomological smoothness}

The main goal of this section is to show how Theorem~\ref{thm:abstract-poincare-duality} can be used to formulate a pretty minimalistic set of assumptions that ensures that any smooth morphism is cohomologically smooth (see Definition~\ref{defn:cohomologically-smooth}). This statement should be thought like a version of Poincar\'e Duality without identifying the dualizing object. \smallskip

We recall that throughout this section we have fixed a {\it weak} $6$-functor formalism $\cal{D}\colon \Corr(\cal{C}) \to \Cat_\infty$. \smallskip


\begin{thm}\label{thm:cohomologically-smooth-criterion} Let $f\colon X \to Y$ be a morphism in $\cal{C}$ with a trace-cycle theory $(\omega_f, \tr_f, \cl_\Delta)$. Then $f$ is cohomologically smooth (see Definition~\ref{defn:cohomologically-smooth}). 
\end{thm}
\begin{proof}
    This follows directly from Theorem~\ref{thm:abstract-poincare-duality} and Corollary~\ref{cor:pullback-PD}.
\end{proof}

\begin{rmk}\label{rmk:equivalent-trace-cycle} It is not hard to see that $f\colon X \to Y$ is cohomologically smooth {\it if and only if} $f$ admits a trace-cycle theory. Indeed, we put $\omega_f\coloneqq f^!\bf{1}_Y$, and $\tr_f\colon f_!\,\omega_f \to \bf{1}_Y$ to be the counit of the $(f_!, f^!)$-adjunction. Then we note that Definition~\ref{defn:cohomologically-smooth} implies that 
\[
\bf{1}_X \simeq \Delta^! p_1^! \bf{1}_X \simeq \Delta^! p_2^* \,\omega_f.
\]
Therefore, we define the cycle morphism $\cl_\Delta \colon \Delta_! \bf{1}_X \to p_{2}^*\omega_f$ to be counit the $(\Delta_!, \Delta^!)$-adjunction. We leave it to the reader to verify that the triple $(\omega_f, \tr_f, \cl_\Delta)$ satisfies the assumptions of Definition~\ref{defn:trace-cycle}. 
\end{rmk}

\begin{thm}\label{thm:cohomologically-smooth} Suppose that $\cal{D}$ is a $6$-functor formalism (see Definition~\ref{defn:6-functors-geometric}). Then the relative projective line $g\colon \bf{P}^1_S \to S$ admits a trace-cycle theory $\left(\omega_{g}, \rm{tr}_g, \rm{cl}_\Delta\right)$ if and only if every smooth morphism $f\colon X \to Y$ is cohomologically smooth (with respect to $\cal{D}$).
\end{thm}
\begin{proof}
    The ``if'' part follows directly from Remark~\ref{rmk:equivalent-trace-cycle}. So we prove the ``only if'' part. \smallskip

    By Lemma~\ref{lemma:coh-smooth-formal-properties}(\ref{lemma:coh-smooth-formal-properties-1}), we can argue analytically locally on $X$ and $Y$. Therefore, \cite[Cor.\,1.6.10]{H3} implies that we may assume that $X$ is \'etale over the relative disk $\bf{D}^d_Y$ (resp. affine space $\bf{A}^d_Y$). Now Lemma~\ref{lemma:coh-smooth-formal-properties}(\ref{lemma:coh-smooth-formal-properties-2}) and Remark~\ref{rmk:coh-smooth-base-composition} ensure that it suffices to show that the natural projection $\bf{D}^d_Y \to Y$ (resp. $\bf{A}^d_Y \to Y$) is cohomologically smooth. Then we use Remark~\ref{rmk:coh-smooth-base-composition} once again to reduce the question further to the case of the one-dimensional relative disk $\bf{D}^1_Y \to Y$ (resp. $\bf{A}^1_Y \to Y$). In this case, it suffices to show it for the relative projective line $\bf{P}^{1}_Y \to Y$ compactifying the relative disk (resp. affine line). In this case, the result follows Theorem~\ref{thm:cohomologically-smooth-criterion}.
\end{proof}

\section{Dualizing object}\label{section:dualizing-complex}

Theorem~\ref{thm:cohomologically-smooth} gives a minimalistic condition that implies Poincar\'e Duality up to computing the dualizing object $\omega_f$. Thus the question of proving the full version of Poincar\'e Duality reduces to computing the dualizing object.  \smallskip

In this section, we show that (under a relatively mild assumption) there is always a ``formula'' for the dualizing object $f^!\bf{1}_Y$ in terms of the relative tangent bundle $\rm{T}_f$. The formula says that $\omega_f$ is equal to $0_X^*g^! \bf{1}_X$, where $g\colon \rm{V}_X(\rm{T}_f) \to X$ is the total space of the relative tangent bundle and $0_X$ is the zero section. In particular, it implies that, for the purpose of computing $f^!\bf{1}_Y$, it suffices to assume that $f$ is the total space of a vector bundle and make the computating in a ``neighborhood'' of the zero section. In the next section, we will use this to show that, in the presence of first Chern classes, one can fully trivialize $f^!\bf{1}_Y$ (up to the appropriate Tate twists). \smallskip

We prove the desired formula in two steps: we first use Verdier's diagonal trick to reduce the question of computing $\omega_f$ for a general smooth morphism to the question of computing $s^*\omega_f$ for a smooth morphism $f$ with a section $s$. Then we use a version of the deformation to the normal cone to reduce further to the case, where $f$ is the total space of the (normal) vector bundle. \smallskip

The methods of this section are essentially independent of Section~\ref{section:abstract-duality}. Therefore, we always put into our assumptions that any smooth morphism in $\cal{C}$ is cohomologically smooth with respect to $\cal{D}$ (see Definition~\ref{defn:cohomologically-smooth}). Theorem~\ref{thm:cohomologically-smooth} shows that this is equivalent to the existence of a trace-cycle theory on the relative projective line. \smallskip

Throughout this section, we fix a locally noetherian analytic adic space $S$ (resp. a scheme $S$). We denote by $\cal{C}$ the category of locally finite type (resp. locally finitely presented) adic $S$-spaces (resp. $S$-schemes), and fix a $6$-functor formalism $\cal{D}\colon \rm{Corr}(\cal{C})\to \Cat_\infty$.

\subsection{Verdier's diagonal trick}

We start the discussion by reviewing a version of Verdier's diagonal trick. 

\begin{prop}\label{prop:diagonal-trick} Let $f\colon X \to Y$ be a cohomologically smooth morphism in $\cal{C}$, let $\Delta \colon X \to X\times_Y X$ be the relative diagonal, and let $p \colon X\times_Y X \to X$ be the projection onto the first factor. Then there is a canonical isomorphism
\[
\Delta^* p^! \bf{1}_X \simeq f^! \bf{1}_Y.
\]
\end{prop}
\begin{proof}
    We consider the commutative diagram
    \[
    \begin{tikzcd}[column sep =3em, row sep = 3em]
        X \arrow{dr}{\Delta} \arrow[rrd, bend left, "\rm{id}"] \arrow[ddr, swap, bend right, "\rm{id}"] & & \\
        & X\times_Y X \arrow{r}{q} \arrow{d}{p} & X\arrow{d}{f} \\
        & X \arrow{r}{f} & Y.
    \end{tikzcd}
\]
    Then we have a sequence of isomorphisms:
    \begin{align*}
        f^! \bf{1}_{Y} & \simeq \Delta^* q^* f^! \bf{1}_Y \\
            & \simeq \Delta^* p^! f^* \bf{1}_Y \\
            & \simeq \Delta^* p^! \bf{1}_{X}.
\end{align*}
    The first isomorphism follows from the equality $q\circ \Delta =\rm{id}$. The second isomorphism follows from the base change condition in the definition of cohomological smoothness. The third isomorphism is trivial. 
\end{proof}

We note that Proposition~\ref{prop:diagonal-trick} allows us to reduce the question of computing $f^!$ for a general smooth morphism $f$ to the question of computing $s^* f^!\bf{1}_X$ in the case $f$ has a section $s$. For our later convenience, we axiomize this construction. We recall that $\rm{Pic}(\cal{D}(Y))$ denotes the group of the isomorphism classes of invertible objects in $\cal{D}(Y)$.

\begin{construction}\label{construction:clausen-map} Let $f\colon X \to Y$ be a cohomologically smooth morphism in $\cal{C}$ with a section $s$. Then we denote by $C(f, s) \in \rm{Pic}(\cal{D}(Y))$ the object
\[
C(f, s) \coloneqq s^* f^! \bf{1}_Y.
\]
By definition of cohomological smoothness, the formation of $C(f, s)$ commutes with an arbitrary base change $Y' \to Y$.
\end{construction}

For the rest of this section, we assume that all smooth morphisms in $\cal{C}$ are cohomologically smooth with respect to $\cal{D}$. 

\begin{variant}\label{variant:vector-bundles}
Let $f\colon \rm{V}_X(\cal{E}) \to X$ be the total space of a vector bundle $\cal{E}$ on $X$ with the zero section $s\colon X \to \rm{V}_X(\cal{E})$. Then we define $C_X(\cal{E})\in \rm{Pic}(\cal{D}(X))$ as
\[
C_X(\cal{E}) = C(f, s) \in \cal{D}(X).
\]
\end{variant}

\begin{rmk}\label{rmk:reduction-to-normal} Using this notation, Proposition~\ref{prop:diagonal-trick} tells us that, for a smooth morphism $f\colon X \to Y$, we have a canonical isomorphism $f^! \bf{1}_Y \simeq C(p, \Delta)$. Our goal is to relate $C(p, \Delta)$ to $C_X(\rm{T}_f)$, where $\rm{T}_f$ is the total space of the relative tangent bundle. This will be done in the next section using (a version) of the deformation to the normal cone. \smallskip
\end{rmk}

In the rest of this section, we would like to show that $C_Y(-)$ defines an additive morphism from $K_0(\rm{Vect}(Y))$ to $\rm{Pic}(\cal{D}(Y))$, where $K_0(\rm{Vect}(Y))$ is the Grothendieck group of vector bundle on $Y$. This will not play any role in this paper, but it seems to be of independent interest as it defines an interesting invariant of a $6$-functor formalism.

\begin{lemma} Assume that all smooth morphisms in $\cal{C}$ are cohomologically smooth with respect to $\cal{D}$ and let $X$ be an object of $\cal{C}$. Then the construction $C_X(E)$ defines an additive homomorphism
\[
C_X\colon K_0(\rm{Vect}(X)) \to \rm{Pic}(\cal{D}(X)).
\]
\end{lemma}
\begin{proof}
    The only thing that we need to show is that, for any short exact sequence of vector bundle
    \[
    0 \to \cal{E}' \xr{i} \cal{E} \xr{\pi} \cal{E}'' \to 0,
    \]
    there is an isomorphism 
    \[
    C_X(\cal{E}) \simeq C_X(\cal{E}')\otimes C_X(\cal{E}'').
    \]
    For this, we denote the structure morphism of $\rm{V}_X(\cal{E})$ by $f$ and the zero section by $0_{\cal{E}}$, similarly for $f', f''$ and $0_{\cal{E}'}$ and $0_{\cal{E}''}$ Now we consider the commutative diagram
    \begin{equation}\label{eqn:additivity}
        \begin{tikzcd}[column sep = 3em, row sep = 4em]
        \rm{V}_X(\cal{E}') \arrow{r}{i} \arrow[swap]{d}{f'} & \rm{V}_X(\cal{E}) \arrow{d}{\pi} \arrow[bend left]{dd}{f} \\
        X \arrow{rd}{\rm{id}} \arrow[swap, bend right]{u}{0_{\cal{E'}}} \arrow{r}{0_{\cal{E''}}} \arrow{ru}{0_{\cal{E}}} & \rm{V}_X(\cal{E}'') \arrow{d}{f''} \\
        & X.
    \end{tikzcd}
    \end{equation}
    Now the result follows from the following sequence of isomorphisms:
    \begin{align*}
        C_X(\cal{E}) &=  0_{\cal{E}}^* \left(f^! \bf{1}_X\right)  \\
        & \simeq 0_{\cal{E}}^*\left( \pi^! (f'')^! \bf{1}_X \right)\\
        & \simeq 0_{\cal{E}}^*\pi^* \left( (f'')^! \bf{1}_X\right) \otimes 0_{\cal{E}}^* \left( \pi^! \bf{1}_{\rm{V}(\cal{E''})}\right) \\
        & \simeq 0_{\cal{E}''}^* \left((f'')^! \bf{1}_X \right)\otimes 0_{\cal{E}'}^*  \left(i^* \pi^! \bf{1}_{\rm{V}(\cal{E''})}\right) \\
        & \simeq 0_{\cal{E}''}^* \left( (f'')^! \bf{1}_X \right)\otimes 0_{\cal{E}'}^* \left( (f')^! \bf{1}_{X} \right)\\
        & = C_X(\cal{E}'') \otimes C_X(\cal{E}').
    \end{align*}
    The first equality holds by definition. The second isomorphism comes from the equality $f= f''\circ \pi$. The third isomorphism comes from invertibility of $(f'')^! \bf{1}_X$ and Lemma~\ref{lemma:invertible-projection}. The fourth isomorphism comes from the equalities $\pi\circ 0_{\cal{E}} = 0_{\cal{E}''}$ and $0_{\cal{E}}=i\circ 0_{\cal{E}'}$. The fifth isomorphism comes from the fact that $\pi$ is cohomologically smooth, and so formation of $\pi^! \bf{1}$ commutes with arbitrary base change. And sixth equality holds by definition. 
\end{proof}

\subsection{Deformation to the normal cone}

Our goal in this section is to fulfill the promise made in Remark~\ref{rmk:reduction-to-normal} by showing that $C(p, \Delta) = C_X(\rm{T}_f)$. We achieve this via deformation to the normal cone, an approach to computing the dualizing object due to Dustin Clausen. In particular, a version of this argument appears in \cite[Lecture XIII]{complex-analytic} to compute the dualizing object for liquid sheaves on complex-analytic spaces.\smallskip

Below, we provide two slightly different arguments for the formula $C(p, \Delta) = C_X(\rm{T}_f)$, corresponding to two different assumptions on the $6$-functor formalism $\cal{D}$. 

\subsubsection{Weakly motivic $6$-functor formalisms}

In this subsection, we show that $C(p, \Delta) = C_X(\rm{T}_f)$ under the assumption that $\cal{D}$ is weakly motivic in the following sense:

\begin{defn}\label{defn:motivic-6-functors} Let $\cal{C}$ be the category of locally finite type (resp. locally finitely presented) adic $S$-spaces (resp. $S$-schemes). A $6$-functor formalism $\cal{D}\colon \Corr(\cal{C}) \to \Cat_\infty$ is {\it weakly motivic} if
\begin{enumerate}
    \item weakly $\bf{A}^1$-invariant (see Definition~\ref{defn:homotopy-invariant}),
    \item any smooth morphism $f$ in $\cal{C}$ is cohomologically smooth with respect to $\cal{D}$.
\end{enumerate}
\end{defn}

The main idea of the proof is to deform a Zariski-closed immersion $s \colon Y \to X$ to the zero section of its normal cone. Since this construction relies on blow-ups, we refer to \cite[Section~7]{adic-notes} for a detailed treatment of the $\mathrm{Proj}$ and blow-up construction in the adic setting, and to \cite[Section~5]{adic-notes} for the definition of lci (Zariski-closed) immersions. In the setting of schemes, these notions are standard. \smallskip

\begin{construction}\label{construction:deformation-to-the-normal-cone}(Deformation to the normal cone) Let $Z\xhookrightarrow{i} X$ be an lci $S$-immersion. Then the {\it deformation to the normal} cone $\rm{D}_Z(X)$ is the $S$-space
\[
\rm{D}_Z(X) \coloneqq \rm{Bl}_{Z\times_S 0_S} \left(X\times_S \bf{A}^1_S\right) - \rm{Bl}_{Z}(X).
\]
By definition, it admits a morphism $\pi\colon \rm{D}_Z(X) \to \bf{A}^1_X$. Moreover, by functoriality, there is a morphism 
\[
\rm{D}_Z(Z) = \bf{A}^1_Z \xr{\widetilde{i}} \rm{D}_Z(X)
\]
making the diagram
\[
\begin{tikzcd}
    \bf{A}^1_Z \arrow{r}{\widetilde{i}} \arrow{rd} & \rm{D}_Z(X) \arrow{d}{\pi} \\
    & \bf{A}^1_X
\end{tikzcd}
\]
commute. 
\end{construction}

We first discuss some basic properties of this construction. 

\begin{rmk}\label{rmk:analytification-normal-cone} The construction of the deformation to the normal cone is compatible with the (relative) analytification functor (see \cite[\textsection 6]{adic-notes}). More precisely, let $S=\Spa(A, A^+)$ be a strongly noetherian Tate affinoid, let $S^\alg=\Spec A$, and let $Z\hookrightarrow X$ be an lci immersion of locally finite type $S^\alg$-schemes. Then \cite[Cor.\,6.8 and Lemma 7.15]{adic-notes} imply that $\bigl(\rm{D}_Z(X)\bigr)^{\an/S} \simeq \rm{D}_{Z^{\an/S}}(X^{\an/S})$.
\end{rmk}

\begin{lemma}\label{lemma:etale-base-change-normal-cone} Let $i\colon Z \hookrightarrow Y$ be an lci $S$-immersion, let $\varphi\colon X \to Y$ be an \'etale morphism, and let $i_X \colon Z_X \coloneqq Z\times_YX  \hookrightarrow X$ be the base change of $i$. Then $i_X$ is an lci $S$-immersion, and the natural diagram
\begin{equation}\label{eqn:etale-base-change-normal-cone}
\begin{tikzcd}
    \rm{D}_{Z_X}(X) \arrow{r}{\widetilde{\varphi}} \arrow{d}{\pi_X} & \rm{D}_{Z}(Y) \arrow{d}{\pi_Y} \\
    \bf{A}^1_X \arrow{r}{f\times \rm{id}} & \bf{A}^1_Y
\end{tikzcd}
\end{equation}
is cartesian. In particular, $\rm{D}_{Z_X}(X) \to \rm{D}_{Z}(Y)$ is \'etale. 
\end{lemma}
\begin{proof}
    First, \cite[Lemma 5.6 and Rmk.\,5.7]{adic-notes} imply that $i_X$ is an lci $S$-immersion. Therefore, in order to show that Diagram~(\ref{eqn:etale-base-change-normal-cone}) is cartesian, it suffices to show blow-ups commute with \'etale base change. This follows from \cite[Lemma 5.8 and Rmk.\,7.8]{adic-notes}.
\end{proof}

We now remark on a slightly different approach to $\rm{D}_Z(X)$.

\begin{rmk}\label{rmk:G_m-action}(Local construction) 
\begin{enumerate}
    \item\label{rmk:G_m-action-1} Suppose first that $X=\Spec A$ and $Z=\Spec A/I$ for a regular ideal $I\subset A$. Then \cite[\textsection 5.1, end of p.51]{Fulton} implies that $\rm{D}_Z(X)$ admits a concrete description as the spectrum of the Rees algebra. More precisely\footnote{In the formula below, the convention is that $I^n=A$ for $n<0$.},
\[
\rm{D}_{Z}(X) \simeq \Spec \bigoplus_{n\in \Z} I^{n} T^{-n},
\]
where $\bigoplus_{n\in \Z} I^{n} T^{-n}$ is the natural subring of $A[T, T^{-1}]$. Furthermore, under this isomorphism, the morphism $\pi\colon \rm{D}_Z(X) \to \bf{A}^1_X$ is equal to the morphism
\[
\Spec \bigoplus_{n\in \Z} I^{n} T^{-n} \to \Spec A[T]
\]
induced by the natural morphism $A[T] \to \bigoplus_{n\in \Z} I^{n} T^{-n}$. The fiber over $0_X$ is isomorphic to $\Spec \bigoplus_{n\leq 0} I^n/I^{n+1}$, the total space of the normal bundle\footnote{Here we use the lci assumption to make sure that $I/I^2$ is projective and $I^n/I^{n+1}=\rm{Sym}^n_{A/I} I/I^2$.}.
    \item Now assume that $X=\Spa(A, A^+)$ and $Z=\Spa\bigl(A/I, (A^+/(I\cap A^+))^+\bigr)$ for a strongly noetherian Tate--Huber pair $(A, A^+)$ and a regular ideal $I\subset A$. Set $X^{\rm{alg}}\coloneqq \Spec A$ and $Z^{\rm{alg}}=\Spec A/I$. Then \cite[Lemma 7.15]{adic-notes} implies that 
    \[
    \rm{D}_{Z}(X) \simeq \bigl(\rm{D}_{Z^{\rm{alg}}}(X^{\rm{alg}})\bigr)^{\rm{an}/X},
    \]
    where $(-)^{\rm{an}/X}$ is the relative analytification over $X$. 
    \item If $Z \subset X$ is a general lci $S$-immersion of pure codimension $c$ (either in the analytic or algebraic world). Then $\rm{D}_Z(X)$ can be constructed via gluing of the local constructions.
\end{enumerate}
\end{rmk}

\begin{rmk}\label{rmk:diagram-deformation-to-the-normal-cone} Similarly to the situation in algebraic geometry (or using the local description from Remark~\ref{rmk:G_m-action}), we obtain the following diagram with cartesian squares: 
\[
\begin{tikzcd}
    \bf{G}_{m, Z} \arrow{d}{i\times \rm{id}}\arrow[r, hook] & \bf{A}^1_Z \arrow{d}{\widetilde{i}} & Z \arrow{d}{0_Z} \arrow[l, hook'] \\
    \bf{G}_{m, X}  \arrow{d}{\wr} \arrow[r, hook] & \rm{D}_Z(X) \arrow{d}{\pi} & \arrow[l, hook'] \rm{V}_Z(\cal{N}_{Z/X}) \arrow{d}\\
    \bf{G}_{m, X} \arrow[r, hook] & \bf{A}^1_X & \arrow[l, hook', swap, "0_X"] X.
\end{tikzcd}
\]
Here, the maps $0_X$ and $0_Z$ denote the respective zero sections. 
\end{rmk}

We now apply this construction to the case where $f \colon X \to Y$ is a smooth morphism and $i=s \colon Y \to X$ is a Zariski-closed section of $f$. Note that $s$ is automatically an lci immersion by \cite[Cor.\,5.11]{adic-notes}. In this case, we slightly adjust our notation as follows:

\begin{notation}\label{not:deformation} In the situation as above, we denote $\rm{D}_{s(Y)}(X)$ by $\widetilde{X}$. It fits into the following commutative diagram with cartesian squares:
\begin{equation}\label{eqn:deform-1}
\begin{tikzcd}[row sep = 4em, column sep = 4em]
    \bf{G}_{m, X}  \arrow{d}{f\times \bf{G}_m} \arrow[r, hook] & \widetilde{X} \arrow{d}{\widetilde{f}} & \arrow[l, hook'] \rm{V}_Y(\cal{N}_{s}) \arrow{d}{f_0}\\
    \bf{G}_{m, Y} \arrow[bend left]{u}{s\times \bf{G}_m}\arrow[r, hook] & \bf{A}^1_Y \arrow[bend left]{u}{\widetilde{s}} & \arrow[l, swap, hook', "0_Y"] Y \arrow[bend left]{u}{s_0}.
\end{tikzcd}
\end{equation}
Here, $\widetilde{f}\colon \widetilde{X} \to \bf{A}^1_Y$ is the composition $\widetilde{X} \to \bf{A}^1_X \to \bf{A}^1_Y$, $\widetilde{s}$ is the morphism previously denoted by $\widetilde{i}$, and $s_0$ is the zero section.  
\end{notation}

We now study some of the basic properties of the $\widetilde{X}$ construction. For this, we will need the following general lemma: 

\begin{lemma}\label{lemma:good-coordinates} Let $f\colon X \to Y$ be a smooth morphism in $\cal{C}$, let $s\colon Y \hookrightarrow X$ be a Zariski-closed section of $f$, and let $y\in Y$ be a point. Then there exists an open subset $U\subset X$ and an \'etale morphism $r\colon U \to \AA^d_V$, where $V\coloneqq f(U)$, satisfying the following list of properties:
\begin{enumerate}
    \item $s(y)$ lies in $U$;
    \item $s(V)\subset U$;
    \item the morphism $f|_U \colon U \to V$ factors as the composition
    \[
    U \xr{r} \AA^d_V \xr{\pi} V,
    \]
    where $\pi$ is the natural projection morphism;
    \item $r^{-1}(0_V)=s(Y)\cap U = s(V)$. 
\end{enumerate}
\end{lemma}
\begin{proof}
    We give a proof when $S$ is a locally noetherian analytic adic space and $\cal{C}$ is the category of locally finite type adic $S$-spaces, the case of a scheme $S$ and locally finitely presented $S$-schemes is analogous (in fact, it is easier). 

    By \cite[Lemma 5.9]{adic-notes}, there exists an open neighborhood $U' \subset X$ of $s(y)$ such that the restriction $f|_{U'} \colon U' \to V'$, where $V'=f(U')$, factors as
    \[
    U' \xrightarrow{r'_0} \mathbf{D}^d_{V'} \to V',
    \]
    with $r'_0$ being \'etale and satisfying $r_0'^{-1}(0_{V'}) = s(Y)\cap U'$. Now we replace $r'_0$ with the composition $U'\xr{r'_0} \bf{D}^d_{V'} \hookrightarrow \AA^d_{V'}$ to get a decomposition of $f|_{U'}$ into a composition 
    \[
    U' \xr{r'} \bf{A}^d_{V'} \to V'
    \]
    with $r'$ being \'etale and satisfying $r'^{-1}(0_{V'}) = s(Y)\cap U' \ni s(y)$.
    
    Now we set $U\coloneqq f^{-1}(s^{-1}(U')) \cap U'$, $V\coloneqq f(U)$, and $r\coloneqq r'|_U \colon U \to \AA^d_V$. Then we see that $s(y)\in U$, $s(V) \subset U$, and $r^{-1}(0_V) = s(Y) \cap U$. So we are only left to show that $s(Y)\cap U = s(V)$. It is formal that $s(Y) \cap U \subset s(V)$, while the inclusion $s(V)\subset s(Y)\cap U$ follows from the previous observation that $s(V)\subset U$.    
\end{proof}

Now we consider the special case of the relative affine space $f\colon \AA^d_S \to S$ with the zero-section $0_S\colon S \hookrightarrow \AA^d_S$. 

\begin{lemma}\label{lemma:normal-cone-zero-section} Let $S=\Spa(A, A^+)$ (resp.\,$S=\Spec A$)  be a strongly noetherian Tate affinoid (resp.\,an affine scheme) and let $f\colon \AA^d_S \to S$ be the natural projection with the zero-section $0_S \colon S \hookrightarrow \AA^d_S$. Then there is a canonical $S$-isomorphism 
    \[
    \widetilde{\AA^d_S} \simeq \AA^d_S \times_S \AA^1_S
    \]
    such that, under this isomorphism, the morphism $\widetilde{f}\colon \widetilde{\AA^d_S} \to \AA^1_S$ becomes equal to the projection $\AA^d_S \times_S \AA^1_S \to \AA^1_S$ and $\widetilde{0}_S\colon \AA^1_S \hookrightarrow \widetilde{\AA^d_S}$ becomes equal to the ``zero''-section $0_S  \times \rm{id}_{\AA^1} \colon \AA^1_S \hookrightarrow \AA^d_S \times_S \AA^1_S$.
\end{lemma}
\begin{proof}
    Remark~\ref{rmk:analytification-normal-cone} implies that it suffices to prove the claim only when $S=\Spec A$ is an affine scheme. In this case, we have $\AA^d_S = \Spec A[T_1, \dots, T_d]$. Thus, Remark~\ref{rmk:G_m-action}(\ref{rmk:G_m-action-1}) implies that $\widetilde{\AA^d_S} \simeq \Spec \bigoplus_{n \in \Z} I^n t^{-n}$, where $I=(T_1, \dots, T_d) \subset A[T_1, \dots, T_d]$. Now we see that the unique $A$-linear morphism
    \[
    A[X_1, \dots, X_d][T] \to \bigoplus_{n\in \Z} I^n t^{-n}
    \]
    sending $X_i$ to $T_it^{-1}$ and $T$ to $t$ is an isomorphism. In particular, we get an $S$-isomorphism $\widetilde{\AA^d_S}\simeq \AA^d_S \times_S \AA^1_S$. After unraveling all the definitions, we also see that, under this isomorphism, the morphism $\widetilde{f}$ becomes equal to the natural projection $\AA^d_S \times_S \AA^1_S \to \AA^1_S$ and $\widetilde{0}_S\colon \AA^1_S \hookrightarrow \widetilde{\AA^d_S}$ becomes equal to the ``zero''-section $0_S  \times \rm{id}_{\AA^1} \colon \AA^1_S \hookrightarrow \AA^d_S \times_S \AA^1_S$.
\end{proof}

Finally, we show that the morphism $\widetilde{f}\colon \widetilde{X} \to \AA^1_Y$ is smooth in full generality. This observation will play a crucial role in our computation of the dualizing complex $\omega_f =f^!\mathbf{1}_Y$ for a smooth morphism $f\colon X \to Y$ (see Theorem~\ref{thm:formula-dualizing-complex}). 

\begin{lemma}\label{lemma:smooth-normal-cone-deformation} Let $f\colon X \to Y$ be a smooth morphism of locally finite type adic $S$-spaces (reps.\,locally finitely presented $S$-schemes), let $s\colon Y \hookrightarrow X$ be a Zariski-closed section of $f$. Then $\widetilde{f}\colon \widetilde{X} \to \bf{A}^1_Y$ is smooth.
\end{lemma}
\begin{proof}
    Pick a point $x\in s(Y)\subset X$ and set $U_x\subset X$ and $V_x\subset Y$ to be the open subspaces from Lemma~\ref{lemma:good-coordinates} applied to $f$, $s$, and $f(x)$ (so $s(f(x))=x$). We also set $U\coloneqq X\smallsetminus s(Y)$. 

    Now Lemma~\ref{lemma:etale-base-change-normal-cone} and the properties of $U_x$ and $V_x$ imply that, for each $x\in s(Y)\subset X$, we have a commutative diagram
    \[
    \begin{tikzcd}
        \mathrm{D}_{s(V_x)}(U_x) = \widetilde{U_x} \arrow[r, hook] \arrow[dd, bend right =50, "\widetilde{f|_{U_x}}"'] \arrow{d} & \mathrm{D}_{s(Y)}(X) =\widetilde{X} \arrow{d}  \arrow[dd, bend left = 50, "\widetilde{f}"]\\
        \AA^1_{U_x} \arrow[r, hook]\arrow{d} & \AA^1_{X} \arrow{d} \\
        \AA^1_{V_x} \arrow[r, hook] & \AA^1_Y
    \end{tikzcd}
    \]
    with the top square being cartesian. Similarly, we have a commutative diagram
    \[
    \begin{tikzcd}
        \mathrm{D}_{\varnothing}(U) \arrow[r, hook] \arrow{d} & \mathrm{D}_{s(Y)}(X) =\widetilde{X} \arrow{d}  \arrow[dd, bend left = 50, "\widetilde{f}"]\\
        \AA^1_{U} \arrow[rd, "f|_U\times \id"'] \arrow[r, hook] & \AA^1_{X} \arrow{d} \\
         & \AA^1_Y
    \end{tikzcd}
    \]
    with the top square being cartesian. Since $\AA^1_X = \AA^1_U \bigcup \cup_{x\in s(Y)} \AA^1_{U_x}$ is an open covering, we conclude that $\widetilde{X} = \mathrm{D}_{\varnothing}(U) \bigcup \cup_{x\in s(Y)} \widetilde{U_x}$ is an open covering as well. Therefore, it suffices to show that the natural morphism $\mathrm{D}_{\varnothing}(U) \to \AA^1_U$ is smooth and that, for each $x\in s(Y)\subset X$, the natural morphism $\widetilde{U_x} \to \AA^1_{V_x}$ is smooth. 

    For the first claim, we note that Construction~\ref{construction:deformation-to-the-normal-cone} directly implies that $\mathrm{D}_{\varnothing}(U)=\mathbf{G}_{m, U}$, so the natural morphism $\mathrm{D}_{\varnothing}(U)=\mathbf{G}_{m, U} \to \AA^1_U$ is clearly smooth (it is even an open immersion). Therefore, we reduce the question to showing that, for each $x\in s(Y)\subset X$, the morphism 
    \[
    \widetilde{f|_{U_x}} \colon \widetilde{U_x} \to \AA^1_{V_x}
    \]
    is smooth. In other words, we can replace $X$ with $U_x$ and $Y$ with $V_x$ to assume that $f\colon X \to Y$ factors as the composition 
    \[
    X \xr{r} \AA^d_Y \xr{\pi} Y
    \]
    such that $s(Y)= r^{-1}(0_Y)$. In this case, we apply Lemma~\ref{lemma:etale-base-change-normal-cone} to $r$ to reduce the question to showing that the morphism $\widetilde{\pi} \colon \widetilde{\AA^d_Y} \to \AA^1_Y$ is smooth. This follows immediately from the Lemma~\ref{lemma:normal-cone-zero-section}. 
\end{proof}

We will use Lemma~\ref{lemma:smooth-normal-cone-deformation} throughout the remainder of this paper, often without explicit reference.

The next proposition is the key statement of this section.

\begin{prop}\label{prop:able-to-deform} Suppose the $6$-functor formalism $\cal{D}$ is weakly motivic (see Definition~\ref{defn:motivic-6-functors}). Let $f\colon X \to Y$ be a smooth morphism, let $s\colon Y\hookrightarrow  X$ be a Zariski-closed section of $f$, and let $\widetilde{f}\colon \widetilde{X} \to \bf{A}^1_Y$ and $\widetilde{s}\colon \bf{A}^1_Y \hookrightarrow \widetilde{X}$ be as in Notation~\ref{not:deformation}.  Then the invertible object
\[
C(\widetilde{f}, \widetilde{s}) = \widetilde{s}^{*} \widetilde{f}^{!} \bf{1}_{\bf{A}^1_{Y}} \in \cal{P}ic\bigl(\cal{D}\left(\bf{A}^1_Y\right)\bigr)
\]
lies in the essential image of the pullback functor $\cal{P}ic\bigl(\cal{D}(Y)\bigr) \to \cal{P}ic\bigl(\cal{D}(\bf{A}^1_Y)\bigr)$.
\end{prop}
\begin{proof}
    {\it Step~$1$. Localize on $Y$ and reduce to a simpler situation.} We first note that Lemma~\ref{lemma:invertible-ff} ensures that the functor
    \[
    g^* \colon \cal{P}ic\left(\cal{D}\left(Y\right)\right) \to \cal{P}ic\left(\cal{D}\left(\bf{A}^1_Y\right)\right)
    \]
    is fully faithful for any $Y\in \cal{C}$. Therefore, using the analytic (resp. Zariski) descent, we can check that an object lies in the essential image of $g^*$ {\it locally on} $Y$. \smallskip
    
    We fix a point $y\in Y$ and apply Lemma~\ref{lemma:good-coordinates} to $f\colon X \to Y$ and $x=s(y)$ to find an open subspace $s(y)\in U\subset X$ and an \'etale morphism $r\colon U \to \AA^d_V$, where $V=f(U)$, such that $f|_U\colon U \to V$ factors as 
    \[
    U \xr{r} \bf{A}^d_V \to V,
    \]
    and satisfies the following properties: $s(V) \subset U$ and $r^{-1}(0_V)=s(V)$. Then we consider the square
    \[
    \begin{tikzcd}[row sep = 4em, column sep = 4em]
        U \arrow[r, hook, "j"] \arrow{d}{f|_U} & X_V \arrow{d}{f_V} \arrow[r, hook] & X \arrow{d}{f}\\
        V \arrow[bend left]{u}{s|^U_V} \arrow[r, "\rm{id}"] & V\arrow[bend left]{u}{s|_V} \arrow[r, hook] & Y \arrow[bend left]{u}{s},
    \end{tikzcd}
    \]
    where the right square is cartesian and all horizontal arrows are open immersions. Then Lemma~\ref{lemma:etale-base-change-normal-cone} ensures that we get a commutative diagram
    \[
    \begin{tikzcd}
        \AA^1_V \arrow[r, "\id"] \arrow[d, hook, "\widetilde{s|_V^U}"]& \AA^1_V \arrow[d, hook, "\widetilde{s|_V}"] \arrow[r, hook] & \AA^1_X \arrow[d, hook, "\widetilde{s}"] \\
        \widetilde{U} \arrow[d, "\widetilde{f|_U}"] \arrow[r, hook, "\widetilde{j}"] & \widetilde{X_V} \arrow[d, "\widetilde{f_V}"] \arrow[r, hook] & \widetilde{X} \arrow[d, "\widetilde{f}"] \\
        \AA^1_V \arrow[r, "\id"] & \AA^1_V \arrow[r, hook] & \AA^1_Y,
    \end{tikzcd}
    \]
    where the two right squares are cartesian and all horizontal arrows are open immersions. Therefore, Lemma~\ref{lemma:cohomologically-smooth-base-change}(\ref{lemma:cohomologically-smooth-base-change-3}) implies that $C(\widetilde{f}, \widetilde{s})|_{\bf{A}^1_V} \simeq  C(\widetilde{f_V}, \widetilde{s_V})$. Now we wish to relate $C(\widetilde{f_V}, \widetilde{s|_V})$ to $C(\widetilde{f|_U}, \widetilde{s|^U_V})$.

    For this, we recall that open immersions are cohomologically \'etale by the definition of a $6$-functor formalism (see Definition~\ref{defn:six-functors}). In particular, we have $\widetilde{j}^*\simeq \widetilde{j}^!$. Therefore, we see that 
    \[
    C(\widetilde{f_V}, \widetilde{s|_V}) = \widetilde{s|_V}^* \widetilde{f_V}^! \bf{1} \simeq \widetilde{s|_V^U}^* \widetilde{j}^* \widetilde{f_V}^!\bf{1} \simeq \widetilde{s|_V^U}^* \widetilde{j}^! \widetilde{f_V}^!\bf{1} \simeq \widetilde{s|_V^U}^* \widetilde{f|_U}^! \bf{1} = C(\widetilde{f|_U}, s|_V^U). 
    \]
    Combining these things together, we get a canonical identification
    \[
    C(\widetilde{f}, \widetilde{s})|_{\bf{A}^1_V} \simeq  C(\widetilde{f_V}, \widetilde{s_V}) \simeq C(\widetilde{f|_U}, \widetilde{s|_V}) \in \rm{Pic}\left(\cal{D}\left(\bf{A}^1_V\right)\right).
    \]
    Since we already showed that the claim of this proposition is {\it local on $Y$}, we may replace the pair $(f, s)$ by the pair $(f|_U, s|^U_V)$ to assume that $f\colon X \to Y$ factors as 
    \[
    X \xr{r} \bf{A}^d_Y \xr{h} Y
    \]
    such that $r$ is \'etale and $s(Y)=r^{-1}(0_Y)$. \smallskip
    
    {\it Step~$2$. Reduce further to the case of the relative affine space $\bf{A}^d_Y \to Y$ and the zero section $s=0_Y$}. We consider the cartesian square
    \[
    \begin{tikzcd}
        Y \arrow[r, hook, "s"] \arrow{d}{\rm{id}} & X \arrow{d}{r} \\
        Y \arrow[r, hook, "0_Y"] & \AA^d_Y.
    \end{tikzcd}
    \]
    Lemma~\ref{lemma:etale-base-change-normal-cone} applied to $X \to \AA^d_Y$ ensures that we have a commutative diagram
    \[
    \begin{tikzcd}
    \bf{A}^1_Y \arrow[d, hook, "\widetilde{s}"] \arrow{r}{\rm{id}}  & \bf{A}^1_Y \arrow[d, hook, "\widetilde{0}_Y"] \\
    \widetilde{X}\arrow{d}{\widetilde{f}} \arrow{r}{\widetilde{r}} &\widetilde{\bf{A}^{d}_Y} \arrow{d}{\widetilde{h}} \\
    \bf{A}^1_Y \arrow{r}{\rm{id}} &  \bf{A}^1_Y,
    \end{tikzcd}
    \]
    where $\widetilde{r}$ is \'etale. In particular, there is a natural equivalence $r^!\simeq r^*$. Therefore, we obtain the following sequence of isomorphisms
    \[
    C(\widetilde{f}, \widetilde{s}) = \widetilde{s}^*\widetilde{f}^! \bf{1}_{\bf{A}^1_Y} \simeq \widetilde{s}^* \widetilde{r}^! \widetilde{h}^! \bf{1}_{\bf{A}^1_Y} \simeq \widetilde{s}^* \widetilde{r}^* \widetilde{h}^! \bf{1}_{\bf{A}^1_Y} \simeq \widetilde{0}_Y^* \widetilde{h}^! \bf{1}_{\bf{A}^1_Y} = C(\widetilde{h}, \widetilde{0}_Y) \in \rm{Pic}(\cal{D}(\bf{A}^1_Y)).
    \]
    Therefore, it suffices to show the claim for the natural projection $f\colon X=\bf{A}^d_Y \to Y$ and the zero section $s=0_Y$. Using that the formation of $C$ commutes with arbitrary base change, we can reduce further to the case $S=Y$. \smallskip 
    
    {\it Step~$3$. The case of the relative affine space $f\colon X=\AA^d_S \to S$ and the zero section $0_S$}. Since the question is local on $S$ (see Step~$1$), we can assume that $S=\Spa(A, A^+)$ is a strongly noetherian Tate affinoid (resp.\,an affine scheme). Then Lemma~\ref{lemma:normal-cone-zero-section} ensures that there is an $S$-isomorphism
    \[
    \widetilde{\AA^d_S} \simeq \AA^d_S \times_S \AA^1_S
    \]
    such that the projection $\widetilde{f}\colon \widetilde{\AA^d_S} \to \AA^1_S$ corresponds to the projection onto the second factor, and the section $\widetilde{0}_S\colon \AA^1_S \hookrightarrow \widetilde{\AA^d_S}$ corresponds to the ``zero''-section $0_S  \times \rm{id}_{\AA^1} \colon \AA^1_S \hookrightarrow \AA^d_S \times_S \AA^1_S$.
    In particular, there is a commutative square
    \[
    \begin{tikzcd}
        \AA^1_S \arrow{r} \arrow[d, hook, "\widetilde{0}_S"]& S \arrow[d, hook, "0_S"] \\
        \widetilde{\AA^d_S} \arrow{r} \arrow[d, "\widetilde{f}"] & \bf{A}^d_S\arrow{d}{g_d} \\
        \bf{A}^1_S \arrow{r}{g_1} & S,
    \end{tikzcd}
    \]
    where each square is cartesian. Since the formation of $C(f, s)$ commutes with arbitrary base change, we conclude that 
    \[
    C\bigl(\widetilde{f}, \widetilde{0}_S\bigr) \simeq g_1^* C\bigl(g_d, 0_S\bigr).
    \]
    This finishes the proof. 
\end{proof}

\begin{cor}\label{cor:finally-deform} In the notation of Proposition~\ref{prop:able-to-deform}, there is a canonical isomorphism
\[
C(f, s) \simeq C_Y(\cal{N}_{s}) \in \cal{D}(Y),
\]
where $\cal{N}_s$ is the normal bundle of $s(Y)$ in $X$.
\end{cor}
\begin{proof}
    Consider the deformation to the normal cone: 
\[
\begin{tikzcd}[row sep = 4em, column sep = 4em]
    \bf{G}_{m, X}  \arrow{d}{f\times \bf{G}_m} \arrow[r, hook] & \widetilde{X} \arrow{d}{\widetilde{f}} & \arrow[l, hook'] \rm{V}_Y(\cal{N}_{s}) \arrow{d}{f_0}\\
    \bf{G}_{m, Y} \arrow[bend left]{u}{s\times \bf{G}_m}\arrow[r, hook] & \bf{A}^1_Y \arrow[bend left]{u}{\widetilde{s}} & \arrow[l, hook', "0_Y"] Y \arrow[bend left]{u}{s_0},
\end{tikzcd}
\]
Then we know that the the formation of $C(\widetilde{f}, \widetilde{s})$ commutes with arbitrary base change\footnote{This step implicitly uses that $\widetilde{f}$ is a smooth morphism. This can either be seen from the proof of Proposition~\ref{prop:able-to-deform} or from the local description in Remark~\ref{rmk:G_m-action}}. Therefore we get isomorphisms
\[
C(\widetilde{f}, \widetilde{s})|_{0_Y} \simeq C(f_0, 0_Y)=C_Y(\cal{N}_s) \in \cal{D}(Y),
\]
\[
C(\widetilde{f}, \widetilde{s})|_{1_Y} \simeq C(f, s) \in \cal{D}(Y). 
\]
Now we note that Proposition~\ref{prop:able-to-deform} comes as a pullback from $\cal{D}(Y)$, so we get a canonical identification of its ``fibers''
\[
C(f, s) \simeq C(\widetilde{f}, \widetilde{s})|_{1_Y}  \simeq C(\widetilde{f}, \widetilde{s})|_{0_Y} \simeq C_S(\cal{N}_s). \qedhere
\]
\end{proof}

\begin{thm}\label{thm:formula-dualizing-complex} Suppose the $6$-functor formalism $\cal{D}$ is weakly motivic. Let $f\colon X \to Y$ be a smooth morphism. Then there is a canonical isomorphism
\[
f^!\bf{1}_Y \simeq C_X(\rm{T}_f) \in \cal{D}(X),
\]
where $\mathrm{T}_f$ is the relative tangent bundle of $f$ and $C_X(\mathrm{T}_f)$ is from Variant~\ref{variant:vector-bundles}.
\end{thm}
\begin{proof}
    Proposition~\ref{prop:diagonal-trick} says that 
    \[
    f^! \bf{1}_Y \simeq \Delta^* p^! \bf{1}_X = C(p, \Delta),
    \]
    where $p\colon X \times_Y X \to X$ is the projection onto the first factor, and $\Delta \colon X \to X\times_Y X$ is the diagonal morphism. Then \cite[Lemma B.7.3]{quotients} ensures that we can decompose $\Delta$ as
    \[
    X \xr{i} U \xr{j} X\times_Y X,
    \]
    where $i$ is a Zariski-closed immersion, and $j$ is an open immersion. Then we see that
    \[
    C(p, \Delta) = \Delta^* p^! \bf{1}_X \simeq i^* j^* p^! \bf{1}_X \simeq i^* (p\circ j)^! \bf{1}_X = C(i, p\circ j).
    \]
    Now $i$ is a Zariski-closed section of the smooth morphism $g\coloneqq p\circ j\colon U \to X$. So the result follows directly from Corollary~\ref{cor:finally-deform} and the observation that the normal bundle of the (relative) diagonal is equal to the (relative) tangent bundle $\rm{T}_f$.  
\end{proof}

\subsubsection{Geometric $6$-functor formalisms}

In this section, we perform the deformation to the normal cone type argument under a different assumption on $\cal{D}$.

\begin{defn}\label{defn:6-functors-geometric} A $6$-functor formalism $\cal{D}\colon \Corr(\cal{C}) \to \Cat_\infty$ is {\it pre-geometric} if, for every object $Y\in \cal{C}$ and an invertible object $L\in \rm{Pic}(\bf{P}^1_Y)$, there is an isomorphism  $L|_{0_Y} \cong L|_{1_Y}$ inside $\cal{D}(Y)$.

  A pre-geometric $6$-functor formalism $\cal{D}\colon \Corr(\cal{C}) \to \Cat_\infty$ is {\it geometric} if any smooth moprhism $f$ in $\cal{C}$ is cohomologically smooth with respect to $\cal{D}$.
\end{defn}

To adapt the proof of Theorem~\ref{thm:formula-dualizing-complex} to a geometric $6$-functor formalism $\cal{D}$, we need to introduce the projective version of Construction~\ref{construction:deformation-to-the-normal-cone}

\begin{construction}\label{construction:projective-deformation-to-the-normal-cone}(Projective deformation to the normal cone) Let $Z\xhookrightarrow{i} X$ be an lci $S$-immersion. Then the {\it projective deformation to the normal} cone $\rm{PD}_Z(X)$ is the $S$-space
\[
\rm{PD}_Z(X) \coloneqq \rm{Bl}_{Z\times_S 0_S} \left(X\times_S \bf{P}^1_S\right) - \rm{Bl}_{Z}(X).
\]
By definition, it admits a morphism $\pi\colon \rm{PD}_Z(X) \to \bf{P}^1_X$. Moreover, by functoriality, there is a morphism 
\[
\rm{PD}_Z(Z) = \bf{P}^1_Z \xr{\widetilde{i}} \rm{PD}_Z(X)
\]
making the diagram
\[
\begin{tikzcd}
    \bf{P}^1_Z \arrow{r}{\widetilde{i}} \arrow{rd} & \rm{PD}_Z(X) \arrow{d}{\pi} \\
    & \bf{P}^1_X
\end{tikzcd}
\]
commute. 
\end{construction}

Similarly to Notation~\ref{not:deformation}, we specialize Construction~\ref{construction:projective-deformation-to-the-normal-cone} to the case when $f\colon X \to Y$ is a smooth morphism and $i=s\colon Y \hookrightarrow X$ is a Zariski-closed section of $f$ (it is automatically an lci immersion by \cite[Cor.\,5.11]{adic-notes}). In this case, we slightly change our notation as follows:

\begin{notation}\label{not:proj-deformation} In the situation as above, we denote $\rm{PD}_{s(Y)}(X)$ by $\widetilde{X}$. It fits into the following commutative diagram with cartesian squares:
\begin{equation}
\begin{tikzcd}[row sep = 4em, column sep = 4em]
    \bf{A}^1_{X}  \arrow{d}{f\times \bf{A}^1_S} \arrow[r, hook] & \widetilde{X} \arrow{d}{\widetilde{f}} & \arrow[l, hook'] \rm{V}_Y(\cal{N}_{s}) \arrow{d}{f_0}\\
    \bf{A}^1_{Y} \arrow[bend left]{u}{s\times \bf{A}^1_S}\arrow[r, hook, "j"] & \bf{P}^1_Y \arrow[bend left]{u}{\widetilde{s}} & \arrow[l, hook', "0_Y"'] Y \arrow[bend left]{u}{s_0}.
\end{tikzcd}
\end{equation}
Here, $j$ is the open complement to the zero section $0_Y \colon Y \hookrightarrow \bf{P}^1_Y$. 
\end{notation}

\begin{thm}\label{thm:formula-dualizing-complex-geometric} Suppose the $6$-functor formalism $\cal{D}$ is geometric. Let $f\colon X \to Y$ be a smooth morphism. Then there is an isomorphism
\[
f^!\bf{1}_Y \cong C_X(\rm{T}_f) \in \cal{D}(X),
\]
where $\mathrm{T}_f$ is the relative tangent bundle of $f$ and $C_X(\mathrm{T}_f)$ is from Variant~\ref{variant:vector-bundles}.
\end{thm}
\begin{proof}
    The same proof as in Theorem~\ref{thm:formula-dualizing-complex} reduces the question to proving that $C(f, s)\simeq C_Y(\cal{N}_s)$ for a smooth morphism $f\colon X \to Y$ with a Zariski-closed section $s$ and a geometric $6$-functor formalism $\cal{D}$. Then we use the projective deformation to the normal cone
    \[
    \begin{tikzcd}[row sep = 4em, column sep = 4em]
        \bf{A}^1_{X}  \arrow{d}{f\times \bf{A}^1_S} \arrow[r, hook] & \widetilde{X} \arrow{d}{\widetilde{f}} & \arrow[l, hook'] \rm{V}_Y(\cal{N}_{s}) \arrow{d}{f_0}\\
        \bf{A}^1_{Y} \arrow[bend left]{u}{s\times \bf{A}^1_S}\arrow[r, hook, "j"] & \bf{P}^1_Y \arrow[bend left]{u}{\widetilde{s}} & \arrow[l, hook', "0_Y"'] Y \arrow[bend left]{u}{s_0}
    \end{tikzcd}
    \]
    and the fact that, for an invertible object $C(\widetilde{f}, \widetilde{s})\in \cal{D}(\bf{P}^1_Y)$, the fibers over $1_Y$ and $0_Y$ are isomorphic to conclude that there is a sequence of isomorphisms
    \[
    C(f, s) \simeq C(\widetilde{f}, \widetilde{s})|_{1_Y} \cong C(\widetilde{f}, \widetilde{s})|_{0_Y} \simeq C(f_0, 0_Y)=C_Y(\cal{N}_s) \in \cal{D}(YS). \qedhere
    \]
\end{proof}

\begin{rmk} In practice, the isomorphism $L|_{1_Y}\simeq L|_{0_Y}$ in Definition~\ref{defn:6-functors-geometric}, can be always achieved to be ``canonical''. This would make the isomorphism in Theorem~\ref{thm:formula-dualizing-complex-geometric} also canonical. In particular, this should apply to the potential crystalline or prismatic $6$-functor formalisms. However, it seems annoying to explicitly spell out what this "canonicity" should mean in an abstract $6$-functor formalism, so we do not discuss it here. 
\end{rmk}

\newpage 
\section{First Chern classes}\label{section:Chern-classes}

We note that Theorem~\ref{thm:cohomologically-smooth} and Theorem~\ref{thm:formula-dualizing-complex} (or Theorem~\ref{thm:formula-dualizing-complex-geometric}) together already imply a big part of Poincar\'e Duality. More precisely, Theorem~\ref{thm:cohomologically-smooth} gives a minimalistic way to check that all smooth morphisms are cohomologically smooth with respect to a $6$-functor formalism $\cal{D}$, and Theorem~\ref{thm:formula-dualizing-complex} gives a ``formula'' for the dualizing object $\omega_f=f^! \bf{1}_Y$.\smallskip

However, in many cases, the dualizing object has a particularly nice description as the tensor power of the ``Tate object'' (e.g. relative reduced cohomology of the projective line). This description is not automatic and does not happen for all (geometric) $6$-functor formalisms (e.g. this is false for the (solid) quasi-coherent $6$-functors). Therefore, this further trivialization requires some new argument. \smallskip

In this section, we give different conditions that imply that a $6$-functor formalism $\cal{D}$ automatically satisfies the strongest possible version of Poincar\'e Duality. The strategy is to use Chern classes to both construct the trace map for the relative projective line, and trivialize the dualizing object. \smallskip

We get essentially the optimal result if $\cal{D}$ satisfies the excision axiom (see Definition~\ref{defn:excision-axiom}); in this case, the existence of a theory of first Chern classes (see Definition~\ref{defn:geometric-chern-classes-without-length}) implies Poincar\'e Duality. After unravelling the definition, a theory of first Chern classes essentially boils down to a sufficiently functorial additive assignment of a first Chern class $c_1(\cal{L})$ to a line bundle $\cal{L}$ with the constraint that it satisfies the projective bundle formula for the relative projective line. \smallskip

For a general $6$-functor formalism (not necessarily satisfying the excision axiom), the results are slightly weaker, and we require additional assumptions on $\cal{D}$ to obtain Poincar\'e Duality. Namely, we assume that $\cal{D}$ is either weakly $\bf{A}^1$-invariant or pre-geometric (see Definition~\ref{defn:6-functors-geometric}), and that there exists a theory of cycle maps underlying a {\it strong} theory of first Chern classes (see Definition~\ref{defn:geometric-chern-classes-without-length}). While the results are not as strong as in the case of $6$-functor formalisms with the excision axiom, these conditions appear verifiable in practice. \smallskip


For the rest of the section, we fix a locally noetherian analytic adic space $S$ (resp. a scheme $S$). We denote by $\cal{C}$ the category of locally finite type (resp. locally finitely presented) adic $S$-spaces (resp. $S$-schemes), and fix a $6$-functor formalism $\cal{D}\colon \rm{Corr}(\cal{C})\to \Cat_\infty$. We also fix an inverible object $\bf{1}_S\langle 1\rangle \in \cal{D}(S)$. \smallskip

\subsection{Notation}

In this section, we fix some notation that we will freely use later. We recall that we fixed an invertible object $\bf{1}_S\langle 1\rangle \in \cal{D}(S)$ for the rest of this section. \smallskip

\begin{notation}
    \begin{enumerate}
    \item\label{notation:tate-twists-base}(Tate objects) For a non-negative integer $d\geq 0$, we define {\it Tate objects} 
    \[
    \bf{1}_S\langle d\rangle \coloneqq \bf{1}_S\langle 1\rangle^{\otimes d} \in \cal{D}(S).
    \]
    Using that $\bf{1}_S\langle 1\rangle$ is invertible, we extend the above formula to negative integers $d$ by the following formula: 
    \[
    \bf{1}_S\langle d\rangle \coloneqq (\bf{1}_S\langle -d\rangle)^{\vee} \in \cal{D}(S).
    \]
    \item\label{notation:tate-twists}(Tate twists) In general, for a morphism $f\colon X \to S$, an object $\F\in \cal{D}(X)$, and an integer $d$, we define {\it its Tate twist}
    \[
    \F\langle d\rangle \coloneqq \F\otimes f^* \bf{1}_S\langle d \rangle \in \cal{D}(X).
    \]
    In particular, the object $\bf{1}_X\langle d\rangle\in \cal{D}(X)$ is defined to be $f^*\bf{1}_S\langle d \rangle$. 
    \end{enumerate}
\end{notation}

\subsection{Theory of first Chern classes}\label{section:first-chern-classes}

The main goal of this section is to define the notion of a theory of first Chern classes and verify some of its formal properties. \smallskip

We start the section by giving a precise definition of a theory of first Chern classes. This will be convenient to do in the $\infty$-categorical setting to automatically keep track of all higher coherences. One nice feature of this definition, is that it allows us to define localized first Chern classes for free, while in the $1$-categorical approach, it seems to be extra data. \smallskip

Recall that we have fixed a $6$-functor formalism
\[
\cal{D} \colon \Corr(\cal{C}) \to \Cat_\infty
\]
with an invertible object $\bf{1}_S\langle 1\rangle \in \cal{D}(S)$. \smallskip

\begin{notation} We write $\cal{C}_{\rm{an}}$ for the site whose underlying category is the category $\cal{C}$ and whose coverings are analytic open coverings (resp. Zariski open coverings). 
\end{notation}

We consider sheaf of abelian group 
\[
\O^\times\colon \cal{C}_{\rm{an}}^{\rm{op}} \to \bf{Mod}_\Z
\]
defined by $X\mapsto \O_X^\times(X)$. We can compose it with the natural morphism $\mathbf{Mod}_\Z \to \cal{D}(\Z)$, to get the $\infty$-functor $ \O^{\times} \colon \cal{C}_{\rm{an}}^{\rm{op}} \to \cal{D}(\Z)$. This functor is {\it not} a $\cal{D}(\Z)$-valued sheaf (in the sense \cite[Def.\,1.3.1.1]{SAG}). 

\begin{notation} The sheafification of the $\cal{D}(\Z)$-valued functor $\O^\times$ is the functor
\[
\rm{R}\Gamma_{\rm{an}}(-, \O^\times) \colon \cal{C}_{\rm{an}}^{\rm{op}} \to \cal{D}(\Z).
\]
By \cite[L.\,3, Cor.\,11]{Dustin-lectures}, the values of this functor on an object $X\in \cal{C}$ are canonically identified with $\rm{R}\Gamma_{\rm{an}}(X, \O^\times_X)$ justifying the name. In what follows, we will usually consider the functor $\rm{R}\Gamma_{\rm{an}}(-, \O^\times)$ as an $\rm{Sp}$-valued sheaf by compositing with the natural ``forgetful'' functor $\cal{D}(\Z) \to \cal{D}(\rm{Sp})$. 
\end{notation}

\begin{notation}\label{notation:absolute-cohomology}  We also consider absolute cohomology as an $\rm{Sp}$-valued functor 
\[
\rm{R}\Gamma(-, \bf{1}\langle c\rangle) \colon \cal{C}^{\rm{op}}_{\rm{an}} \to \rm{Sp}
\]
that sends an object $X\in \cal{C}$ to $\rm{R}\Gamma(X, \bf{1}_X\langle c\rangle)=\rm{Hom}_{X}(\bf{1}_X, \bf{1}_X\langle c\rangle)$. One easily checks that it is a $\rm{Sp}$-valued sheaf due to the fact that  $\cal{D}$ satisfies analytic descent.
\end{notation}

\begin{defn}\label{defn:chern-classes} A {\it weak theory of first Chern classes} on a $6$-functor formalism $\cal{D}$ is a morphism 
\[
c_1 \colon \rm{R}\Gamma_{\rm{an}}(-, \O^\times)[1] \to \rm{R}\Gamma(-, \bf{1}\langle 1\rangle)
\]
of $\rm{Sp}$-valued sheaves on $\cal{C}_{\rm{an}}$.
\end{defn}

This definition may seem a bit random at first. However, it does have a strong connection to is classically called a theory of (additive) first Chern classes. We will see in a moment that this definition, in particular, assigns a cohomology class to each line bundle. Furthermore, this assignment is sufficiently functorial so, in the presence of the excision axiom, it even allows us to assign ``localized'' classes to a line bundle with a trivialization. It also encodes functoriality and additivity of this classes. \smallskip


In the following remark, we partially unravel the content of Definition~\ref{defn:chern-classes}. 

\begin{rmk}\label{rmk:properties-first-chern-theory-unlocalized}
\begin{enumerate}
    \item\label{rmk:properties-first-chern-theory-unlocalized-classes}(First Chern classes) By passing to $\rm{H}^0$, a weak theory of first Chern classes gives a group homomorphism
    \[
    \rm{H}^1_{\rm{an}}(X, \O^\times_X) \to \rm{H}^0(X, \bf{1}_X\langle 1\rangle).
    \]
    Recall that the group $\rm{H}^1_{\rm{an}}(X, \O^\times_X)$ classifies the isomorphism classes of line bundles on $X$, so, for each isomorphism class of line bundles $\cal{L}$, a weak theory of first Chern classes assigns {\it the first Chern class of $\cal{L}$} as an element 
    \[
    c_1(\cal{L})\in \rm{H}^0(X, \bf{1}\langle1\rangle)= \rm{Hom}_{D(X)}(\bf{1}_X, \bf{1}_X\langle 1 \rangle).
    \]
    For our purposes, it will be convenient to also consider this class as a homotopy class of morphisms
    \[
    c_1(\cal{L}) \colon \bf{1}_X \to \bf{1}_{X}\langle 1\rangle.
    \]
     \item\label{rmk:properties-first-chern-theory-classes-unlocalized-additivity}(Additivity) Since $c_1$ is a map of spectra, we see that first Chern classes are additive. If $\cal{L}$ and $\cal{L}'$ two isomorphism classes of line bundles on $X$, then 
    \[
    c_1(\cal{L}) + c_1(\cal{L}') = c_1(\cal{L}\otimes \cal{L}').
    \]
    \item\label{rmk:properties-first-chern-theory-unlocalizaed-base-change}(Base Change) The formation of $c_1(\cal{L})$ commutes with arbitrary base due to functoriality of $c_1$. More precisely, if $Y\to X$ is a morphism in $\cal{C}$. Then we have an equality of classes
    \[
    f^*\big(c_1(\cal{L})\big) = c_1\big(f^*\cal{L}\big) \in \rm{Hom}_{D(Y)}(\bf{1}_{Z'}, \bf{1}_{Y}\langle 1\rangle). 
    \]
    \end{enumerate}
\end{rmk}

Now we show that if $\cal{D}$ satisfies the excision axiom (see Definition~\ref{defn:excision-axiom}), then one can also define the localized version of the usual first Chern classes:

\begin{rmk}\label{rmk:properties-first-chern-theory}
\begin{enumerate}
    \item\label{rmk:properties-first-chern-theory-localized-classes}(Localized first Chern classes) More generally, let $Z\xhookrightarrow{i} X$ be a Zariski-closed subset with the complement $U$. Then the group
    \[
    \rm{H}^0\bigg(\rm{fib} \left(\rm{R}\Gamma_{\rm{an}}(X, \O_X^\times) \to \rm{R}\Gamma_{\rm{an}}(U, \O_U^\times)\right)[1]\bigg)=\rm{H}^1_Z(X, \O_X^\times)
    \]
    classifies\footnote{Even though this fact is well-known, it does not seem to be explicitly formulated in the literature. The interested reader may adapt the argument used in \cite[2.13]{Olsson} to this situation.} isomorphism classes of pairs $(\cal{L}, \phi_U)$ of a line bundle $\cal{L}$ and a trivialization $\phi\colon \O_U \to \cal{L}|_{U}$ on $U$. Therefore, for any such isomorphism class, a weak theory of first Chern classes assigns the {\it localized Chern class of $(\cal{L}, \varphi_U)$} as an element\footnote{Use the excision sequence from Remark~\ref{rmk:excision-sequence} for the second isomorphism below.}
    \begin{align*}
    c_1(\cal{L}, \varphi_U) & \in \rm{H}^0\bigg(\rm{fib} \left(\rm{R}\Gamma(X, \bf{1}_X\langle 1\rangle) \to \rm{R}\Gamma(U, \bf{1}_U\langle 1\rangle)\right)\bigg) \\ 
    &\simeq \rm{H}^0_Z(X, \bf{1}_X\langle1\rangle) = \rm{Hom}_{D(X)}(i_*\bf{1}_Z \to \bf{1}_{X}\langle 1\rangle).
    \end{align*}
    Again, for our purposes, it will also be convenient to think about the localized first Chern class as of a homotopy class of morphisms
    \[
    c_1(\cal{L}, \varphi_U) \colon i_*\bf{1}_Z \to \bf{1}_{X}\langle 1\rangle.
    \]
    Non-localized first Chern classes can be recovered from this construction by taking $Z=X$.
    \item\label{rmk:properties-first-chern-theory-additivity}(Additivity) Since $c_1$ is a map of spectra, we see that localized first Chern classes are additive. If $(\cal{L}, \varphi_U)$ and $(\cal{L}', \varphi'_U)$ two isomorphism classes of line bundles with a trivialization on $U$, then 
    \[
    c_1(\cal{L}, \varphi_U) + c_1(\cal{L}', \varphi'_U) = c_1(\cal{L}\otimes \cal{L}', \varphi_U\otimes \varphi'_U).
    \]
    \item\label{rmk:properties-first-chern-theory-base-change}(Base Change) The formation of $c_1(\cal{L}, \varphi_U)$ commutes with arbitrary base due to functoriality of $c_1$. More precisely, if 
    \[
    \begin{tikzcd}
    Z' \arrow{d}{i'} \arrow{r}{f'} & Z \arrow{d}{i} \\
    Y \arrow{r}{f} & X
    \end{tikzcd}
    \]
    is a cartesian diagram in $\cal{C}$. Then we have an equality of classes
    \[
    f^*\big(c_1(\cal{L}, \varphi_U)\big) = c_1\big(f^*\cal{L}, f^*(\varphi_U)\big) \in \rm{Hom}_{D(Y)}(i'_*\bf{1}_{Z'}, \bf{1}_{Y}\langle 1\rangle). 
    \]
    In other words, the diagram
    \[
    \begin{tikzcd}[column sep = 5em]
        f^*i_* \bf{1}_Z \arrow{d}{\wr} \arrow{r}{f^*\bigl(c_1(\cal{L}, \varphi_U)\bigr)} & f^*(\bf{1}_{X}\langle 1\rangle) \arrow{d}{\wr} \\
        i'_*\bf{1}_{Z'} \arrow{r}{c_1\big(f^*\cal{L}, f^*\varphi_U\big)} & \bf{1}_Y\langle 1\rangle
    \end{tikzcd}
    \]
    commutes (up to homotopy), where the left vertical map is the base-change morphism. 
    \item\label{rmk:properties-first-chern-theory-localization}(Localization) Now we discuss another instance of functoriality of $c_1$. Let $i_1$ and $i_2$
    \[
    \begin{tikzcd} Z_1 \arrow[rr, bend left, "i_1"] \arrow{r} & Z_2\arrow{r}{i_2} & X\end{tikzcd}
    \]
    be Zariski-closed immersions with open complements $U_1$ and $U_2$ respectively, and $(\cal{L}, \varphi_{U_1})$ a pair of a line bundle and its trivialization on $U_1$. Then the diagram
    \[
    \begin{tikzcd}[row sep = 3em, column sep =3em]
        i_{2, *} \bf{1}_{Z_2} \arrow[swap]{rd}{c_1(\cal{L}, \varphi_{U_1}|_{U_2})} \arrow{r} & i_{1, *}\bf{1}_{Z_1} \arrow{d}{c_1(\cal{L}, \varphi_{U_1})} \\
        & \bf{1}_{X}\langle 1\rangle 
    \end{tikzcd}
    \]
    commutes (up to homotopy).
\end{enumerate}
\end{rmk}

\begin{construction}\label{construction:adjoint-classes} Suppose that $f\colon X \to Y$ is a morphism is $\cal{C}$ and $c\colon f^*\bf{1}_Y=\bf{1}_X \to \bf{1}_X\langle1\rangle$ is a morphism in $\cal{D}(X)$. By the $(f^*, f_*)$-adjunction, this uniquely defines a morphism
\[
^\rm{adj}c\colon \bf{1}_Y  \to f_*\bf{1}_X\langle 1\rangle.
\]
Unless there is some possible confusion, we will denote the morphism $^\rm{adj}c$ simply by $c$. Applying the same construction to tensor powers of $c$, we get morphisms
\[
c^k \colon \bf{1}_Y \to f_* \bf{1}_X \langle k \rangle.
\]
We note that, for $k=0$, we get simply the adjunction morphism that we denote by 
\[
f^*\colon \bf{1}_Y \to f_*\bf{1}_X.
\]
\end{construction}

Now we apply this construction to the projective bundle $f\colon \bf{P}_Y(\cal{E}) \to Y$ for some vector bundle $\cal{E}$ on $Y$ of rank $d+1$ (see \cite[Def.\,7.11]{adic-notes}) and the first Chern class morhism of the universal line bundle:
\[
c_1=c_1(\O(1))\colon \bf{1}_{\bf{P}_Y(\cal{E})}\to \bf{1}_{\bf{P}_Y(\cal{E})}\langle 1\rangle.
\]
Then Construction~\ref{construction:adjoint-classes} gives us a morphism
\[
\sum_{k=0}^d c_1^k\langle d-k\rangle \colon \bigoplus_{k=0}^d \bf{1}_Y\langle d-k\rangle \to f_* \bf{1}_{\bf{P}_Y(\cal{E})}\langle d\rangle.
\]

\begin{defn}\label{defn:geometric-chern-classes-without-length} A {\it theory of first Chern classes} is a weak theory of first Chern classes $c_1$ such that, for the relative projective line $f\colon \bf{P}^1_S \to S$, the morphism
\[
c_1 + f^*\langle 1\rangle \colon \bf{1}_S \oplus \bf{1}_S\langle 1\rangle \to f_* \bf{1}_{\bf{P}^1_S}\langle 1\rangle.
\]
is an isomorphism. 

A {\it strong theory of first Chern classes} is a weak theory of first Chern classes $c_1$ such that, for any integer $d\geq 1$ and the relative projective space $f\colon \bf{P}^d_S \to S$, the morphism
\[
\sum_{k=0}^d c_1^k\langle d-k\rangle \colon \bigoplus_{k=0}^d \bf{1}_S\langle d-k\rangle \to f_* \bf{1}_{\bf{P}_S^d}\langle d\rangle.
\]
is an isomorphism. 
\end{defn}

\begin{rmk} Definition~\ref{defn:geometric-chern-classes-without-length} implies that, if $c_1$ is a theory of first Chern classes, then
\[
\bf{1}_S\langle -1\rangle \simeq \rm{Cone}\left(\bf{1}_S \to f_* \bf{1}_{\bf{P}^1_S}\right).
\]
So the invertible object $\bf{1}_S\langle 1\rangle$ is unique up to an isomorphism, and axiomitizes the ``Tate twist''. 
\end{rmk}

\begin{lemma}\label{lemma:projective-bundle-formula}(Projective Bundle Formula) Let $c_1$ be a theory of strong first Chern classes, let $Y$ be an element of $\cal{C}$, and let $f\colon \bf{P}_Y(\cal{E}) \to Y$ be the projective bundle associated to a vector bundle $\cal{E}$ of rank $d+1$. Then the morphism
\[
\sum_{k=0}^d c_1^k\langle d-k\rangle \colon \bigoplus_{k=0}^d \bf{1}_Y\langle d-k\rangle \to f_* \bf{1}_{\bf{P}_Y(\cal{E})}\langle d\rangle.
\]
is an isomorphism. If $c_1$ is a theory first Chern classes, the same holds for vector bundles of rank $2$. 
\end{lemma}
\begin{proof}
    Since $\cal{D}$ is an analytic sheaf, we can check that $\sum_{i=0}^d c_1^k\langle d-k\rangle$ is an isomorphism analytically locally on $Y$. Therefore, we may and do assume that $\cal{E}$ is a trivial vector bundle of rank $d$. In this case, the result follows from Definition~\ref{defn:geometric-chern-classes-without-length}, shriek base change (together with the canonical equivalence $f_!\simeq f_*$), and the fact that $c_1(\O(1))$ commutes with base change along $Y\to S$. 
\end{proof}

Now we show that a {\it strong} theory of first Chern classes automatically implies that the braiding morphism
\[
    s\colon \bf{1}_S\langle 1\rangle^{\otimes 2} \to \bf{1}_S\langle 1\rangle^{\otimes 2}
\]
is homotopic to the identity morphism. This will be used later to simplify the second diagram in Definition~\ref{defn:trace-cycle} in the presence of a strong theory of first Chern classes. 

\begin{lemma}\label{lemma:first-chern-classes-orientable} Let $c_1$ be a theory of strong first Chern classes on a $6$-functor formalism $\cal{D}$. Then the braiding morphism
\[
    s\colon \bf{1}_S\langle 1\rangle^{\otimes 2} \to \bf{1}_S\langle 1\rangle^{\otimes 2}
\]
is homotopic to the identity morphism. 
\end{lemma}
\begin{proof}
    Firstly, it suffices to prove the analogous claim for $\bf{1}_S\langle -1\rangle$. The key is that $\bf{1}_S\langle -1\rangle$ can be realized as a direct summand of the ``relative'' cohomology of $\bf{P}^2_S$ .\smallskip
    
    We first fix the relative projective space $f\colon \bf{P}^2_S \to S$. Now we note that $f_*$ is a right-adjoint to a symmetric monoidal functor $f^*$, so it is lax-monoidal. In particular, for every object $\F\in D(\bf{P}^2_S)$ with the braiding morphism $s_{\F}\colon \F^{\otimes 2} \to \F^{\otimes 2}$, we have a commutative diagram
    \begin{equation}\label{eqn:lax-monoidal}
    \begin{tikzcd}
    (f_* \F)^{\otimes 2} \arrow{d}{s_{f_*\F}} \arrow{r}{\cup} & f_*( \F^{\otimes 2}) \arrow{d}{f_*(s_\F)} \\
    (f_* \F)^{\otimes 2} \arrow{r}{\cup} & f_*( \F^{\otimes 2}).
    \end{tikzcd}
    \end{equation}
    in the homotopy category $D(S)$. \smallskip
    
    Now we consider the (twisted) first Chern class morphism $c_1(\cal{O}(1))\langle -1\rangle \colon \bf{1}_{\bf{P}^2_S} \langle -1 \rangle \to \bf{1}_{\bf{P}^2_S}$. Then similarly to Construction~\ref{construction:adjoint-classes}, we get the morphism
    \[
    ^{\rm{adj}}c_1 \colon \bf{1}_S\langle -1\rangle \to f_* \bf{1}_{\bf{P}^2_S}. 
    \]
    The same construction applied to $c_1(\cal{O}(1))\langle-1\rangle^{\otimes 2} \colon \bf{1}_{\bf{P}^2_S} \langle -1 \rangle^{\otimes 2} \to \bf{1}_{\bf{P}^2_S}^{\otimes 2}$ produces the morphism
    \[
    ^{\rm{adj}}c_1^{2} \colon \bf{1}_S\langle -1\rangle^{\otimes 2} \to f_*(\bf{1}_{\bf{P}^2_S}^{\otimes 2}).
    \]
    A formal diagram chase implies that the diagram 
    \[
    \begin{tikzcd}
    \bf{1}_{S}\langle -1\rangle^{\otimes 2}\arrow[swap]{d}{^{\rm{adj}}c_1\otimes ^{\rm{adj}}c_1} \arrow{rd}{^{\rm{adj}}c_1^{2}} & \\
    \bigl(f_* \bf{1}_{\bf{P}^2_S}\bigr)^{\otimes 2} \arrow{r}{\cup} & f_*\bigl(\bf{1}_{\bf{P}^2_S}^{\otimes 2} \bigr)
    \end{tikzcd}
    \]
    commutes in $D(S)$. Definition~\ref{defn:geometric-chern-classes-without-length} (with maps twisted by $\bf{1}\langle -2\rangle$) implies that $^{\rm{adj}}c_1^{2}$ realizes $\bf{1}_S\langle -1\rangle^{\otimes 2}$ as a direct summand of $f_* \bigl(\bf{1}_{\bf{P}^2_S}^{\otimes 2} \bigr)$. Now we consider the commutative diagram 
    \[
    \begin{tikzcd}[column sep = 5em]
    \bf{1}_{S}\langle -1\rangle^{\otimes 2} \arrow{d}{s} \arrow{r}{^{\rm{adj}}c_1\otimes ^{\rm{adj}}c_1} \arrow[bend left]{rr}{^{\rm{adj}}c_1^{2}} & \bigl(f_* \bf{1}_{\bf{P}^2_S}\bigr)^{\otimes 2} \arrow{r}{\cup} \arrow{d}{s_{f_*(\bf{1})}}& f_*\bigl(\bf{1}_{\bf{P}^2_S}^{\otimes 2} \bigr) \arrow{d}{f_*(s_{\bf{1}})} \\
    \bf{1}_{S}\langle -1\rangle^{\otimes 2} \arrow{r}{^{\rm{adj}}c_1\otimes ^{\rm{adj}}c_1} \arrow[bend right]{rr}{^{\rm{adj}}c_1^{2}} & \bigl(f_* \bf{1}_{\bf{P}^2_S}\bigr)^{\otimes 2} \arrow{r}{\cup} & f_*\bigl(\bf{1}_{\bf{P}^2_S}^{\otimes 2} \bigr),
    \end{tikzcd}
    \]
    where $s$ stands for the braiding morphisms. The left square commutes by the definition of a symmetric monoidal category, and the right square commutes due to Diagram~(\ref{eqn:lax-monoidal}). Since $^{\rm{adj}}c_1^{2}$ splits, it suffices to show that $f_*(s_{\bf{1}})$ is equal to $\rm{id}$. But this is clear since the braiding morphism of the unit object is homotopic to the identity morphism.
\end{proof}

In the next couple of sections, we will show how a theory of first Chern classes can be used to prove the full version of Poincar\'e Duality. 

\subsection{Theory of cycle maps}

The main goal of this section is to axiomitize a theory of cycle maps (for divisors) on a $6$-functor formalism $\cal{D}$ ``compatible'' with a weak theory of first Chern classes $c_1$ on $\cal{D}$. Then we show that, if $\cal{D}$ satisfies the excision axiom, one can canonically construct such a theory from any weak theory of first Chern classes. 

\subsubsection{Definitions}

In this subsection, we explain the definition of a theory of cycle maps (for divisors) and what it means for a theory of first Chern classes to underlie a theory of cycle maps.  \smallskip

As previously, we fix an invertible object $\bf{1}_S\langle 1\rangle\in \cal{D}(S)$ and always consider (weak) theories of first Chern Classes with respect to this invertible object. \smallskip

\begin{defn}\label{defn:cartier-divisor} Let $i\colon D\hookrightarrow X$ be an effective Cartier divisor with the associated coherent ideal sheaf $\cal{I}=\rm{ker}(\O_X \to i_*\O_D)\subset \O_X$ (see \cite[Def.\,5.3]{adic-notes}). The {\it associated line bundle} $\O_X(D) \coloneqq \cal{I}^{\vee}$ is the dual of $\cal{I}$, we denote its dual by $\O_X(-D)$ (that is simply just a different name for $\cal{I}$).
\end{defn}

\begin{defn}\label{defn:theory-of-cycle-classes} A {\it theory of cycles maps (for effective Cartier divisors)} $\cl_{\bullet}$ on a $6$-functor formalism $\cal{D}\colon \Corr(\cal{C})\to \Cat_\infty$ is a collection of morphisms 
\[
\rm{cl}_i\colon i_* \bf{1}_Y \to \bf{1}_X\langle 1\rangle \text{ in the homotopy category } D(X)
\]
for each effective Cartier divisor $i\colon Y \to X$ such that they satisfy transversal base change, i.e., for any cartesian diagram 
    \[
    \begin{tikzcd}
    Y' \arrow{r}{g'} \arrow[d, hook, "i'"] & Y \arrow[d, hook, "i"] \\
    X' \arrow{r}{g} & X
    \end{tikzcd}
    \]
    such that the vertical arrows are effective Cartier divisors, the diagram \[
    \begin{tikzcd}
    g^* i_* \bf{1}_Y \arrow{d}{\wr} \arrow{r}{g^*(\rm{cl}_i)} & g^*(\bf{1}_X\langle 1\rangle) \arrow{d}{\wr}\\
    i'_* \bf{1}_{Y'} \arrow{r}{\rm{cl}_{i'}}& \bf{1}_{X'}\langle 1\rangle 
    \end{tikzcd}
    \]
    commutes in $D(X)$.
\end{defn}

\begin{defn}\label{defn:theory-of-cycle-classes-chern-classes} A weak theory of first Chern classes $c_1$ {\it underlies a theory of cycle maps $\cl_{\bullet}$} if, for every effective Cartier divisor $i\colon Y\to X$, the composition
\[
\bf{1}_X \to i_*\bf{1}_Y \xr{\cl_i} \bf{1}_X\langle 1\rangle
\]
is equal to $c_1(\O_X(Y))$ in the homotopy category $D(X)$.
\end{defn}

For the next remark, we fix a weak theory of first Chern classes $c_1$ underlying a theory of cycle maps $\cl_\bullet$. 

\begin{rmk}\label{rmk:first-chern-first-cycle-comp} Let $f\colon X \to Y$ be a morphism in $\cal{C}$ and let $i\colon D \hookrightarrow X$ be an effective Cartier divisor. We can apply Construction~\ref{construction:adjoint-classes} to the composition morphism
\[
\begin{tikzcd}[column sep =5em]
\bf{1}_X \arrow[bend left]{rr}{c_1(\cal{O}_X(D))} \arrow{r} & i_* \bf{1}_D  \arrow{r}{\rm{cl}_i}& \bf{1}_X\langle 1\rangle
\end{tikzcd}
\]
to get the morphism $c\colon \bf{1}_Y \to f_*\bf{1}_{X}\langle 1\rangle$. Then $c$ has an alternative description as the composition
\[
\bf{1}_Y \xr{} f_*i_* \bf{1}_D \xr{f_*(\rm{cl}_i)} f_*(\bf{1}_X)\langle 1\rangle.
\]
\end{rmk}

\subsubsection{Constructing cycle maps}

The main goal of this subsection is to show that, if $\cal{D}$ satisfies the excision axiom, then any weak theory of first Chern classes $c_1$ canonically underlies a theory of cycle maps. 

\begin{warning} We do not know a way to extract a theory of cycle maps from a weak theory of first Chern classes without the excision axiom. In practice, however, all $6$-functor formalisms with a strong theory of first Chern classes admit a compatible theory of cycle maps. This suggests that a weaker assumption on $\mathcal{D}$ might suffice to canonically construct cycle maps from first Chern classes.

However, since essentially all $6$-functor formalisms with a strong theory of first Chern classes admit compatible cycle maps, it may be possible to canonically construct these maps using a weaker assumption on $\mathcal{D}$.
\end{warning}

For the rest of this section, we fix a $6$-functor formalism $\cal{D}$ satisfying the excision axiom and a weak theory of first Chern classes $c_1$. \smallskip

To construct cycle maps, we note that an effective Cartier divisor $D$ comes with the canonical short exact sequence (see Definition~\ref{defn:cartier-divisor}):
\[
0 \to \O_X(-D) \to \O_X \to i_*\O_D \to 0.
\]
By passing to duals, we get a morphism $\O_X \to \O_X(D)$ that is an isomorphism over $U\coloneqq X\setminus D$. We denote its restriction on $U$ by an isomorphism
\[
\varphi_U\colon \O_U \xr{\simeq} \O_X(D)|_U.
\]

Now, in the presence of the excision axiom, we can give the following definition: 

\begin{defn}\label{defn:cycle-clas-divisors} A {\it cycle map (relative to $c_1$) of an effective divisor $D \subset X$} is a homotopy class of morphisms 
\[
\rm{cl}_i \colon i_*\bf{1}_D \to \bf{1}_X\langle1\rangle
\]
equal to $c_1(\cal{O}_X(D), \varphi_U) \in \rm{H}^0_D(X, \bf{1}_X\langle 1\rangle)=\rm{Hom}_{D(X)}(i_*\bf{1}_D, \bf{1}_X\langle 1\rangle)$ (see Remark~\ref{rmk:properties-first-chern-theory}(\ref{rmk:properties-first-chern-theory-localized-classes}). 
\end{defn}

\begin{lemma}\label{lemma:canonical-cycle-classes} Let $\cal{D}$ be a $6$-functor formalism satisfying the excision axiom, and $c_1$ is a weak theory of first Chern classes on $\cal{D}$. Then the construction of cycle maps $\cl_\bullet$ from Definition~\ref{defn:cycle-clas-divisors} defines a theory of cycle maps (see Definition~\ref{defn:theory-of-cycle-classes}) such that $c_1$ underlies $\cl_\bullet$ (see Definition~\ref{defn:theory-of-cycle-classes-chern-classes}).
\end{lemma}
\begin{proof}
    We need to check two things: cycle maps commute with transversal base change and, for each effective Cartier divisor $i\colon Y \hookrightarrow X$, the composition
    \[
    \bf{1}_X \to i_*\bf{1}_Y \xr{\cl_i} \bf{1}_X\langle 1\rangle
    \]
    is equal to $c_1(\O_X(Y))$. \smallskip
    
    The first claim is automatic from Remark~\ref{rmk:properties-first-chern-theory}(\ref{rmk:properties-first-chern-theory-base-change})  and \cite[Lemma 5.8]{adic-notes}. The second claim is automatic from Remark~\ref{rmk:properties-first-chern-theory}(\ref{rmk:properties-first-chern-theory-localization}) by taking $Z_1=Y$ and $Z_2=X$. 
\end{proof}

\subsection{Cycle map of a point}

In this section, we construct the (naive) cycle map of the (``zero'') section on the relative projective space $f_d\colon \bf{P}^d_S \to S$. We do not develop a robust theory of cycle maps for all lci closed immersions of higher co-dimension, instead we give an ad hoc construction in this particular case. The theory of higher dimensional cycle classes can be developed if $\cal{D}$ satisfies the excision axiom (following the strategy of defining cycle classes in \'etale cohomology developed in \cite{Fujiwara-purity}), but we are not aware of a way of doing this for a general $\cal{D}$ so we do not discuss it in this paper. The ad hoc construction mentioned above is enough for all purposes of this paper. \smallskip

Before we go into details, we point out that this construction will be used both in establishing Poincar\'e Duality for weakly $\bf{A}^1$-invariant or pre-geometric (see Definition~\ref{defn:homotopy-invariant} and Definition~\ref{defn:6-functors-geometric}) $6$-functor formalisms with a strong theory of first Chern classes underlying a theory of cycle maps, {\it and} in proving that a theory of first Chern classes is automatically a strong theory of first Chern classes if $\cal{D}$ satisfies the excision axiom. \smallskip

For the rest of this section, we fix a $6$-functor formalism $\cal{D}$ with a theory of weak first Chern classes $c_1$ underlying a theory of cycle maps $\cl_\bullet$ (see Definition~\ref{defn:theory-of-cycle-classes-chern-classes}). \smallskip



We fix a relative projective space $f_d\colon \bf{P}^d_Y \to Y$ with homogenenous coordinates $X_1, \dots, X_{d+1}$ and a set of $d+1$-standard $Y$-hyperplanes
\[
H_1, \dots, H_d, H_{d+1}\subset \bf{P}^d_Y
\]
given as the vanishing locus of the homogeneous coordinate $X_i$ respectively. We note that the intersection $H_1\cap H_2\cap \dots H_d$ is canonically isomorphic to $Y$ and the natural embedding
\[
s\colon H_1\cap H_2\cap \dots H_d = Y \hookrightarrow \bf{P}^d_Y
\]
defines the ``zero'' section of $\bf{P}^d_Y$. We also denote by $i_k\colon H_k \hookrightarrow \bf{P}^d_Y$ the natural immersion of $H_k$ into $\bf{P}^d_Y$ for $k=1, \dots, d+1$, and by $s'\colon Y \hookrightarrow H_d$ the closed immersion of $H_1\cap H_2\cap \dots H_d$ into $H_d$. In particular, we have the following commutative diagram:
\[
\begin{tikzcd}
Y \arrow[rr, bend left, hook, "s"] \arrow[r, hook, "s'"] & H_d \arrow[r, hook, "i_d"] &\bf{P}^d_Y.
\end{tikzcd}
\]


\begin{defn}\label{defn:naive-cycle-class}(Naive Cycle map of the (``zero'') section) We define the {\it naive cycle map of $s$ (relative to $c_1$, $\cl_\bullet$)} to be the homotopy class of morphisms $\cl_s \colon s_*\bf{1}_Y \to \bf{1}_{\bf{P}^d_Y}\langle d\rangle$ inductively obtained by the following rule:
\begin{enumerate}
    \item if $d=1$, $s$ is an effective Cartier divisors, so $\cl_s$ is the cycle map of the corresponding effective Cartier divisor;
    \item if $d>1$, we suppose that we defined $\cl_s$ for all $d'< d$ (so, in particular, it is defined for $s'$), and define $\cl_i$ as the composition
    \[
    s_*\bf{1}_Y \simeq i_{d, *} s'_* \bf{1}_Y \xr{i_{d, *}(\cl_{s'})} i_{d, *} \bf{1}_{H_d}\langle d-1\rangle \xr{\rm{id}_{\bf{1}\langle d-1\rangle} \otimes \cl_{i_d}} \bf{1}_{\bf{P}^d_Y}\langle d\rangle, 
    \]
    where $\cl_{s'}$ is defined due to the induction hypothesis and $\cl_{i_d}$ is the cycle map of an effective Cartier divisor. 
\end{enumerate}
\end{defn}

\begin{warning} The definition~\ref{defn:naive-cycle-class}, a priori, depends on the choice of coordinates on $\bf{P}^d_Y$. In particular, it is not clear that the cycle map $\cl_i$ does not change if we permute coordinates on $\bf{P}^d_Y$. 
\end{warning}

\begin{lemma}\label{lemma:naive-cycle-of-point-product-of-chern-classes} Let $c_1$ be a weak theory of first Chern classes on $\cal{D}$ underlying a theory of cycle maps $\cl_\bullet$, let $f_d\colon \bf{P}^d_Y \to Y$ be the relative projective space, and let
\[
\cl_s\colon s_*\bf{1}_Y \to \bf{1}_{\bf{P}^d_Y}\langle d\rangle
\]
be the naive cycle map from Definition~\ref{defn:naive-cycle-class}. Then the diagram
\[
    \begin{tikzcd}
    \bf{1}_{\bf{P}^d_Y} \arrow[bend left]{rr}{c_1(\O_{\bf{P}^d_Y/Y}(1))^{\otimes d}}\arrow{r}{\rm{adj}_s} & s_{*} \bf{1}_Y \arrow{r}{\cl_{s}}  &  \bf{1}_{\bf{P}^d_Y}\langle d\rangle.
    \end{tikzcd}
\]
commutes in $D(\bf{P}^d_Y)$, where $\rm{adj}_s$ is the canonical morphism coming from the $(s^*, s_*)$-adjunction.
\end{lemma}
\begin{proof}
    We argue by induction. If $d=1$, the claim follows directly from Definition~\ref{defn:theory-of-cycle-classes-chern-classes}. \smallskip
    
    Now we suppose the claim is know for all $d'<d$ and wish to show it for $d$. Note that, in particular, the induction hypothesis applies to the morphism $s'\colon Y \hookrightarrow H_d\simeq \bf{P}^{d-1}_Y$. In particular, we conclude that the diagram
    \[
    \begin{tikzcd}[column sep =4em]
    i_{d, *}\bf{1}_{H_d} \arrow[bend left]{rr}{i_{d, *}\bigl(c_1\left(\O_{H_d/Y}(1)\right)^{\otimes d-1}\bigr)}\arrow{r}{i_{d,*}(\rm{adj}_{s'})} & i_{d, *} s'_* \bf{1}_Y \arrow{r}{i_{d, *}(\cl_{s'})}  & i_{d, *} \bf{1}_{H_d}\langle d-1\rangle.
    \end{tikzcd}
    \]
    commutes in $D(\bf{P}^d_Y)$. Now note that $\O_{H_d/Y}(1) \simeq i_d^* \O_{\bf{P}^d_Y/Y}(1)$ to conclude that the following diagram commutes in $D(\bf{P}^d_Y)$:
    \begin{equation}\label{eqn:induction}
    \begin{tikzcd}[column sep = 5em, row sep = 3em]
    \bf{1}_{\bf{P}^d_Y} \arrow{dd}{\rm{adj}_{i_d}} \arrow{rd}{\rm{adj}_s}\arrow{rr}{c_1\bigl(\O_{\bf{P}^d_Y/Y}(1)\bigr)^{\otimes d-1}} & & \bf{1}_{\bf{P}^d_Y}\langle d-1\rangle  \arrow{dd}{\rm{adj}_{i_d}}\\
    & s_*\bf{1}_Y\arrow{d}{\wr} & \\
    i_{d, *}\bf{1}_{H_d} \arrow[bend right, rr, "i_{d, *}\bigg(c_1(\O_{H_d/Y}(1))^{\otimes d-1}\bigg)"]\arrow{r}{i_{d,*}(\rm{adj}_{s'})} & i_{d, *} s'_* \bf{1}_Y \arrow{r}{i_{d, *}(\cl_{s'})}  & i_{d, *} \bf{1}_{H_d}\langle d-1\rangle.
    \end{tikzcd}
    \end{equation}
    By definition of a (weak) theory of first Chern classes underlying a theory of cycle maps (see Definition~\ref{defn:theory-of-cycle-classes-chern-classes}), we also get a commutative diagram
    \begin{equation}\label{eqn:divisor}
    \begin{tikzcd}[row sep = 4em, column sep = 5em]
        \bf{1}_{\bf{P}^d_Y}\langle d-1\rangle \arrow{r}{\rm{adj}_{i_d}} \arrow[bend left]{rr}{\rm{id}_{\bf{1}\langle d-1\rangle} \otimes c_1\bigl(\O_{\bf{P}^d_Y/Y}(1)\bigr)} & i_{d, *} \bf{1}_{H_d}\langle d-1\rangle \arrow{r}{\rm{id}_{\bf{1}\langle d-1\rangle} \otimes \cl_{i_d}} & \bf{1}_{\bf{P}^d_Y}\langle d\rangle 
    \end{tikzcd}
    \end{equation}
    Therefore, we may combine Diagram~(\ref{eqn:induction}) and Diagram~(\ref{eqn:divisor}) to conclude that the composition
    \[
    \bf{1}_{\bf{P}^d_Y}\xr{\rm{adj}_s} s_{*} \bf{1}_Y \xr{\cl_{s}}  \bf{1}_{\bf{P}^d_Y}\langle d\rangle.
    \]
    is equal (in the homotopy category $D(\bf{P}^d_Y)$) to the following composition:
    \[
    \bf{1}_{\bf{P}^d_Y}  \xr{c_1\bigl(\O_{\bf{P}^d_Y/Y}(1)\bigr)^{\otimes d-1}} \bf{1}_{\bf{P}^d_Y}\langle d-1\rangle \xr{\rm{id}_{\bf{1}\langle d-1\rangle} \otimes c_1(\O_{\bf{P}^d_Y/Y}(1))} \bf{1}_{\bf{P}^d_Y}\langle d\rangle
    \]
    that is just equal to $c_1(\O_{\bf{P}^d_Y/Y}(1))^{\otimes d}$. This finishes the proof. 
\end{proof}

\subsection{First Chern classes and excision}\label{section:first-chern-classes-excision}

The main goal of this section is to show that, if $\cal{D}$ satisfies the excision axiom, then any theory of first Chern classes on $\cal{D}$ is automatically a strong theory of first Chern classes (see Definition~\ref{defn:geometric-chern-classes-without-length}). More precisely, we have to show that the projective bundle formula for the $\bf{P}^1_S$ implies the projective bundle formula for all higher dimensional relative projective spaces in the presence of the excision axiom. We show this by induction on $d$ cutting $\bf{P}^d_S$ into a closed subspace $\bf{P}^{d-1}_S$ and an open complement $\bf{A}^d_S$. To deal with the open complement, we use the naive cycle map of the zero section from Definition~\ref{defn:naive-cycle-class}. \smallskip

For the rest of the section, we fix a $6$-functor formalism $\cal{D}$ {\it satisfying the excision axiom}, and a theory of first Chern classes $c_1$. We also fix an object $Y\in \cal{C}$. \smallskip

\begin{setup}\label{setup:an-pn} We denote by  $0_Y\colon Y \to \bf{A}^{d}_Y$ the zero section. This fits into the following commutative diagram:
\[
\begin{tikzcd}[row sep = 4em]
Y \arrow[r, hook, "0_Y"] \arrow[rr, bend left, "s"] \arrow{dr}{\rm{id}} & \bf{A}^{d}_Y \arrow{d}{g} \arrow[r, hook, "j"] & \bf{P}^{d}_Y\arrow{dl}{f_d} & \bf{P}^{d-1}_Y\simeq H_{d+1} \arrow{dll}{f_{d-1}} \arrow[l, swap, hook', "i_{d+1}"] \\
& Y, &
\end{tikzcd}
\]
where $f_d, f_{d+1}$, and $g$ are the structure morphisms, $j$ is the natural open immersion, and $s$ and $i_{d+1}$ are respectively the ``zero'' section and the $X_{d+1}$-hyperplane section from the discussion above Definition~\ref{defn:naive-cycle-class}.
\end{setup}

\begin{defn}\label{defn:naive-cycle-class-affine}(Naive cycle map of the zero section) We define the {\it naive cycle map of $0_Y$} to be the homotopy class of morphisms 
\[
\cl_{0_Y} \colon 0_{Y, *}\bf{1}_Y \to \bf{1}_{\bf{A}^d_Y}\langle d\rangle
\]
equal to $j^*(\cl_{s})$, where $\cl_s$ is from Definition~\ref{defn:naive-cycle-class}. More precisely, $\cl_{0_Y}$ is obtained as the composition
\[
0_{Y, *}\bf{1}_Y \simeq j^*s_{*} \bf{1}_Y \xr{j^*(\cl_s)} j^*\bf{1}_{\bf{P}^d_Y}\langle d\rangle \simeq \bf{1}_{\bf{A}^d_Y}\langle d\rangle. 
\]
\end{defn}

\begin{rmk}\label{rmk:inductive-definition} Alternatively, one can repeat Definition~\ref{defn:naive-cycle-class} in the affine case, and define $\cl_{0_Y}$ to be the composition of $d-1$ cycle maps of divisors. 
\end{rmk}

\begin{lemma}\label{lemma:cycle-point-vs-chern-classes} Following the notion from Setup~\ref{setup:an-pn}, let $c_1^{d}\colon \bf{1}_Y \to  f_{d, *} \bf{1}_{\bf{P}^{d}_Y}\langle d\rangle$ be the morphism obtained by applying Construction~\ref{construction:adjoint-classes} to $c_1(\O_{\bf{P}^d_Y/Y}(1))^{\otimes d}$. Then the diagram
\[
\begin{tikzcd}
\bf{1}_Y \arrow{r}{g_!\left(\rm{cl}_{0_{Y}}\right)}\arrow[swap]{rd}{c_1^{d}}  &g_! \bf{1}_{\bf{A}_Y^{d}}\langle d\rangle  \arrow{d}{\rm{can}}\\
 & f_{d, *} \bf{1}_{\bf{P}^{d}_Y}\langle d\rangle 
\end{tikzcd}
\]
commutes in (the homotopy category) $D(Y)$.
\end{lemma}
\begin{proof}
    Essentially by construction, we have the following commutative diagram
    \[
    \begin{tikzcd}
    \bf{1}_Y \arrow{r}{g_!(\rm{cl}_{0_{Y}})} \arrow[swap]{rd}{f_{d, *}\left(\rm{cl}_s\right)} & g_! \bf{1}_{\bf{A}^{d}_Y}\langle d\rangle \arrow{d}{\rm{can}}  \\
     & f_{d, *} \bf{1}_{\bf{P}^{d}_Y}\langle d\rangle.
    \end{tikzcd}
    \]
    Thus, we are only left to identify $f_{d, *}(\rm{cl}_s)$ with $c_1^{d}$. This follows from Remark~\ref{rmk:first-chern-first-cycle-comp} and Lemma~\ref{lemma:naive-cycle-of-point-product-of-chern-classes}.
\end{proof}

\begin{lemma}\label{lemma:short-exact-sequence} Suppose $\cal{D}$ satisfies the excision axiom and that $c_1$ is a theory of first Chern classes. Then, following the notation from Setup~\ref{setup:an-pn}, there is a morphism of exact triangles
\[
\begin{tikzcd}[row sep = 3em]
\bf{1}_Y \arrow{d}{g_!\left(\rm{cl}_{0_Y}\right)} \arrow{r} & \bigoplus_{k=0}^{d} \bf{1}_Y\langle d-k\rangle \arrow{d}{\sum_{k=0}^{d} c_1^k\langle d-k\rangle} \arrow{r} & \bigoplus_{k=0}^{d-1} \bf{1}_Y\langle d-k\rangle \arrow{d}{\sum_{k=0}^{d-1} c_1^k\langle d-k\rangle} \\
g_! \bf{1}_{\bf{A}^{d}_Y} \langle d\rangle \arrow{r} & f_{d, *} \bf{1}_{\bf{P}^{d}_Y} \langle d\rangle  \arrow{r} & f_{d-1, *} \bf{1}_{\bf{P}^{d-1}_Y} \langle d\rangle \end{tikzcd}
\]
in $D(S)$, where the left lower map is the evident inclusion and the right lower map is the evident projection. 
\end{lemma}
\begin{proof}
    The upper exact triangle is evident, and the lower exact triangle comes by applying $f_{d, *}=f_{d, !}$ to the excision fiber sequence (see Remark~\ref{rmk:excision-sequence})
    \[
    j_! \bf{1}_{\bf{A}^{d}_Y}\langle d \rangle \to \bf{1}_{\bf{P}^{d}_Y}\langle d \rangle \to i_{d+1, *} \bf{1}_{\bf{P}^{d-1}_Y}\langle d \rangle. 
    \]
    Lemma~\ref{lemma:cycle-point-vs-chern-classes} ensures that the left square commutes. So using the axioms of triangulated categories, we can extend this commutative square to a morphism of exact triangles:
    \[
    \begin{tikzcd}[row sep = 3em]
        \bf{1}_Y \arrow{d}{g_!(\rm{cl}_{0_Y})} \arrow{r} & \bigoplus_{k=0}^{d} \bf{1}_Y\langle d-k\rangle \arrow{d}{\sum_{k=0}^{d} c_1^k\langle d-k\rangle} \arrow{r} &     \bigoplus_{k=0}^{d-1} \bf{1}_Y\langle d-k\rangle \arrow{d}{c} \\
        g_! \bf{1}_{\bf{A}^{d}_Y} \langle d\rangle \arrow{r} & f_{d, *} \bf{1}_{\bf{P}^{d}_Y}     \langle d\rangle  \arrow{r} & f_{d-1, *} \bf{1}_{\bf{P}^{d-1}_Y} \langle d\rangle.
    \end{tikzcd}
    \]
    The only thing we are left to show is to compute $c$. It suffices to do separately on each direct summand $\bf{1}_Y\langle d-k\rangle$. Then we use that the upper exact triangle is split to see that $c|_{\bf{1}_Y\langle d-k\rangle}$ must be equal to the composition
    \[
    \bf{1}_Y\langle d-k\rangle \xr{c_1^i\langle d-k\rangle} f_{d, *} \bf{1}_{\bf{P}^{d}_Y}\langle d\rangle \xr{\rm{can}} f_{d-1, *} \bf{1}_{\bf{P}^{d-1}_Y} \langle d\rangle.
    \]
    Using the first Chern classes commute with pullbacks and $\O_{\bf{P}^{d}_Y/Y}(1)|_{\bf{P}^{d-1}_Y}=\O_{\bf{P}^{d-1}_Y/Y}(1)$, one easily sees that the composition is equal to 
    \[
    c_1^k\langle d-k\rangle \colon \bf{1}_Y\langle d-k\rangle \to f_{d-1, *} \bf{1}_{\bf{P}^{d-1}_Y} \langle d\rangle.
    \]
\end{proof}

\begin{lemma}\label{lemma:tr-cl-point-1} Suppose that $\cal{D}$ satisfies the excision axiom and that $c_1$ is a theory of first Chern classes. Let $g\colon \bf{A}^d_Y \to Y$ be a relative affine space and let $0_Y \colon Y \to \bf{A}^d_Y$ be the zero section. Then the natural morphism
\[
\bf{1}_Y \xr{g_!(\rm{cl}_{0_Y})} g_! \big(\bf{1}_{\bf{A}^{d}_Y} \langle d\rangle\big)
\]
is an isomorphism for any $d$. 
\end{lemma}
\begin{proof}
    We prove this claim by induction on $d$. \smallskip
    
    
    {\it Step~$1$. Base of induction.} Here, we follow the notation of Setup~\ref{setup:an-pn} with $d=1$. In this case, we note that the Zariski-closed immersion $i_2\colon \bf{P}^0_Y \to \bf{P}^1_Y$ is the ``$\infty$''-section of $\bf{P}^1_Y$. So the commutative diagram from Lemma~\ref{lemma:short-exact-sequence} simplifies to the following form:
    \[
    \begin{tikzcd}
        \bf{1}_Y \arrow{d}{g_!(\rm{cl}_{0_Y})} \arrow{r} & \bf{1}_Y \oplus \bf{1}_{Y}\langle 1\rangle \arrow{d}{c_1+f^*\langle 1\rangle} \arrow{r} & \bf{1}_Y\langle 1\rangle \arrow{d}{\rm{id}} \\
        g_! \bf{1}_{\bf{A}^{1}_Y} \langle 1\rangle \arrow{r} & f_* \bf{1}_{\bf{P}^{1}_Y} \langle 1\rangle  \arrow{r} & \bf{1}_Y\langle 1\rangle.
    \end{tikzcd}
    \]
    The right vertical map is clearly an isomorphism, and the middle vertical arrow is an isomorphism by Lemma~\ref{lemma:projective-bundle-formula}. Therefore, we conclude that $g_!(\rm{cl}_{0_Y})$ is also an isomorphism finishing this step. \smallskip
    
    {\it Step~$2$. Inductive argument.} We suppose that $d\geq 2$ and that we know the result for all integers less than $d$ and wish to deduce the result for $d$. For this, we consider the commutative diagram
    \[
    \begin{tikzcd}[row sep = 2em, column sep = 4em]
    Y \arrow[r, hook, "i"] \arrow[rd, hook, "i"] \arrow[swap]{rdd}{\rm{id}} \arrow[rr, bend left, hook, "0_Y"] &\bf{A}^{d-1}_Y \arrow{d}{\rm{id}}\arrow[r, hook, "j"] & \bf{A}^{d}_Y \arrow[swap]{dl}{f} \arrow{ldd}{g} \\
    & \bf{A}^{d-1}_Y \arrow{d}{h} & \\
    & Y,&  
    \end{tikzcd}
    \]
    where $i$ is the zero section of $\bf{A}^d_Y$, and $j$ is the Zariski-closed immersion realizing $\bf{A}^{d-1}_Y$ inside $\bf{A}^d_Y$ as the vanishing locus of the last coordinate. We warn the reader that this notation is different from the one used in Setup~\ref{setup:an-pn}. \smallskip

    By Remark~\ref{rmk:inductive-definition}, we have an equality (up to canonical identifications\footnote{In this proof, we will ignore canonical identifications and write ``$=$'' meaning canonically isomorphic. This does not cause any problems because our goal is to show that a well-defined morphism is an isomorphism.}) 
    \begin{equation}\label{eqn:composition-cycles}
    \rm{cl}_{0_Y} = \rm{cl}_{j}\langle1\rangle \circ j_*(\rm{cl}_i), 
    \end{equation}
    where $\cl_i$ is the naive cycle of the zero section $i\colon Y \hookrightarrow \bf{A}^{d-1}_Y$. 
    Therefore, we have the following sequence of equalities
    \begin{align*}
        g_!(\rm{cl}_{0_S}) & =g_!\bigg(\rm{cl}_j\langle 1\rangle \circ j_*\left(\rm{cl}_i\right)\bigg) \\
        & = g_!\bigg(\rm{cl}_j\langle 1\rangle\bigg) \circ g_!\bigg(j_*\left(\rm{cl}_i\right) \bigg) \\
        &= h_!\bigg(f_!(\rm{cl}_j\langle 1\rangle)\bigg) \circ 
         h_!\bigg(f_!\left(j_*\left(\rm{cl}_i\right)\right)\bigg) \\ 
        & = h_!\bigg(f_!(\rm{cl}_j\langle 1\rangle)\bigg) \circ 
         h_!\bigg(\rm{cl}_i\bigg).
    \end{align*}
    The first equality comes from Equation~(\ref{eqn:composition-cycles}). The second equality comes from the fact that $g_!$ is a functor. The third equality comes from the fact that $g=h\circ f$. The fourth equality comes from the fact that $f\circ j=\rm{id}$ and $j_!=j_*$ (because $j$ is a closed immersion). \smallskip
    
    Now we note that the induction hypothesis implies that $h_!(\rm{cl}_i)$ is an isomorphism. Similarly, we note that the induction hypothesis implies that $f_!(\rm{cl}_j)$ is an isomorphism by applying it to relative $\bf{A}^1$-morphism $f\colon \bf{A}^{d+1}_Y \to \bf{A}^{d}_Y$. Therefore, we conclude that the composition
    \[
     g_!(\rm{cl}_{0_S}) = h_!\bigl(f_!(\rm{cl}_j\langle 1\rangle)\bigr) \circ 
         h_!\bigl(\rm{cl}_i\bigr)
    \]
    must be an isomorphism as well.  
\end{proof}

\begin{thm}\label{thm:theory-is-strong-theory} Suppose that $\cal{D}$ satisfies the excision axiom and that $c_1$ is a theory of first Chern classes. Then $c_1$ is a strong theory of first Chern classes (see Definition~\ref{defn:geometric-chern-classes-without-length}). 
\end{thm}
\begin{proof}
    Following the notation of Definition~\ref{defn:geometric-chern-classes-without-length}, we need to show that the morphism
    \[
    \sum_{k=0}^d c_1^k\langle d-k\rangle \colon \bigoplus_{k=0}^d \bf{1}_S\langle d-k\rangle \to f_{d, *} \bf{1}_{\bf{P}_S^d}\langle d\rangle
    \]
    is an isomorphism for the relative projective space $f_d\colon \bf{P}^d_S \to S$ for any $d\geq 1$. For $d=1$, this is the definition of a theory of first Chern classes. For $d>1$, this follows from Lemma~\ref{lemma:short-exact-sequence} and Lemma~\ref{lemma:tr-cl-point-1} by an evident inductive argument.
\end{proof}

\subsection{Trace morphisms}\label{section:trace-map}

The main goal of this section is to construct the trace morphism for the relative projective line from a theory of first Chern classes. Then we show that any theory of first Chern classes underlying a theory of cycle maps (see Definition~\ref{defn:theory-of-cycle-classes-chern-classes}) admits a trace-cycle theory on the relative projective line (see Definition~\ref{defn:trace-cycle}). When combined with Theorem~\ref{thm:cohomologically-smooth}, this already shows that any smooth morphism is cohomologically smooth with respect to a $6$-functor formalism with a theory of first Chern classes. \smallskip

As previously, we fix an invertible object $\bf{1}_S\langle 1\rangle\in \cal{D}(S)$. In this section, we also fix a theory of first Chern Classes with respect $\bf{1}_S\langle 1\rangle$ (see Definition~\ref{defn:geometric-chern-classes-without-length}). \smallskip

\subsubsection{Recovering trace morphisms} Now we discuss the construction of the trace morphism for the relative projective line. It comes as the ``inverse'' of the first Chern class morphism. More precisely, we fix the relative projective line $f\colon \bf{P}_Y^1 \to Y$ and recall that Lemma~\ref{lemma:projective-bundle-formula} provides us with the isomorphism
\begin{equation}\label{eqn:decomposition-p1}
c_1+f^*\langle 1\rangle \colon \bf{1}_Y \oplus \bf{1}_Y\langle1\rangle  \to f_* \bf{1}_{\bf{P}^1_Y}\langle 1\rangle.
\end{equation}
We denote by $(c_1)^{-1}\colon f_* \bf{1}_{\bf{P}^1_Y}\langle 1\rangle \to \bf{1}_Y$ the projection onto the first component of the decomposition~(\ref{eqn:decomposition-p1}).

\begin{construction}\label{construction:trace-p1} The {\it trace map} $\tr_f\colon f_*\bf{1}_{\bf{P}^1_Y} \langle 1\rangle \to \bf{1}_Y$ is the morphism 
\[
(c_1)^{-1}\colon f_*\bf{1}_{\bf{P}^1_Y} \langle 1\rangle \to \bf{1}_Y.
\]
\end{construction}

\begin{rmk} The formation of $\tr_f$ commutes with arbitrary base change. This formally follows from the fact that $c_1(\O_{\bf{P}^1_Y/Y}(1))$ commutes with arbitrary base change.
\end{rmk}

\begin{warning} This construction is well-defined only if we assume that $c_1$ is a theory of first Chern classes, and not merely a weak theory of first Chern classes. 
\end{warning}

For the later reference, it will also be convenient to discuss a more general construction of trace morphisms for a {\it strong} theory of first Chern classes (see Definition~\ref{defn:geometric-chern-classes-without-length}). In this situation, Lemma~\ref{lemma:projective-bundle-formula} provides us with the isomorphism
\begin{equation}\label{eqn:decomposition-projective-bundle}
\sum_{k=0}^d c_1^k\langle d-k\rangle \colon \bigoplus_{k=0}^d \bf{1}_Y\langle d-k\rangle \to f_* \bf{1}_{\bf{P}_Y(\cal{E})}\langle d\rangle.
\end{equation} 
for any object $Y\in \cal{C}$, a rank $d+1$ vector bundle $\cal{E}$, and the corresponding projective bundle 
\[
f\colon \bf{P}_Y(\cal{E}) \to Y.
\] 
As above, it make sense to define $(c_1^d)^{-1}\colon f_* \bf{1}_{\bf{P}_Y(\cal{E})}\langle d\rangle \to \bf{1}_Y$ to be the projection onto the last component of decomposition~(\ref{eqn:decomposition-projective-bundle}).

\begin{construction}\label{construction:trace-projective-bundle} In the notation as above, the {\it trace map} $\tr_f\colon f_*\bf{1}_{\bf{P}_Y(\cal{E})} \langle d\rangle \to \bf{1}_Y$ is the morphism 
\[
(c_1^d)^{-1}\colon f_* \bf{1}_{\bf{P}_Y(\cal{E})}\langle d\rangle \to \bf{1}_Y.
\]
\end{construction}

\subsubsection{Properties of the trace morphism}

Our next goal is to show that, if $c_1$ is a theory of first Chern classes underlying a theory of cycle maps $\cl_\bullet$, then the triple $(\bf{1}_{\bf{P}^1_S}\langle 1\rangle, \tr_f, \cl_{\Delta})$ satisfies the definition of a trace-cycle theory (see Definition~\ref{defn:trace-cycle}) where $f\colon \bf{P}^1_S \to S$ is the relative projective line. For this, we will actually show a stronger statement:

\begin{prop}\label{prop:projective-bundle-trace-cycle} Let $c_1$ be a theory of first Chern classes on $\cal{D}$ underlying a theory of cycle maps $\cl_\bullet$ (see Definition~\ref{defn:theory-of-cycle-classes-chern-classes}), let $f\colon \bf{P}^1_Y \to Y$ be the relative projective line, let $s\in \bf{P}^1_Y(Y)$ be a section, and let $\tr_f\colon f_*\bf{1}_{\bf{P}^1_Y}\langle 1\rangle \to \bf{1}_Y$ be the trace morphism from Construction~\ref{construction:trace-p1}. Then the diagram 
\[
\begin{tikzcd}
\bf{1}_Y \arrow{r}{\sim} \arrow{d}{\rm{Id}}& {f}_{*}\left(s_* \bf{1}_Y\right) \arrow{d}{{f}_{*}\left(\rm{cl}_{s}\right)}\\
\bf{1}_Y & \arrow{l}{\tr_{f}} f_{*} \bf{1}_{\bf{P}^1_Y}\langle 1\rangle
\end{tikzcd}
\]
commutes in $D(Y)$. 
\end{prop}
Whenever we use Construction~\ref{construction:adjoint-classes} in the following proof, we use the notation $^{\rm{adj}}c$ to distinguish Chern morphisms on the base and morphisms adjoint to Chern morphisms on $\bf{P}^1_Y$. 
\begin{proof}
    We first note that Remark~\ref{rmk:first-chern-first-cycle-comp} implies that $f_*(\rm{cl}_s)$ (up to a canonical identification $f_*s_*\bf{1}_Y\simeq \bf{1}_Y$) is equal to 
    \[
    ^{\rm{adj}}c_1(\O(s)) \colon \bf{1}_Y \to f_* \bf{1}_{\bf{P}^1_Y}\langle1\rangle,
    \]
    where $\O(s)$ is the line bundle corresponding to the effective Cartier divisor $s\colon Y \to \bf{P}^1_Y$. We wish to show that 
    \begin{equation}\label{eqn:tr-cycle-1-pl}
    \tr_f \circ ^{\rm{adj}}c_1(\O(s)) = \rm{id}.
    \end{equation}
    \cite[Cor.\,8.9]{adic-notes} (resp. its schematic counterpart) implies that there is a decomposition of $Y$ into clopen subspaces $Y=\sqcup_{i\in I} Y_i$ with the induced morphisms
    \[
    f_i\colon \bf{P}^1_{Y_i} \to Y_i, \ s_i\colon Y_i\to \bf{P}^1_{Y_i}
    \]
    such that, $\O_{\bf{P}^1_{Y_i}}(s_i)=f_i^*\cal{L}_i \otimes \O_{\bf{P}^1_{Y_i}/Y_i}(n_i)$ for some $\cal{L_i}\in \rm{Pic}(Y_i)$ and integers $n_i$. Equation~(\ref{eqn:tr-cycle-1-pl}) can be checked on each $Y_i$ separately, so we can assume that $\O(s)\simeq f^*\cal{L}\otimes \O(n)$ for a line bundle $\cal{L}$ on $Y$ and an integer $n$.\smallskip
    
    By restricting onto a fiber, one concludes that $n=1$, so we have an isomorphism
    \[
    \O(Y) \simeq f^*\cal{L} \otimes \O(1).
    \]
    Therefore, we see that 
    \[
    ^{\rm{adj}}c_1(\O(s)) = ^{\rm{adj}}c_1(f^*\cal{L}) + ^{\rm{adj}}c_1(\O(1))\colon \bf{1}_Y \to f_* \bf{1}_{\bf{P}^1_Y}\langle 1\rangle. 
    \]
    By definition, we know that $\tr_f \circ ^{\rm{adj}}c_1(\O(1)) = \rm{id}$. Thus we reduce the question to showing that $\tr_f \circ ^{\rm{adj}}c_1(f^*\cal{L}) = 0$ for any line bundle $\cal{L}$ on $Y$. For this, we note that 
    \[
    c_1(f^*\cal{L}) = f^*c_1(\cal{L}).
    \]
    Therefore, after unravelling Construction~\ref{construction:adjoint-classes}, we get that $^{\rm{adj}}c_1(f^*\cal{L})$ is equal to the composition
    \[
    \bf{1}_Y \xr{c_1(\cal{L})} \bf{1}_Y \langle 1\rangle \xr{f^*\langle 1\rangle} f_*\bf{1}_{\bf{P}^1_Y}\langle 1\rangle.
    \]
    By definition of the trace map, we have $\tr_f\circ f^*\langle 1\rangle=0$. Therefore, this formally implies that
    \[
    \tr_f \circ ^{\rm{adj}}c_1(f^*\cal{L}) = 0
    \]
    finishing the proof. 
\end{proof}

Proposition~\ref{prop:projective-bundle-trace-cycle} already has some non-trivial consequences:

\begin{cor}\label{cor:first-chern-cycle-theory} Let $c_1$ be a {\it strong} theory of first Chern classes on $\cal{D}$ underlying a theory of cycle maps $\cl_\bullet$. Then the triple 
\[
(\bf{1}_{\bf{P}^1_S}\langle 1\rangle, \tr_f, \cl_{\Delta})
\]
forms a trace-cycle theory on the relative projective line $f\colon \bf{P}^1_S \to S$. In particular, any smooth morphism in $\cal{C}$ is cohomologically smooth with respect to $\cal{D}$ (see Definition~\ref{defn:cohomologically-smooth}). 
\end{cor}
\begin{proof}
    In this proof, we will freely identify 
    \[
    p_1^*\bf{1}_{\bf{P}^1_S} \simeq \bf{1}_{\bf{P}^1_S \times_S \bf{P}^1_S} \simeq p_2^*\bf{1}_{\bf{P}^1_S}.
    \]
    Thus, the cycle map of the diagonal takes the form
    \[
    \rm{cl}_\Delta \colon \Delta_! \bf{1}_{\bf{P}^1_S} \to \bf{1}_{\bf{P}^1_S\times_S \bf{P}^1_S} \langle 1\rangle
    \]
    defining a cycle map in the sense of Definition~\ref{defn:trace-cycle}. \smallskip
    
    Now commutativity of the first diagram in Definition~\ref{defn:trace-cycle} follows directly from Proposition~\ref{prop:projective-bundle-trace-cycle} by taking $Y=\bf{P}^1_S$, $f=p_1$, and $s=\Delta$. We wish to establish commutativity of the second diagram. \smallskip
    
    For brevity, we denote $\bf{P}^1_S$ by $X$ and $\bf{P}^1_S\times_S \bf{P}^1_S$ by $X^2$. We have to check that the composition
    \[
    \bf{1}_{X}\langle 1\rangle \simeq p_{2, !}\left(\bf{1}_{X^2}\langle 1\rangle \otimes \Delta_! \bf{1}_{X}\right) \xr{p_{2, !}(\rm{id} \otimes \rm{cl}_\Delta)} p_{2, !}(\bf{1}_{X^2}\langle 1\rangle \otimes \bf{1}_{X^2}\langle 1\rangle) \simeq p_{2, !}\bf{1}_{X^2}\langle 1\rangle \otimes \bf{1}_X\langle 1\rangle \xr{\tr_{p_2}\otimes \rm{id}} \bf{1}_X\langle 1\rangle 
    \]
    is equal to the identity morphism (in the homotopy category $D(X)$). For this, we first note that Lemma~\ref{lemma:first-chern-classes-orientable} implies that the diagram
    \[
    \begin{tikzcd}[column sep = 4em]
        \bf{1}_{X^2}\langle1\rangle \otimes \Delta_! \bf{1}_{X} \arrow{d}{\wr}\arrow{r}{\rm{id}\otimes \rm{cl}_\Delta} & \bf{1}_{X^2}\langle 1\rangle \otimes \bf{1}_{X^2}\langle 1\rangle  \\
        \Delta_! \bf{1}_{X} \otimes \bf{1}_{X^2}\langle1\rangle \arrow[swap]{ru}{\rm{cl}_\Delta \otimes \rm{id}} & 
    \end{tikzcd}
    \]
    commutes in $D(X^2)$, where the left vertical map is the braiding morphism. Therefore, we have the following commutative diagram
    \[
    \begin{tikzcd}[column sep = 6em]
        \bf{1}_X\langle 1\rangle \arrow{r}{\sim}\arrow{dd}{\rm{id}} &    p_{2, !}\left(\bf{1}_{X^2}\langle 1\rangle \otimes \Delta_! \bf{1}_X \right) \arrow{d}{\wr} \arrow{r}{p_{2, !}\left(\rm{id} \otimes \rm{cl}_\Delta\right)} & p_{2, !}(\bf{1}_{X^2}\langle 1\rangle \otimes \bf{1}_{X^2}\langle 1\rangle) \arrow{d}{\rm{id}}\\
        &    p_{2, !}\left(\Delta_! \bf{1}_{X} \otimes \bf{1}_{X^2}\langle1\rangle\right) \arrow{r}{p_{2, !}\left(\rm{cl}_\Delta \otimes \rm{id} \right)}\arrow{d}{\wr} & p_{2, !}(\bf{1}_{X^2}\langle 1\rangle \otimes \bf{1}_{X^2}\langle 1\rangle) \arrow{d}{\wr} \\
        \bf{1}_X\langle 1\rangle\arrow{r}{\sim} &    \bf{1}_X \otimes \bf{1}_{X}\langle 1\rangle \arrow{r}{p_{2, !}(\cl_\Delta) \otimes \rm{id}}& p_{2, !} \bf{1}_{X^2}\langle 1\rangle \otimes \bf{1}_X\langle 1\rangle,
    \end{tikzcd}
    \]
    where the two bottom vertical maps come from the projection formula. Therefore, it suffices to show that the composition
    \[
    \bf{1}_X\langle 1\rangle \xr{p_{2, !}(\cl_\Delta) \otimes \rm{id}} p_{2, !} \bf{1}_{X^2}\langle 1\rangle \otimes \bf{1}_X\langle 1\rangle \xr{\tr_{p_2}\otimes \rm{id}} \bf{1}_X\langle 1\rangle 
    \]
    is equal to the identity morphism (in the homotopy category $D(X)$). For this, it suffices to show that 
    \[
    \rm{tr}_{p_2} \circ p_{2, !}(\rm{cl}_\Delta) = \rm{id}. 
    \]
    This follows from Proposition~\ref{prop:projective-bundle-trace-cycle} by taking $Y=\bf{P}^1_S$, $f=p_2$, and $s=\Delta$. \smallskip
    
    Overall, this proves that $(\bf{1}_{\bf{P}^1_S}\langle 1\rangle, \tr_f, \cl_{\Delta})$ forms a trace-cycle theory. The ``in particular'' claim follows directly from Theorem~\ref{thm:cohomologically-smooth}.
\end{proof}


Now we discuss another consequence of Proposition~\ref{prop:projective-bundle-trace-cycle}: we show that a $6$-functor formalism $\cal{D}$ satisfying the excision axiom and admitting a theory of first Chern classes is automatically weakly $\bf{A}^1$-invariant (see Definition~\ref{defn:homotopy-invariant}). For this, we need the following construction:

\begin{construction}\label{construction:adjoints-trace} Let $f\colon \bf{P}^{1}_Y \to Y$ be the relative projective line with a trace morphism $\rm{tr}\colon f_* \bf{1}_{\bf{P}^1_Y}\langle 1\rangle \to \bf{1}_Y$. By the $(f_*, f^!)$-adjunction, it also defines the {\it adjoint trace} morphism 
\[
^{\rm{adj}}\rm{tr} \colon \bf{1}_{\bf{P}^1_Y}\langle 1\rangle  \to f^!\left(\bf{1}_Y\right).
\]
\end{construction}

\begin{lemma}\label{cor:excision-first-classes-geometric} Let $\cal{D}$ be a $6$-functor formalism satisfying the excision axiom and let $c_1$ be a theory of first Chern classes. Then $\cal{D}$ is weakly motivic (see Definition~\ref{defn:motivic-6-functors}).
\end{lemma}
\begin{proof}
    Firstly, we note that Lemma~\ref{lemma:canonical-cycle-classes} constructs a theory of cycle maps underlying $c_1$. Furthermore, Theorem~\ref{thm:theory-is-strong-theory} implies that $c_1$ is a strong theory of first Chern classes. Therefore, Corollary~\ref{cor:first-chern-cycle-theory} ensures that any smooth morphism is cohomologically smooth. So we only need to show that $\cal{D}$ is weakly $\bf{A}^1$-invariant.\smallskip
    
    We fix a relative affine line $g\colon \bf{A}^1_Y \to Y$ and compactify it to a relative projective line $f\colon \bf{P}^1_Y \to Y$. The complement of the open immersion $j \colon \bf{A}^1_Y \hookrightarrow \bf{P}^1_Y$ forms a section $s\colon Y \to \bf{P}^1_Y$. Then Definition~\ref{defn:cycle-clas-divisors} defines a theory cycle maps underlying $c_1$. In particular, it defines a morphism
    \[
    s_* \bf{1}_{Y} \to \bf{1}_{\bf{P}^1_Y}\langle 1\rangle.
    \]
    Using Proposition~\ref{prop:projective-bundle-trace-cycle}, it is essentially formal to verify that the following diagram commutes:
    \[
    \begin{tikzcd}
        s_* \bf{1}_Y \arrow{d}{\rm{cl}_s} \arrow{r}{\sim} & s_* s^! f^! \bf{1}_Y \arrow{d}{\rm{adj}} \\
        \bf{1}_{\bf{P}^1_Y}\langle 1\rangle \arrow{r}{^{\rm{adj}}\rm{tr}} & f^! \bf{1}_Y. 
    \end{tikzcd}
    \]
    Now Corollary~\ref{cor:first-chern-cycle-theory} and Theorem~\ref{thm:abstract-poincare-duality} provide us with an equivalence of functors $\gamma\colon f^*(-)\otimes \bf{1}_{\PP^1_Y}\langle 1\rangle \xr{\sim} f^!(-)$. By construction, the $(f_!\simeq f_*, f^!)$-adjoint to $\gamma(\bf{1}_Y)$ is given by $\tr\colon f_* \bf{1}_{\PP^1_Y}\langle 1\rangle \to \bf{1}_Y$. In other words, we conclude that $^{\rm{adj}}\rm{tr} = \gamma(\bf{1}_Y)$ is an isomorphism. 
    Thus, the exact triangle $s_*s^! f^! \bf{1}_Y \xr{\rm{adj}} f^!\bf{1}_Y \xr{\rm{can}} j_*j^*f^!\bf{1}_Y$ becomes isomorphism to the following exact triangle: 
    \[
    s_* \bf{1}_Y \xr{\rm{cl}_s} \bf{1}_{\bf{P}^1_Y}\langle 1\rangle \xr{\rm{can}} j_* \bf{1}_{\bf{A}^1_Y}\langle 1\rangle.
    \]
    Now we apply $f_*$ (and Remark~\ref{rmk:first-chern-first-cycle-comp}) to this sequence to get an exact triangle
    \[
    \bf{1}_Y \xr{c_1} f_* \bf{1}_{\bf{P}^1_Y}\langle 1\rangle \to g_* \bf{1}_{\bf{A}^1_Y}\langle 1\rangle.
    \]
    In particular, we have a commutative diagram of exact triangles 
    \[
    \begin{tikzcd}
        \bf{1}_Y \arrow{r} \arrow{d}{\rm{id}} & \bf{1}_Y \oplus \bf{1}_{Y}\langle 1\rangle \arrow{d}{c_1+f^*\langle 1\rangle} \arrow{r} &\bf{1}_Y\langle 1\rangle \arrow{d}{g^*\langle 1\rangle} \\
        \bf{1}_Y \arrow{r}{c_1} & f_* \bf{1}_{\bf{P}^1_Y}\langle 1\rangle \arrow{r}& g_* \bf{1}_{\bf{A}^1_Y}\langle 1\rangle.
    \end{tikzcd}
    \]
    Now the definition of the first Chern classes and the $2$-out-of-$3$ property implies that 
    \[
    \bf{1}_Y\langle 1\rangle \to g_* \bf{1}_{\bf{A}^1_Y}\langle 1\rangle
    \]
    is an isomorphism. Since $\bf{1}_Y\langle 1\rangle$ is an invertible sheaf, this formally implies that the natural morphism $\bf{1}_Y \to g_* \bf{1}_{\bf{A}^1_Y}$ is an isomorphism as well. \smallskip
\end{proof}

\subsection{Poincar\'e Duality}\label{section:duality}

The first goal of this section is to show that a strong theory of first Chern classes $c_1$ underlying a theory of cycle maps (see Definition~\ref{defn:geometric-chern-classes-without-length} and Definition~\ref{defn:theory-of-cycle-classes-chern-classes}) implies the strongest version of Poincar\'e Duality under the additional assumption that $\cal{D}$ is either weakly $\bf{A}^1$-invariant or pre-geometric (see Definition~\ref{defn:6-functors-geometric}). The second goal is to show that, if $\cal{D}$ satisfies the excision axiom, it suffices to assume that $\cal{D}$ admits a theory of first Chern classes. \smallskip

We now briefly sketch the idea behind the proof. Corollary~\ref{cor:first-chern-cycle-theory} reduces the question of proving Poincar\'e Duality to the question of computing dualizing object $f^! \bf{1}_Y$. For this, we use Theorem~\ref{thm:formula-dualizing-complex} (or Theorem~\ref{thm:formula-dualizing-complex-geometric}) to reduce the question to computing $C(\rm{T}_f)$. This is done via compactifying $\rm{T}_f$ to a projective bundle and the (naive) cycle map of a point from Definition~\ref{defn:naive-cycle-class}. \smallskip

For the rest of this section, we fix a $6$-functor formalism $\cal{D}$ with a {\it strong} theory of first Chern classes $c_1$ underlying a theory of cycle maps $\cl_\bullet$ (see Definition~\ref{defn:theory-of-cycle-classes-chern-classes}). \smallskip

We start by defining the adjoint to the trace map from Construction~\ref{construction:trace-projective-bundle}. More precisely, let $Y$ be an object of $\cal{C}$, $\cal{E}$ is a vector bundle on $Y$ of rank $d+1$, and 
\[
f\colon \bf{P}_Y(\cal{E}) \to Y
\]
be the corresponding projective bundle. Then Construction~\ref{construction:trace-projective-bundle} defines the trace morphism
\[
\tr_f \colon f_* \bf{1}_{\bf{P}_Y(\cal{E})}\langle d\rangle \to \bf{1}_Y
\]

\begin{construction}\label{construction:adjoints-trace-vector-bundle} Let $f\colon \bf{P}_Y(\cal{E}) \to Y
$ and $\tr_f$ be as above. By the $(f_*, f^!)$-adjunction, $\tr_f$ uniquely defines the {\it adjoint trace} morphism 
\[
^{\rm{adj}}\rm{tr} \colon \bf{1}_{\bf{P}_Y(\cal{E})}\langle d\rangle  \to f^!\left(\bf{1}_Y\right).
\]
in $D(\bf{P}_Y(\cal{E}))$.
\end{construction}

Now suppose that $\cal{E}=\O_Y^{d+1}$, so $\bf{P}_Y(\cal{E})=\bf{P}^d_Y$. Then Definition~\ref{defn:naive-cycle-class} defines the (cycle) class of the ``zero'' section
\[
\rm{cl}_s\colon s_*\bf{1}_Y \to \bf{1}_{\bf{P}^d_Y}\langle d \rangle.
\]

\begin{construction}\label{construction:adjoints-Gysin} In the notation as above, $\cl_s$ uniquely defines the {\it adjoint cycle map} morphism 
\[
^{\rm{adj}} \cl_s \colon \bf{1}_Y \to s^!\left(\bf{1}_{\bf{P}^d_Y}\langle d \rangle\right).
\]
in $D(Y)$.
\end{construction}

\begin{lemma}\label{lemma:split} Let $c_1$ be a strong theory of first Chern classes on $\cal{D}$ underlying a theory of cycle maps $\cl_\bullet$ and let $f\colon \bf{P}^d_Y \to Y$ be the relative projective space. Then the following diagram
\[
\begin{tikzcd}
    \bf{1}_Y \arrow{r}{^{\rm{adj}} \cl_s} \arrow[rd, "\sim" {sloped, anchor=center, above}] & s^! \bf{1}_{\bf{P}^d_Y}\langle d\rangle \arrow{d}{s^!(^{\rm{adj}}\tr_f)} \\
    & s^!f^!\bf{1}
\end{tikzcd}
\]
commutes in $D(Y)$.
\end{lemma}
\begin{proof}
    By passing to adjoints, it suffices to show that the diagram
    \[
    \begin{tikzcd}
        f_* s_* \bf{1}_Y \arrow{r}{f_*(\cl_s)} \arrow[rd, "\sim" {sloped, anchor=center, above}] \arrow[swap]{rd}{h}& f_*\bf{1}_{\bf{P}^d_Y}\langle d\rangle     \arrow{d}{\tr_f} \\
        & \bf{1}_Y 
    \end{tikzcd}
    \]
    commutes in $D(Y)$. Lemma~\ref{lemma:naive-cycle-of-point-product-of-chern-classes} and a formal argument with adjoints (similar to Remark~\ref{rmk:first-chern-first-cycle-comp}) implies that the composition
    \[
    \bf{1}_Y \xr{h^{-1}} f_* s_* \bf{1}_Y \xr{f_*(\cl_s)} f_*\bf{1}_{\bf{P}^d_Y}\langle d\rangle
    \]
    is equal to the morphism adjoint to $c_1^d(\O_{\bf{P}^d_Y/Y}(1)) \colon \bf{1}_{\bf{P}^d_Y} \to \bf{1}_{\bf{P}^d_Y}\langle d\rangle$. In other words, this composition is equal to the morphism
    \[
    c_1^d \colon \bf{1}_Y \to f_*\bf{1}_{\bf{P}^d_Y}\langle d\rangle
    \]
    from Construction~\ref{construction:adjoint-classes} applied to $c=c_1\bigl(\O_{\bf{P}^d_Y/Y}(1)\bigr)$. Therefore, the question boils down to showing that the composition 
    \[
    \bf{1}_Y \xr{c_1^d} f_*\bf{1}_{\bf{P}^d_Y}\langle d\rangle \xr{\tr_f} \bf{1}_Y
    \]
    is the identity morphism (in $D(Y)$). However, this follows from the definition of the trace morphism (see Construction~\ref{construction:trace-projective-bundle}). 
\end{proof}

Now we turn to the proof of Poincar\'e Duality. In the process of the proof, we will need the following simple (but useful) lemma:

\begin{lemma}\label{lemma:invertible-indecomposable} Let $D$ be a closed symmetric monoidal additive category with a unit object $\bf{1}$ and $L$ be an invertible object. Suppose that $L=\bf{1}\oplus X$. Then $X\simeq 0$.
\end{lemma}
\begin{proof}
    If $L$ is an invertible object, then the natural evaluation morphism
    \[
    L\otimes L^{\vee} \to \bf{1}
    \]
    must be an isomorphism. Now we write 
    \[
    L\otimes L^{\vee} \simeq (\bf{1}\oplus X)\otimes (\bf{1}\oplus X)^{\vee} \simeq (\bf{1}\oplus X)\otimes (\bf{1}\oplus X^{\vee}) \simeq \bf{1}\oplus X\oplus X^{\vee} \oplus X\otimes X^{\vee}
    \]
    to conclude that $X=X^{\vee}=0$. 
\end{proof}

Now we specialize to the case of a vector bundle of the form $\cal{E}'=\cal{E}\oplus \O$ on an object $Y\in \cal{C}$. Then the relative projective bundle
\[
f\colon \bf{P}_Y(\cal{E}\oplus \cal{O})\to Y
\]
has a canonical section $s\colon Y \to \bf{P}_Y(\cal{E}\oplus \cal{O})$ corresponding to the quotient $\cal{E}\oplus \O\xr{p} \O$. 

\begin{lemma}\label{lemma:trivialize} Let $c_1$ be a strong theory of first Chern classes on $\cal{D}$ underlying a theory of cycle maps $\cl_\bullet$, let $Y$ be an object of $\cal{C}$, let $\cal{E}$ be a vector bundle of rank $d+1$ on $Y$, and let
\[
f\colon \bf{P}_Y(\cal{E}\oplus \O)\to Y
\]
be the relative projective bundle with the canonical section $s$. Then the natural morphism
\[
s^!(^{\rm{adj}}\tr_f) \colon s^!\bf{1}_{\bf{P}_Y(\cal{E}\oplus \O)}\langle d\rangle \to s^!f^!\bf{1}_Y
\]
is an isomorphism, where $^{\rm{adj}}\tr_f$ is from Construction~\ref{construction:adjoints-trace-vector-bundle}.
\end{lemma}
\begin{proof}
    We first note that the question is local on $Y$, so we can assume that $\cal{E}=\O_Y^{\oplus d+1}$. So $\bf{P}_Y(\cal{E}\oplus \O)\simeq \bf{P}^d_Y$, and $s$ corresponds to the ``zero'' section defined just before Definition~\ref{defn:naive-cycle-class}. \smallskip
    
    Now we note that $s^!\bf{1}_{\bf{P}_Y^d}$ is an invertible object. Indeed, Corollary~\ref{cor:first-chern-cycle-theory} (and Definition~\ref{defn:cohomologically-smooth}) implies that $f^! \bf{1}_Y$ is an invertible object. Therefore, Lemma~\ref{lemma:invertible-projection} implies that
    \[
    \bf{1}_Y \simeq s^!f^!\bf{1}_Y \simeq s^!\bf{1}_{\bf{P}^d_Y} \otimes s^* f^! \bf{1}_Y.
    \]
    Since $s^*f^! \bf{1}_Y$ is invertible and $s^!\bf{1}_{\bf{P}^d_Y}$ is dual to it, we formally conclude that $s^!\bf{1}_{\bf{P}^d_Y}$ is invertible as well. \smallskip
    
    Now we note that Construction~\ref{construction:adjoints-Gysin} defines a morphism
    \[
    ^{\rm{adj}}\cl_s \colon \bf{1}_Y\to s^! \bf{1}_{\bf{P}^d_Y}.
    \]
    Lemma~\ref{lemma:split} implies that the composition
    \[
    \bf{1}_Y\xr{^{\rm{adj}}\cl_s} s^! \bf{1}_{\bf{P}^d_Y}\langle d\rangle \xr{s^!(^{\rm{adj}}\tr_f)} s^!f^!\bf{1}_Y\simeq \bf{1}_Y
    \]
    is the identity morphism (in the homotopy category $D(Y)$). So $\bf{1}_Y$ is a direct summand of the invertible object $s^! \bf{1}_{\bf{P}^d_Y}\langle d\rangle$. Therefore, Lemma~\ref{lemma:invertible-indecomposable} implies that both $^{\rm{adj}}\cl_s$ and $s^!(^{\rm{adj}}\tr_f)$ must be isomorphisms. 
\end{proof}

\begin{thm}\label{thm:poincare-duality-chern-classes} Suppose that a $6$-functor formalism $\cal{D}$ is either weakly $\bf{A}^1$-invariant or pre-geometric. Let $c_1$ be a strong theory of first Chern classes on $\cal{D}$ underlying a theory of cycle maps $\cl_\bullet$ and let $f\colon X \to Y$ be a smooth morphism of pure relative dimension $d$ (see \cite[Def.\,1.8.1]{H3}). Then the right adjoint to the functor
\[
f_!\colon \cal{D}(X) \to \cal{D}(Y)
\]
is given by the formula
\[
f^!(-) = f^*(-)\otimes \bf{1}_X\langle d\rangle \colon \cal{D}(Y) \to \cal{D}(X).
\]
\end{thm}
\begin{proof}
    Corollary~\ref{cor:first-chern-cycle-theory} already implies that any smooth morphism $f\colon X \to Y$ is cohomologically smooth. Thus, the question of computing $f^!$ boils down to computing the dualizing object $\omega_f=f^!\bf{1}_Y$.\smallskip
    
    Now Theorem~\ref{thm:formula-dualizing-complex} (if $\cal{D}$ is weakly $\bf{A}^1$-invariant) and Theorem~\ref{thm:formula-dualizing-complex-geometric} (if $\cal{D}$ is pre-geometric) imply that $f^!\bf{1}_Y$ is given by the formula
    \[
    f^!\bf{1}_Y \simeq C_X(\rm{T}_f) \simeq s^*g^!\bf{1}_X,
    \]
    where $g\colon \rm{V}_X(\rm{T}_f) \to X$ is the total space of the (relative) tangent bundle, and $s$ is the zero section. We may compactify $g$ to the morphism\footnote{The dual vector bundle $\rm{T}_f^{\vee}$ shows up due to the convention used in \cite[Def.\,7.11]{adic-notes}.}
    \[
    \ov{g}\colon P\coloneqq \bf{P}_X(\rm{T}^{\vee}_f \oplus \O_X) \to X,
    \]
    where $s$ corresponds to the ``zero'' section defined just before Definition~\ref{defn:naive-cycle-class}. Therefore, it suffices show that 
    \[
    s^*\ov{g}^!\bf{1}_X\simeq \bf{1}_X\langle d\rangle. 
    \]
    For this, we note that 
    \[
    \bf{1}_X \simeq s^!\ov{g}^!\bf{1}_X \simeq s^!\bf{1}_{P} \otimes s^* \ov{g}^! \bf{1}_X,
    \]
    where the second isomorphism follows from Lemma~\ref{lemma:invertible-projection} and the fact that $\ov{g}^!\bf{1}_X$ is invertible due to cohomological smoothness. Thus, it suffices to produce an isomorphism
    \[
    s^!\bf{1}_{P} \simeq \bf{1}_X\langle -d\rangle. 
    \]
    This follows from Lemma~\ref{lemma:trivialize} and Lemma~\ref{lemma:invertible-projection}.
\end{proof}

\begin{thm}\label{thm:main-thm} Let $\cal{D}$ be a $6$-functor formalism satisfying the excision axiom (see Definition~\ref{defn:excision-axiom}) and admitting a theory of first Chern classes  $c_1$. Suppose that $f\colon X \to Y$ is a smooth morphism of pure relative dimension $d$. Then the right adjoint to the functor 
\[
f_!\colon \cal{D}(X) \to \cal{D}(Y)
\]
is given by the formula
\[
f^!(-) = f^*(-)\otimes \bf{1}_X\langle d\rangle \colon \cal{D}(Y) \to \cal{D}(X).
\]
\end{thm}
\begin{proof}
    Firstly, we note that Lemma~\ref{lemma:canonical-cycle-classes} constructs a theory of cycle maps underlying $c_1$. Furthemore, Theorem~\ref{thm:theory-is-strong-theory} ensures that $c_1$ is a {\it strong} theory of first Chern classes. Then Corollary~\ref{cor:excision-first-classes-geometric} implies that $\cal{D}$ is weakly $\bf{A}^1$-invariant (or even weakly motivic). Thus the result follows from Theorem~\ref{thm:poincare-duality-chern-classes}.
\end{proof}

\section{Poincar\'e Duality in examples}

In this section, we apply Theorem~\ref{thm:main-thm} to two particular examples of $6$-functor formalisms: $\ell$-adic \'etale sheaves on locally noetherian analytic adic spaces (resp. schemes) developed by R.\,Huber in \cite{H3}, and ``solid almost $\O^+/p$-$\varphi$-modules'' on $p$-adic adic spaces developed by L.\,Mann in \cite{Lucas-thesis}. \smallskip

In the first example, we recover Poincar\'e Duality previously established by R.\,Huber in \cite[Th\,7.5.3]{H3}. The proof is essentially formal: after unravelling all the definitions, Theorem~\ref{thm:main-thm} tells us that, for the purpose of proving Poincar\'e Duality, it suffices to construct a theory of first Chern classes and compute cohomology of the relative projective line. Both things are particularly easy in the case of \'etale sheaves: the theory of first Chern classes comes from the Kummer exact sequence, and the computation of \'etale cohomology of the projective line essentially boils down to proving $\rm{Pic}(\bf{P}^1_C)\simeq \Z$. This proof completely avoids quite elaborate construction of the trace map and verification of Deligne's fundamental lemma (see \cite[\textsection 7.2-7.4]{H3}). The same proof applies to $\ell$-adic sheaves on schemes and simplifies the argument as well. \smallskip

Then we apply the same methods to the theory of ``solid almost $\O^+/p$-$\varphi$-modules''. The proof of Poincar\'e Duality for $\ell$-adic sheaves applies essentially verbatim in this context. The main new ingredient is to verify that this $6$-functor formalism satisfies the excision axiom; this is not automatic in this situation. Nevertheless, the approach taken in this paper simplifies the proof of Poincar\'e Duality established in \cite[Cor.\,3.9.25]{Lucas-thesis}. In particular, it avoids any usage of Grothendieck Duality on the special fiber, and any explicit computations related to the ``$p$-adic nearby cycles'' on the formal model of $\bf{D}^1_C$. 
\subsection{$\ell$-adic duality}

The main goal of this section is to give an essentially formal proof of Poincar\'e Duality for \'etale cohomology of schemes and (locally noetherian) adic spaces. The proof is almost uniform in both setups: the only difference is the computation of the cohomology groups of the projective line. \smallskip

In this section, we fix a locally noetherian analytic adic space $S$ (resp. a scheme $S$) and an integer $n$ invertible in $\O_S$. We emphasize that, in the case of adic spaces, we do not make the assumption that $n$ is invertible in $\O_S^+$ until the very end. In what follows, $\cal{C}$ denotes the category of locally finite type adic $S$-spaces (resp. locally finitely presented $S$-schemes). \smallskip

We begin the section by defining the theory of \'etale first Chern classes. Before we start the construction, we advise the reader to take a look at Section~\ref{section:first-chern-classes} since we will follow the notations introduced there. In particular, we recall that in order to speak of (weak) first Chern classes, we first fix an invertible object $\bf{1}_S\langle 1\rangle \in \cal{D}(S)$. \smallskip

\begin{defn} We define the {\it Tate twist} as $\bf{1}_S\langle 1\rangle \coloneqq \mu_n[2]\in \cal{D}(S_\et; \Z/n\Z)$. This object is clearly invertible, so it fits into the assumptions of Section~\ref{section:Chern-classes}. 
\end{defn}

Now we recall that there is a natural Kummer exact sequence
\[
0 \to \mu_n \to \bf{G}_m \xr{f\mapsto f^n} \bf{G}_m \to 0
\]
on $X_\et$ for any $X\in \cal{C}$. This sequence is functorial in $X$, so defines a morphism of $\cal{D}(\Z)$-valued presheaves:
\[
\bf{G}_m[1] \xr{c} \mu_n[2] \colon \cal{C}^{\rm{op}} \to \cal{D}(\Z).
\]
By passing to the derived \'etale sheafifications (see \cite[L.\,3, Cor.\,11]{Dustin-lectures}), we get a morphism of $\cal{D}(\Z)$-valued sheaves
\[
\rm{R}\Gamma_{\et}(-, \bf{G}_m)[1] \xr{c} \rm{R}\Gamma_{\et}(-, \mu_n)[2] \colon \cal{C}^{\rm{op}} \to \cal{D}(\Z).
\]

\begin{defn}\label{defn:etale-first-chern-classes} A {\it theory of \'etale} first Chern classes is the homomorphism of $\cal{D}(\Z)$-valued analytic sheaves 
\[
c_1^\et\colon \rm{R}\Gamma_{\rm{an}}(-, \O^\times)[1] \to \rm{R}\Gamma_{\et}(-, \mu_n)[2]=\rm{R}\Gamma(-; \bf{1}\langle 1\rangle)
\]
obtained as the composition
\[
\rm{R}\Gamma_{\rm{an}}(-, \O^\times)[1] \to \rm{R}\Gamma_{\et}(-, \bf{G}_m)[1] \xr{c} \rm{R}\Gamma_{\et}(-, \mu_n)[2] \colon \cal{C}^{\rm{op}} \to \cal{D}(\Z),
\]
where the first map is the natural morphism from the analytic cohomology of $\O^\times$ to the \'etale cohomology of $\bf{G}_m$. 
\end{defn}

\begin{construction}\label{construction:first-chern-classes-etale-concrete} Let $X$ be an adic $S$-space. Then, after passing to $\rm{H}^0(-)$, Definition~\ref{defn:etale-first-chern-classes} defines a homomorphism
\[
c_1^\et\colon \rm{Pic}(X) \simeq \rm{H}^1_{\rm{an}}(X, \O_X^\times) \to \rm{H}^2(X, \mu_n).
\]
In what follows, we slightly abuse the notation and do not distinguish between these two versions of the homomorphism $c_1^\et$. 
\end{construction}

Now we will later need to know that $c_1^\et$ is a theory of first Chern classes in the sense of Definition~\ref{defn:chern-classes} (if $n$ is invertible in $\O_S^+$). Concretely, this means that we have to show that the natural morphism
\[
c_1^\et(\O(1)) + f^*\colon \ud{\Z/n\Z}_S \oplus \mu_n[2] \to \rm{R}f_* \mu_{n, \bf{P}^1_S}[2]
\]
is an isomorphism for the relative projective line $f\colon \bf{P}^1_S \to S$. We will show this claim with the assumption that $n$ is only invertible in $\O_S$.\smallskip

In the rest of this section, we do the computations entirely in the analytic context. In the algebraic case, the computation is standard (see \cite[Th.\,7.2.9]{Lei-Fu}). \smallskip

We start with the case when $S$ is a ``geometric point''. More explicitly, we fix an algebraically closed non-archimedean field $C$ and assume that $S=\Spa(C, \O_C)$. \smallskip

\begin{lemma}\label{lemma:mu-n-cohomology-curves} Let $X$ be a $1$-dimensional rigid-analytic variety over $S=\Spa(C, \O_C)$ and let $n$ be an integer invertible in $C$. Then 
\begin{enumerate}
    \item the natural morphism $\mu_n(C) \to \rm{H}^0(X; \mu_n)$ is an isomorphism if $X$ is connected;
    \item we have $\rm{H}^i(X, \mu_n)=0$ for $i\geq 3$;
    \item the first Chern class $c^\et_1\colon \rm{Pic}(X)/n \to \rm{H}^2(X, \mu_n)$ is an isomorphism (see Construction~\ref{construction:first-chern-classes-etale-concrete}).
\end{enumerate}
\end{lemma}
\begin{proof}
    {\it Step~$0$. The morphism $\mu_n(C) \to \rm{H}^0(X; \mu_n)$ is an isomorphism if $X$ is connected.} Since $C$ is algebraically closed, we can choose a non-canonical isomorphism $\mu_n \simeq \ud{\Z/n\Z}$. Therefore, it suffices to show that the natural morphism
    \[
    \Z/n\Z \to \rm{H}^0(X, \Z/n\Z)
    \]
    is an isomorphism for a connected $X$. This is a standard result that we leave to the interested reader. \smallskip
    
    To prove the other parts, we consider the morphism of sites $\pi \colon X_{\et} \to X_{\rm{an}}$. \smallskip
    
    {\it Step~$1$. $\rm{R}^i\pi_* \mu_n=0$ for $i\geq 2$.} It suffices to show that the stalk $(\rm{R}^i\pi_* \mu_n)_x=0$ for every $x\in X$. Now \cite[Cor.\,2.4.6]{H3} ensures that, for each integer $i$ and $x\in X$, 
    \[
    \left(\rm{R}^i\pi_* \mu_n\right)_x \simeq \rm{H}^i\left(\Spa\left(K\left(x\right), K\left(x\right)^+\right), \mu_n\right).
    \]
    Thus \cite[Lemma 10.3]{adic-notes} implies that it suffices to prove the vanishing for {\it rank-$1$} points $x\in X$. In this case, 
    \[
    \rm{H}^i(\Spa\left(K\left(x\right), \O_{K\left(x\right)}\right), \mu_n) \simeq \rm{H}^i_{\rm{cont}}(G_{K(x)}, \mu_n).
    \]
    So it suffices to show that $G_{K(x)}$ is of cohomological degree $1$ for any $x\in X$. This follows from \cite[Cor.\,1.8.8 and Lemma 2.8.3]{H3}\footnote{The henselization in \cite[Lemma 2.8.3]{H3} disappears in the rank-$1$ case because $\O_K$ is henselian with respect to its pseudo-uniformizer $\varpi$ and $\m=\mathrm{rad}(\varpi)$ (see \cite[\href{https://stacks.math.columbia.edu/tag/09XJ}{Tag 09XJ}]{stacks-project}).} or one can adapt the proof of \cite[Lemma 5.2.5]{Ber}. \smallskip
    
    {\it Step~$2$. $\rm{R}^1\pi_* \bf{G}_m=0$.} We first note that \cite[(2.2.7)]{H3}  implies that the natural morphism 
    \[
    \rm{Pic}(U)\simeq \rm{H}^1_{\rm{an}}(U, \O_U^\times) \to \rm{H}^1_{\et}(U, \bf{G}_m)
    \]
    is an isomorphism (alternatively, this can be deduced from \cite[Th\,2.5.11]{KedLiu2}). Therefore, the definition of higher pushforwards imply that $\rm{R}^1\pi_*\bf{G}_m$ is the sheafification (in the analytic topology on $X$) of the presheaf
    \[
    U\mapsto \rm{Pic}(U).
    \]
    Since any class $\alpha\in \rm{Pic}(U)$ trivializes {\it analytically} locally on $U$, we conclude the sheafification of this presheaf is zero. \smallskip
    
    {\it Step~$3$. Finish the proof.} The Kummer exact sequence
    \[
    0 \to \mu_n \to \bf{G}_m \xr{\cdot n} \bf{G}_m \to 0
    \]
    implies that we have an exact triangle
    \begin{equation}\label{eqn:pushforward}
    \rm{R}\pi_* \mu_n \to \rm{R}\pi_*\bf{G}_m \xr{\cdot n} \rm{R}\pi_*\bf{G}_m.
    \end{equation}
    Note that $\pi_*\bf{G}_m = \O_X^\times$, so Steps~$(1)$ and $(2)$ imply that (\ref{eqn:pushforward}) stays exact after applying $\tau^{\leq 1}$ to $\rm{R}\pi_*\bf{G}_m$. Thus we get the following exact triangle
    \[
    \rm{R}\pi_* \mu_n \to \O_X^\times \xr{f\mapsto f^n} \O_X^\times.
    \]
    Since $\rm{H}^i(X_{\rm{an}}, \O_X^\times)=0$ for $i\geq 2$ by \cite[Cor.\,1.8.8]{H3} and \cite[\href{https://stacks.math.columbia.edu/tag/0A3G}{Tag 0A3G}]{stacks-project}, we conclude that 
    $\rm{H}^i(X, \mu_n)=0$ for $i\geq 3$ and the natural morphism 
    \[
    \rm{Pic}(X)/n\simeq \rm{H}^1(X_{\rm{an}}, \O_X^\times)/n \to \rm{H}^2(X, \mu_n)
    \]
    is an isomorphism. After unravelling the definitions, one sees that this morphism coincides with $c_1$ from Construction~\ref{construction:first-chern-classes-etale-concrete}. 
\end{proof}

\begin{cor}\label{cor:cohohomology-projective-line} Let $X=\bf{P}^1_C$ be the (analytic) projective line over $\Spa(C, \O_C)$ and let $n$ be an integer invertible in $C$. Then 
\begin{enumerate}
    \item the natural morphism $\mu_n(C) \to \rm{H}^0(\bf{P}^1_C, \mu_n)$ is an isomorphism;
    \item we have $\rm{H}^i(\bf{P}^1_C, \mu_n)=0$ for $i\geq 3$;
    \item the unique homomorphism $c_1\colon \Z/n\Z \to \rm{H}^2(\bf{P}^1_C, \mu_n)$ sending $1$ to $c_1(\O(1))$ is an isomorphism.
\end{enumerate}
\end{cor}
\begin{proof}
    This follows formally from Lemma~\ref{lemma:mu-n-cohomology-curves} and the fact that the morphism
    \[
    \Z \to \rm{Pic}(\bf{P}^1_C),
    \]
    sending $n$ to $\O(n)$, is an isomorphism. The latter fact follows from \cite[Cor.\,8.3]{adic-notes}.
\end{proof}

Now we go back to the case of a general locally noetherian analytic adic base $S$. Then we consider the relative (analytic) projective line $f\colon \bf{P}^1_S \to S$. This comes with the ``universal'' line bundle $\O(1)$ (see \cite[Rmk.\,7.10]{adic-notes} for the construction in the analytic setup). The first Chern class $c_1(\O(1))$ defines a morphism
\[
c_1(\O(1)) \colon \ud{\Z/n\Z}_{\bf{P}^1_S} \to \mu_n[2].
\]
in the (triangulated) derived category $D(\bf{P}^1_S; \Z/n\Z)$. Due to the $(f^*, \rm{R}f_*)$-adjunction, $c_1(\O(1))$ defines a morphism
\[
c_1^\et(\O(1)) \colon \ud{\Z/n\Z}_S \to \rm{R}f_* \mu_{n, \bf{P}^1_S}[2].
\]

\begin{prop}\label{prop:projective-bundle-formula-etale} Let $f\colon \bf{P}^1_S \to S$ be the relative (analytic) projective line over $S$ and let $n$ be an integer invertible in $S$. Then the natural morphism
\[
c_1^\et(\O(1)) + f^*\colon \ud{\Z/n\Z}_S \oplus \mu_n[2] \to \rm{R}f_* \mu_{n, \bf{P}^1_S}[2]
\]
is an isomorphism\footnote{The notation ``$f^*$'' means the natural morphism $\mu_n[2] \to \rm{R}f_* \mu_{n, \bf{P}^1_S}[2]$ coming as the unit of the $(f^*, \rm{R}f_*)$-adjunction.}.
\end{prop}
\begin{proof}
    It suffices to show that the morphism $c_1^\et(\O(1)) + f^*$ is an isomorphism on stalks. First, \cite[Lemma 10.4]{adic-notes} ensures that $\rm{R}f_*$ preserves overconvergent sheaves, so it is sufficient on stalks over {\it rank-$1$} points. Now we note that the formation of first Chern classes commute with arbitrary base change (similarly to Remark~\ref{rmk:properties-first-chern-theory}(\ref{rmk:properties-first-chern-theory-base-change})), \cite[Prop.\,2.6.1]{H3} ensures that it suffices to prove the claim under the additional assumption that $S=\Spa(C, \O_C)$ for an algebraically closed, non-archimedean field $C$. Then the result follows directly from Corollary~\ref{cor:cohohomology-projective-line}.
\end{proof}

\begin{thm}\label{thm:ell-adic-poincare-package} Let $S$ be a locally noetherian analaytic adic space, $n$ an integer invertible in $\O_S^+$, and $\cal{D}_\et(-; \Z/n\Z) \colon \rm{Corr}(\cal{C}) \to \Cat_\infty$ be the $6$-functor formalism formalism constructed in \cite[Th.\,9.4 and Rmk.\,9.5]{adic-notes}. Then
\begin{enumerate}
    \item $\cal{D}_\et(-; \Z/n\Z)$ satisfies the excision axiom (see Definition~\ref{defn:excision-axiom});
    \item Definition~\ref{defn:etale-first-chern-classes} defines a theory of first Chern classes on $\cal{D}_\et(-; \Z/n\Z)$ (see Definition~\ref{defn:geometric-chern-classes-without-length}) with $\bf{1}_S\langle 1\rangle = \mu_n[2]$.
\end{enumerate}
\end{thm}
\begin{proof}
    It is essentially obvious that $\cal{D}_\et(-; \Z/n\Z)$ satisfies the excision axiom. More precisely, it suffices to show that, for any locally finite type adic $S$-space $X$, a complex $\F\in \cal{D}_\et(X; \Z/n\Z)$, and a Zariski-closed immersion $i\colon Z\to X$, the triangle
    \[
    j_!j^* \F \to \F \to i_*i^*\F
    \]
    is exact, where $j\colon U \to X$ is the open complement of $Z$. This is clear by arguing on stalks. The fact that $c_1$ is a theory of first Chern classes follows directly from Proposition~\ref{prop:projective-bundle-formula-etale}. 
\end{proof}

Before we state the general version of Poincar\'e Duality, we recall that the Tate twist $\ud{\Z/n\Z}(m)$ is by definition the \'etale sheaf $\mu_n^{\otimes m}$ (with the obvious meaning if $m$ is negative). Likewise, for a sheaf $\F\in \cal{D}(X_\et; \Z/n\Z)$, we denote its Tate twist $\F\otimes \ud{\Z/n\Z}(m)$ simply by $\F(m)$.

\begin{thm}\label{thm:main-PD-ell-adic} Let $Y$ be a locally noetherian analytic adic space, let $f\colon X\to Y$ be a smooth morphism of pure dimension $d$, and let $n$ be an integer invertible in $\O_Y^+$. Then the functor
\[
\rm{R}f_!\colon \cal{D}(X_\et; \Z/n\Z) \to \cal{D}(Y_\et; \Z/n\Z)
\]
admits a right adjoint given by the formula
\[
f^*(d)[2d] \colon \cal{D}(Y_\et; \Z/n\Z) \to \cal{D}(X_\et; \Z/n\Z).
\]
\end{thm}
\begin{proof}
    Put $S=Y$ and consider the \'etale $6$-functor formalism $\cal{D}_\et(-; \Z/n\Z)$ that associates to $X$ the $\infty$-derived category $\cal{D}(X_\et; \Z/n\Z)$ (see \cite[Th.\,9.4 and Rmk.\,9.5]{adic-notes}). Then Theorem~\ref{thm:ell-adic-poincare-package} implies that $\cal{D}_\et$ satisfies the excision axiom and admits a theory of first Chern classes with $\bf{1}_S\langle 1\rangle =\mu_n[2]$. Thus, the result follows from Theorem~\ref{thm:main-thm}.
\end{proof}

\begin{rmk}\label{rmk:scheme-duality} The proof of Theorem~\ref{thm:main-PD-ell-adic} works in essentially the same way for the $6$-functor formalism of \'etale $\Z/n\Z$-sheaves on schemes (see \cite[Rmk.\,9.6]{adic-notes} for the construction of \'etale $6$-functor formalism). In particular, this reproves the classical Poincar\'e Duality in the theory of \'etale cohomology of schemes.
\end{rmk}

\begin{rmk} Note that the only place, where we used that $n$ is invertible in $\O_S^+$ (as opposed to being invertible in $\O_S$) is to make sure that the categories $\cal{D}_\et(-; \Z/n\Z)$ can be arranged into a $6$-functor formalism. If $n$ is not invertible in $\O_S^+$, the problem is that the proper base change formula does not hold in general. In the next section, we work around this issue by using another $6$-functor formalism closely related to the $p$-adic cohomology of $p$-adic rigid-analytic spaces.
\end{rmk}

\subsection{$p$-adic duality}

The goal of this section is to give a new proof of Poincar\'e Duality for $\O^{+}/p$-$\varphi$-modules''. \smallskip

In what follows, we fix a locally noetherian analytic adic space $S$ with a morphism $S \to \Spa(\Q_p, \Z_p)$, and $\cal{C}$ the category of locally finite type adic $S$-spaces. \smallskip

Now we briefly sketch the construction of the $6$-functor formalism of $\O^+/p$-($\varphi$-)modules developed in \cite{Lucas-thesis}. We will not discuss the full construction of this formalism here; instead we only sketch the part that are important for the discussion of this section, and refer to \cite{Lucas-thesis} for the thorough construction of this $6$-functor formalism. \smallskip 

To begin with, we recall that \cite[Th.\,3.6.12 and Prop.\,3.9.13]{Lucas-thesis} define\footnote{See also \cite[Prop.\,3.5.14]{Lucas-thesis} to conclude that any locally finite type morphism of analytic adic spaces is bdcs in the sense of \cite[Def.\,3.6.9]{Lucas-thesis}.} two (closely related) $6$-functor formalisms
\[
\cal{D}_{\Box}^{\rm{a}}(-; \O^+/p) \colon \Corr(\cal{C}) \to \Cat_\infty,
\]
and 
\[
\cal{D}_{\Box}^{\rm{a}}(-; \O^+/p)^{\varphi} \colon \Corr(\cal{C}) \to \Cat_\infty.
\]
These two $6$-functor formalisms are defined in a significantly more general setup, that generality will not play a huge role in our discussion beyond the point that we can evaluate $\cal{D}_{\Box}^{\rm{a}}(-; \O^+/p)$ on strictly totally disconnected perfectoids over $S$ (which are essentially never locally finite type over $S$). \smallskip

We briefly discuss the construction of the category $\cal{D}_{\Box}^{\rm{a}}(X; \O^+/p)$ in \cite{Lucas-thesis}. First, for a (strictly) totally disconnected perfectoid space with a map $\Spa(R, R^+)\to S$, one puts 
\[
\cal{D}_{\Box}^{\rm{a}}(\Spa(R, R^+); \O^+/p) = \cal{D}_{\Box}^{\rm{a}}(R^+/p)
\]
the almost category of solid $R^+/p$-modules (see \cite[Def.\,3.1.2]{Lucas-thesis}). Then one shows that this assignment satisfies (hyper-)descent in the $\rm{v}$-topology (see \cite[Th.\,3.1.27 and Def.\,3.1.3]{Lucas-thesis}) on (strictly) totally disconnected perfectoid spaces over $S$. After that, Mann formally extends $\cal{D}_{\Box}^{\rm{a}}(X; \O^+/p)$ to all adic $S$-spaces by descent. This category comes equipped with the usual $4$ functors: $f_*$, $f^*$, $\ud{\Hom}$, and $\otimes$. The question of defining the shriek functors is quite subtle and we refer to \cite[\textsection 3.6]{Lucas-thesis} for their construction. \smallskip

The $\varphi$-version of $\cal{D}_{\Box}^{\rm{a}}(X; \O^+/p)$ is defined as the equalizer (in the $\infty$-categorical sense) 
\[
\cal{D}^{\rm{a}}_{\Box}(X; \O^+/p)^{\varphi} \coloneqq \rm{eq}\left(\cal{D}_{\Box}^{\rm{a}}(-; \O^+/p)  \xr{\varphi - \rm{id}} \cal{D}_{\Box}^{\rm{a}}(-; \O^+/p)\right). 
\]
Then \cite[Prop.\,3.9.13]{Lucas-thesis} extends the $6$-functors to $\cal{D}^{\rm{a}}_{\Box}(X; \O^+/p)^{\varphi}$. \smallskip

Our first goal is to show that both of these $6$-functor formalisms satisfy the excision axiom (see Definition~\ref{defn:excision-axiom}). This will allow us to apply Theorem~\ref{thm:main-thm} to this situation and reduce the question of proving Poincar\'e Duality to the question of constructing a theory of first Chern classes and computing the cohomology groups of the projective line $\bf{P}^1_C$. \smallskip

One useful tool in proving the excision axiom will be the (sub)category of discrete objects $\cal{D}^{\rm{a}}_\Box(X; \O_X^+/p)_{\omega} \subset \cal{D}^{\rm{a}}_\Box(X; \O_X^+/p)$ introduced in \cite[Def.\,3.2.17]{Lucas-thesis}. If $X$ admits a map\footnote{This condition ensures that $X\in X_{\rm{v}}^{\Lambda}$ in the sense of \cite[Def.\,3.2.5]{Lucas-thesis}.} to an affinoid perfectoid space $\Spa(R, R^+)$ \cite[Prop.\,3.3.16]{Lucas-thesis} justifies the name and shows that there is a functorial equivalence
\[
\cal{D}^{\rm{a}}_\Box(X; \O_X^+/p)_{\omega} \simeq \rm{Shv}^{\wdh{ }}(X_\et; \O_X^{+, \rm{a}}/p)^{\rm{oc}}
\]
between discrete objects in $\cal{D}^{\rm{a}}_\Box(X; \O_X^{+, a}/p)$ and overconvergent objects in the left-completed $\infty$-derived category of \'etale sheaves of almost $\O_X^+/p$-modules (see \cite[Prop.\,3.3.16]{Lucas-thesis}).

\begin{lemma}\label{excision:perfectoid} Let $X=\Spa(R, R^+)$ be a strictly totally disconnected perfectoid space over $S$, let $i\colon Z \hookrightarrow X$ be a Zariski-closed affinoid perfectoid subspace (in the sense of \cite[Def.\, 5.7]{Scholze-diamond}), and let $j\colon U \hookrightarrow X$ be its open complement. Then 
\[
j_! \O_U^{+, \rm{a}}/p \to \O_X^{+, \rm{a}}/p \to i_* \O_Z^{+, \rm{a}}/p
\]
is a fiber sequence in $\cal{D}^{\rm{a}}_{\Box}(X; \O^{+}/p)$.
\end{lemma}
\begin{proof}
    {\it Step~$1$. $j_! \O_U^{+, \rm{a}}/p$ is discrete.} We first consider the morphism $\pi\colon |X| \to \pi_0(X)$ from \cite[Lemma 7.3]{Scholze-diamond}. Since $Z$ is Zariski-closed, it is both closed under generalizations and specializations. Thus the same holds for $U$, so the natural morphism $U \to \pi^{-1}(\pi(U))$ is an isomorphism. Since $\pi$ is a quotient morphism, we conclude that $U'\coloneqq \pi(U)$ must be open in $\pi_0(X)$. \smallskip
    
    Now recall that $\pi_0(X)$ is a profinite set. So clopen subsets form a base of topology on $\pi_0(X)$. Therefore $U' = \cup_{i\in I} U'_i$ is a filtered union of clopen subset $U'_i$ (in particular, they are quasi-compact). Thus we conclude that $U = \cup_{i\in I} \pi^{-1}(U'_i)$ is a filtered union of clopen subspaces of $X$. We denote the pre-image $U'_i$ by $j_i \colon U_i \to X$. Then, by construction (see \cite[Lemma 3.6.2]{Lucas-thesis}), we have 
    \[
    j_! \O_U^{+, \rm{a}}/p \simeq \colim j_{i, !} \O_{U_i}^{+, \rm{a}}/p.
    \]
    Since each $j_i$ is clopen, we conclude that $j_{i, !}=j_{i, *}$. Thus each $j_{i, !} \O_{U_i}^{+,\rm{a}}/p=j_{i, *}\O_{U_i}^{+, \rm{a}}/p$ is discrete by \cite[Lemma 3.3.10(ii)]{Lucas-thesis}. So the colimit is also discrete by \cite[Lemma 3.2.19]{Lucas-thesis}. \smallskip
    
    {\it Step~$2$. Reduce to the case $X=\Spa(C, C^+)$.} Now we note that $i_*\O_Z^{+, \rm{a}}/p$ is discrete by \cite[Lemma 3.3.10]{Lucas-thesis}. So we can check that the morphism 
    \begin{equation}\label{eqn:iso?}
        j_!\O_U^{+, \rm{a}}/p \to \rm{fib}\left(\O_X^{+,\rm{a}}/p \to i_*\O_Z^{+, \rm{a}}/p\right)
    \end{equation}
    is an isomorphism in $\cal{D}^{\rm{a}}_\Box(X; \O_X^{+}/p)_{\omega} \simeq \rm{Shv}^{\wdh{ }}(X_\et; \O_X^{+,\rm{a}}/p)^{\rm{oc}}$. However, a property of a map being an isomorphism in $\rm{Shv}^{\wdh{ }}(X_\et; \O_X^{+, a}/p)^{\rm{oc}}$ can be checked on stalks. Therefore, it suffices to prove the claim after a pullback\footnote{Here, we implicitly use base change for both $j_!$ and $i_*$} along each morphism 
    \[
    \Spa(C, C^+) \to X,
    \]
    where $C$ is an algebraically closed non-archimedean field, and $C^+\subset C$ is an open bounded valuation ring. But this is essentially obvious: note that $Z\times_X \Spa(C, C^+)$ is a Zariski-closed subspace of $\Spa(C, C^+)$, so it is either empty or equal to $\Spa(C, C^+)$. In both cases, Morphism~(\ref{eqn:iso?}) is tautologically an isomorphism. 
\end{proof}

\begin{lemma}\label{lemma:excision-phi-modules} The $6$-functor formalisms $\cal{D}_{\Box}^{\rm{a}}(-; \O^+/p)$ and $\cal{D}_{\Box}^{\rm{a}}(-; \O^+/p)^{\varphi}$ satisfy the excision axiom.
\end{lemma}
\begin{proof}
    We fix a locally finite type adic $S$-space $X$, a Zariski-closed immersion $Z\xhookrightarrow{i} X$, and its open complement $U\xhookrightarrow{j} X$. We wish to show that, for any $\F \in \cal{D}_{\Box}^{\rm{a}}(X; \O^+/p)$ (resp. $\F \in \cal{D}_{\Box}^{\rm{a}}(X; \O^+/p)^{\varphi}$), the natural morphism
    \[
    j_!j^*\F \to \rm{fib}\left(\F \to i_*i^*\F\right)
    \]
    is an isomorphism. Since the forgetful functor 
    \[
    \cal{D}_{\Box}^{\rm{a}}(X; \O^+/p)^{\varphi} \to \cal{D}_{\Box}^{\rm{a}}(X; \O^+/p)
    \]
    commutes with limits, all $6$-functors, and is conservative (see \cite[Lem.\,3.9.12]{Lucas-thesis}), it is sufficient to prove that $\cal{D}_{\Box}^{\rm{a}}(-; \O^+/p)$ satisfies excision. \smallskip
    
    For this, we note that the projection formulas for $i_*$ and $j_!$ imply that it suffices to show that the natural morphism
    \[
    j_!\O_U^{+,\rm{a}}/p \to \rm{fib}\left(\O_X^{+, \rm{a}}/p \to i_*\O_Z^{+,\rm{a}}/p\right)
    \]
    is an isomorphism. By the $\rm{v}$-descent and shriek base-change, it can be checked on the basis of strictly totally disconnected perfectoid spaces. Now \cite[Lemma 5.2]{adic-notes} ensures that, for any perfectoid space $T$ with a morphism $T \to X$, the pullback $(Z\times_X T)^\diam \to T^\diam$ is represented by a Zariski-closed immersion. Therefore, the result follows from Lemma~\ref{excision:perfectoid} and the fact that, for any adic space $Y$ over $\Spa(\Q_p, \Z_p)$, the construction of  $\cal{D}^{\rm{a}}_{\Box}(Y; \O^+/p)^{\varphi}$ depends only on the associated diamond of $Y$.
\end{proof}

Now we discuss the computation of the cohomology groups of the projective line, and the construction of first Chern classes. An important tool to deal with these questions is the Riemann--Hilbert functor from \cite[\textsection 3.9]{Lucas-thesis}. We follow the notation of \cite{Lucas-thesis}, and denote by $\cal{D}_\et(X; \bf{F}_p)$ the {\it left-completed} $\infty$-derived category\footnote{It may be more appropriate to denote this category by $\widehat{\cal{D}}_\et(X; \bf{F}_p)$ or $\rm{Shv}^{\wdh{}}(X_\et; \bf{F}_p)$, but we prefer to stick to the notation used in \cite{Lucas-thesis}. The reason to use this notation is that the left completed version naturally arises as the ``derived'' category of \'etale $\bf{F}_p$-sheaves on the associated diamond $X^\diam$.} of \'etale sheaves of $\bf{F}_p$-modules on $X$. We also denote by $\cal{D}_\et(X; \bf{F}_p)^{\rm{oc}} \subset \cal{D}_\et(X; \bf{F}_p)$ the full $\infty$-subcategory spanned by overconvergent sheaves (see \cite[Def.\,3.9.17]{Lucas-thesis}). Then \cite[Def.\,3.9.21]{Lucas-thesis} defines the Riemann--Hilbert functor
\[
-\otimes \O_X^{+, \rm{a}}/p \colon \cal{D}_\et(X; \bf{F}_p)^{\rm{oc}} \to \cal{D}^{\rm{a}}_\Box(X; \O_X^+/p)^{\varphi}.
\]
If $X$ admits a map to an affinoid perfectoid field $\Spa(R, R^+)$, then (essentially by construction) the following diagram
\begin{equation}\label{eqn:RH}
\begin{tikzcd}
    \cal{D}_\et(X; \bf{F}_p)^{\rm{oc}} \arrow{d}{-\otimes \O_X^{+, \rm{a}}/p} \arrow{r}{-\otimes \O_X^{+, \rm{a}}/p} & \cal{D}^{\rm{a}}_\Box(X; \O_X^+/p)^{\varphi} \arrow{d}{\rm{can}} \\
    \rm{Shv}^{\wdh{}}(X_\et; \O_X^{+, \rm{a}}/p)^{\rm{oc}} \arrow{r} & \cal{D}^{\rm{a}}_\Box(X; \O_X^+/p)
\end{tikzcd}
\end{equation}
commutes up to a homotopy, where the left vertical functor is (the left completion) of the naive (derived) tensor product functor, and the bottom horizontal functor is the canonical identification of $\rm{Shv}^{\wdh{}}_\et(X; \O_X^{+, \rm{a}}/p)^{\rm{oc}}$ with the subcategory of discrete objects in $\cal{D}^{\rm{a}}_\Box(X; \O_X^+/p)$. \smallskip

\begin{defn} The {\it $p$-adic Tate twist} $\O_X^{+, \rm{a}}/p(i) \in \cal{D}^{\rm{a}}_\Box(X; \O_X^+/p)^{\varphi}$ (resp. $\O_X^{+, \rm{a}}/p(i) \in \cal{D}^{\rm{a}}_\Box(X; \O_X^+/p)$) is the image of the Tate twist $\ud{\bf{F}}_p(i)$ under the Riemann--Hilbert functor, i.e.,
\[
\O_X^{+, \rm{a}}/p(i) \simeq \ud{\bf{F}}_p(i)\otimes \O_X^{+, \rm{a}}/p.
\]
\end{defn}

\begin{warning} In the next lemma, we follow the terminology of \cite{Lucas-thesis} and do not write $\rm{R}$ for the derived functors on the category of $\bf{F}_p$-sheaves.
\end{warning}

\begin{lemma}\label{lemma:primitive-comparison} Let $f\colon X \to Y$ be a proper morphism in $\cal{C}$, and $k$ an integer. Then the natural morphism
\[
\bigl(f_{\et, *} \ud{\bf{F}_p}(k)\bigr) \otimes \O_{Y}^{+, \rm{a}}/p \to f_* \bigl(\O_X^{+, \rm{a}}/p(k)\bigr)
\]
is an isomorphism in $\cal{D}^{\rm{a}}_\Box(Y; \O_Y^+/p)^{\varphi}$.
\end{lemma}
\begin{proof}
    The claim is $\rm{v}$-local on the base, so we can assume that $Y$ (and, therefore, $X$) admits a morphism to an affinoid perfectoid space $\Spa(R, R^+)$. Then we wish to leverage Diagram~(\ref{eqn:RH}) to reduce the question to the classical Primitive Comparison Theorem. \smallskip
    
    More precisely, we first note that the forgetful functor $\cal{D}^{\rm{a}}_\Box(X; \O_X^+/p)^{\varphi} \to \cal{D}^{\rm{a}}_\Box(X; \O_X^+/p)$ is conservative by \cite[Lem.\,3.9.12(i)]{Lucas-thesis}. Thus, it suffices to show that the corresponding morphism
    \[
    f_{\et, *} \ud{\bf{F}_p}(k) \otimes \O_{Y}^{+, \rm{a}}/p \to f_* \O_X^{+, \rm{a}}/p(k)
    \]
    is an isomorphism in $\cal{D}^{\rm{a}}_\Box(Y; \O_Y^+/p)$. Now we note that \cite[Prop.\,3.3.16 and Lemmas 3.3.10(ii), 3.3.15(iii)]{Lucas-thesis} imply that the diagram
\begin{equation}\label{eqn:RH-2}
\begin{tikzcd}
     \rm{Shv}^{\wdh{}}(X_\et; \O_X^{+, \rm{a}}/p)^{\rm{oc}} \arrow{d}{f_{\et, *}} \arrow{r}{} & \cal{D}^{\rm{a}}_\Box(X; \O_X^+/p) \arrow{d}{f_{*}} \\
    \rm{Shv}^{\wdh{}}(Y_\et; \O_X^{+, \rm{a}}/p)^{\rm{oc}} \arrow{r} & \cal{D}^{\rm{a}}_\Box(Y; \O_X^+/p)
\end{tikzcd}
\end{equation}
    commutes up to a homotopy. Therefore, Diagram~(\ref{eqn:RH}) ensures that it suffices to show that the natural morphism
    \[
    \left(f_{\et, *} \ud{\bf{F}_p}(k)\right) \otimes \O_{Y}^{+, \rm{a}}/p \to f_{\et, *} \O_X^{+, \rm{a}}/p(k)
    \]
    is an isomorphism in $\rm{Shv}^{\wdh{}}(Y_\et; \O_Y^{+, \rm{a}}/p)$. More explicitly, we reduced the question to showing that, for each $k$ and $d$, the natural morphism
    \[
    \rm{R}^df_{\et, *} \ud{\bf{F}}_p(k) \otimes_{\bf{F}_p} \O_Y^+/p \to \rm{R}^df_{\et, *} \O_X^+/p(k)
    \]
    is an {\it almost} isomorphism of \'etale $\O_Y^+/p$-module. This follows from the standard Primitive Comparison Theorem from the $p$-adic Hodge theory, see \cite[Cor.\,5.11]{Scholze-p-adic} or \cite[Lemma 6.3.7]{almost-coherent}.
\end{proof}

Now we are ready to define first Chern classes on $\cal{D}^{\rm{a}}_\Box(-; \O^+/p)^\varphi$ and $\cal{D}^{\rm{a}}_\Box(-; \O^+/p)$. For this, we note that the Riemann--Hilbert functor $\cal{D}_\et(X; \bf{F}_p)^{\rm{ov}} \to \cal{D}^{\rm{a}}_\Box(X; \O_X^+/p)^\varphi$ sends the constant sheaf $\ud{\bf{F}_p}$ to the unit object $\O_X^+/p$, and so it defines a functorial in $X$ morphism:
\[
\rm{R}\Gamma_{\et}(X, \mu_p) \to \rm{R}\Gamma(X, \O_X^{+, \rm{a}}/p(1)) \coloneqq \Hom_{\cal{D}^{\rm{a}}_\Box(X; \O_X^+/p)^\varphi}(\O_X^{+, \rm{a}}/p, \O_X^{+, \rm{a}}/p(1)). 
\]

\begin{defn} We define the {\it Tate twist} as $\bf{1}_S\langle 1\rangle \coloneqq \O_S^{+,\rm{a}}/p(1)[2]\in \cal{D}_\Box^{\rm{a}}(S; \O_S^+/p)^{\varphi}$. This object is invertible (since $-\otimes \O_{S}^{+, \rm{a}}/p$ is symmetric monoidal), so it fits into the assumptions of Section~\ref{section:Chern-classes}. 
\end{defn}

\begin{defn}\label{defn:first-chern-classes-mod-p} A theory of first Chern classes on the $6$-functor formalism $\cal{D}^{\rm{a}}_\Box(-; \O^+/p)^{\varphi}$ is the morphism of $\rm{Sp}$-valued presheaves 
\[
c_1^{\varphi}\colon \rm{R}\Gamma_{\rm{an}}(-, \O^\times)[1] \to \rm{R}\Gamma(-, \O^{+, \rm{a}}/p)[2]=\rm{R}\Gamma(-, \bf{1}\langle 1\rangle)
\]
obtained as the composition
\[
\rm{R}\Gamma_{\rm{an}}(-; \O^\times)[1] \xr{c_1^\et} \rm{R}\Gamma_{\et}(-; \mu_p)[2] \to \rm{R}\Gamma(-; \O^+/p(1))[2],
\]
where the first morphism comes from Definition~\ref{defn:etale-first-chern-classes}.
\end{defn}

\begin{thm}\label{thm:p-adic-poincare-package} Let $S$ be a locally noetherian analaytic adic space over $\Spa(\Q_p, \Z_p)$. Then
\begin{enumerate}
    \item $\cal{D}^{\rm{a}}_\Box(-; \O^+/p)^{\varphi}$ satisfies the excision axiom;
    \item $c_1^{\varphi}$ is a theory of first Chern classes on $\cal{D}^{\rm{a}}_\Box(-; \O^+/p)^{\varphi}$.
\end{enumerate}
\end{thm}
\begin{proof}
    Lemma~\ref{lemma:excision-phi-modules} ensures that $\cal{D}^{\rm{a}}_\Box(-; \O^+/p)^{\varphi}$ satisfies the excision sequence. To show that $c_1$ is a theory of first Chern classes, we have to show that the natural morphism
    \[
    c_1^{\varphi}(\O(1)) + f^* \colon \O_S^{+, \rm{a}}/p \oplus \O_S^{+, \rm{a}}/p(1)[2] \to f_* \left( \O_{\bf{P}^1_S}^{+, \rm{a}}/p(1)[2]\right) 
    \]
    is an isomorphism. For this, we use the commutative diagram
    \[
    \begin{tikzcd}[column sep = 9em]
        \left(\ud{\bf{F}}_p \oplus \mu_p[2]\right) \otimes \O_S^{+, \rm{a}}/p  \arrow{d}\arrow{r}{\left(c_1^\et(\O(1))+ f^*_\et\right) \otimes \O^+/p} & f_{\et, *}\left(\mu_p[2]\right) \otimes \O_S^{+, \rm{a}}/p \arrow{d}\\
        \O_S^{+, \rm{a}}/p \oplus \O_S^{+, \rm{a}}/p(1)[2] \arrow{r}{c^{\varphi}_1(\O(1))+f^*} & f_*\left(\O^+/p(1)[2]\right).
    \end{tikzcd}
    \]
    The left vertical arrow is an isomorphism by definition, the right vertical arrow is an isomorphism by Lemma~\ref{lemma:primitive-comparison}, and the top horizontal map is an isomorphism by Proposition~\ref{prop:projective-bundle-formula-etale}. Therefore, the bottom horizontal arrow must be an isomorphism as well finishing the proof.  
\end{proof}

\begin{thm}\label{thm:main-PD-p-adic} Let $Y$ be a locally noetherian analytic adic space over $\Spa(\Q_p, \Z_p)$ and let $f\colon X\to Y$ be a smooth morphism of pure dimension $d$. Then the functor 
\[
f_!\colon \cal{D}^{\rm{a}}_\Box(X; \O_X^+/p)^{\varphi} \to \cal{D}^{\rm{a}}_\Box(Y; \O_Y^+/p)^{\varphi}
\]
admits a right adjoint given by the formula
\[
f^* \otimes \O_X^{+, \rm{a}}/p(d)[2d] \colon \cal{D}^{\rm{a}}_\Box(Y; \O_Y^+/p)^{\varphi} \to \cal{D}^{\rm{a}}_\Box(X; \O_X^+/p)^{\varphi}.
\]
\end{thm}
\begin{proof}
    This is a direct consequence of Theorem~\ref{thm:p-adic-poincare-package} and Theorem~\ref{thm:main-thm}.
\end{proof}

\begin{rmk} Essentially the same proof applies to the $6$-functor formalism $\cal{D}^{\rm{a}}_\Box(-; \O^+/p)$.
\end{rmk}





\bibliography{biblio}

\end{document}